\documentclass[10pt]{amsart}
      \usepackage[mathscr]{eucal}
      \usepackage{amsmath,amsfonts}
      \parskip\smallskipamount
          \newtheorem{theorem}{Theorem}[section]
      
      \newtheorem{proposition}[theorem]{Proposition}
      \newtheorem{corollary}[theorem]{Corollary}
      \newtheorem{lemma}[theorem]{Lemma}

      \newtheorem{remark}[theorem]{Remark}
      \makeatletter
      \@addtoreset{equation}{section}
      \makeatother

      \newcommand{\BB}{{\mathbb B}}
      \newcommand{\CC}{{\mathbb C}}
      \newcommand{\NN}{{\mathbb N}}

      \newcommand{\DD}{{\mathbb D}}
      \newcommand{\RR}{{\mathbb R}}
      \newcommand{\FF}{{\mathbb F}}
      \newcommand{\TT}{{\mathbb T}}

      \newcommand{\cA}{{\mathcal A}}
      \newcommand{\cB}{{\mathcal B}}
      \newcommand{\cC}{{\mathcal C}}
      \newcommand{\cD}{{\mathcal D}}
      \newcommand{\cE}{{\mathcal E}}

      \newcommand{\cH}{{\mathcal H}}
      
      \newcommand{\cK}{{\mathcal K}}

      \newcommand{\cP}{{\mathcal P}}
      \newcommand{\cR}{{\mathcal R}}
      \newcommand{\cS}{{\mathcal S}}

      \newcommand{\cY}{{\mathcal Y}}
      \newcommand{\cX}{{\mathcal X}}

      \newdimen\expt
      \def\boxit#1{\setbox0\hbox{$\displaystyle{#1}$}
            \hbox{\lower.4\expt
       \hbox{\lower3\expt\hbox{\lower\dp0
            \hbox{\vbox{\hrule height.4\expt
       \hbox{\vrule width.4\expt\hskip3\expt
            \vbox{\vskip3\expt\box0\vskip2\expt}%
       \hskip3\expt\vrule width.4\expt}\hrule height.4\expt}}}}}}
      
\begin{document}
       \pagestyle{myheadings}
      \markboth{ Gelu Popescu}{ Hyperbolic  geometry
       on  noncommutative balls   }


      \title [   Hyperbolic  geometry
       on  noncommutative balls ]
{       Hyperbolic  geometry
       on  noncommutative balls
       }
        \author{Gelu Popescu}

\date{October 22, 2009 (revised version)}
      \thanks{Research supported in part by an NSF grant}
      \subjclass[2000]{Primary: 46L52;  46T25;  47A20; Secondary:  32Q45; 47A56}
      \keywords{Noncommutative  hyperbolic geometry; Noncommutative function
      theory;
      Harnack part;  Hyperbolic distance; Carath\' eodory distance; Free holomorphic function;
      Free pluriharmonic function;
         Fock
      space; Creation operator;
       Joint operator radius; Joint numerical radius; Joint spectral radius;
        von Neumann inequality; Schwarz lemma.
        }

      \address{Department of Mathematics, The University of Texas
      at San Antonio,  San Antonio, TX 78249, USA}
      \email{\tt gelu.popescu@utsa.edu}

\begin{abstract} In this paper, we study the noncommutative balls
$$
\cC_\rho:=\{(X_1,\ldots, X_n)\in  B(\cH)^n: \
\omega_\rho(X_1,\ldots, X_n)\leq 1\},\qquad \rho\in (0,\infty],
$$
where $\omega_\rho$ is the joint operator radius for $n$-tuples of
bounded linear operators on a Hilbert space. In particular,
$\omega_1$ is the operator norm, $\omega_2$ is the   joint numerical
radius, and $\omega_\infty$ is the joint spectral radius.

We introduce a  Harnack type  equivalence relation on $\cC_\rho$,
$\rho>0$,  and use it to define a hyperbolic distance $\delta_\rho$
on the Harnack parts (equivalence classes) of $\cC_\rho$. We prove
that the open ball
$$
[\cC_\rho]_{<1}:=\{(X_1,\ldots, X_n)\in  B(\cH)^n: \
\omega_\rho(X_1,\ldots, X_n)<1\},\qquad \rho>0,
$$
is the Harnack part containing  $0$  and
 obtain a concrete formula for the hyperbolic distance, in terms of the
reconstruction operator associated with the right creation operators
on the full Fock space with $n$ generators. Moreover, we show that
the $\delta_\rho$-topology and the usual operator norm topology
coincide on  $[\cC_\rho]_{<1}$.  While the open ball
$[\cC_\rho]_{<1}$  is not a complete metric space with respect to
the operator norm topology,  we prove that it is a complete metric
space with respect to the hyperbolic metric $\delta_\rho$. In the
particular case when $\rho=1$ and $\cH=\CC$, the hyperbolic metric
$\delta_\rho$ coincides with the Poincar\' e-Bergman distance on the
open unit ball of $\CC^n$.

We introduce a Carath\' eodory type metric on $[\cC_\infty]_{<1} $,
the set of all $n$-tuples of operators with joint spectral radius
strictly less then $1$,  by setting
$$
d_K(A,B)=\sup_p \|\Re  p(A)-\Re p(B)\|,\qquad A,B\in
[\cC_\infty]_{<1},
$$
where the supremum is taken over all noncommutative polynomials with
matrix-valued coefficients $p\in \CC[X_1,\ldots, X_n]\otimes M_{m}$,
$m\in \NN$, with $\Re p(0)=I$ and $\Re p(X)\geq 0$ for all $X\in
\cC_1$. We obtain a concrete formula for $d_K$
  in terms of the free pluriharmonic kernel on  the noncommutative ball $[\cC_\infty]_{<1}$.
We also  prove that the metric $d_K$ is complete on
$[\cC_\infty]_{<1}$ and its topology coincides with  the operator
norm topology.

We   provide mapping theorems,   von Neumann inequalities, and
Schwarz type lemmas  for free holomorphic functions on
noncommutative balls, with respect to the hyperbolic metric
$\delta_\rho$, the Carath\' eodory  metric $d_K$,   and the joint
operator radius $\omega_\rho$, $\rho\in (0,\infty]$.
\end{abstract}

      \maketitle

\section*{Contents}
{\it

\quad Introduction

\begin{enumerate}
\item[1.]  The noncommutative ball $\cC_\rho$ and  a free pluriharmonic  functional
calculus
   \item[2.]    Harnack domination on noncommutative balls
   \item[3.]    Hyperbolic metric on Harnack parts of the
   noncommutative ball $\cC_\rho$
   \item[4.] Mapping theorems for free holomorphic functions on noncommutative balls
 \item[5.]    Carath\' eodory metric on the open  noncommutative ball
 $[\cC_\infty]_{<1}$ and Lipschitz  mappings
\item[6.]    Three metric topologies on Harnack parts of $\cC_\rho$
\item[7.]  Harnack domination and hyperbolic  metric for
$\rho$-contractions $($case $n=1)$
   \end{enumerate}

\quad References

}

\bigskip

\bigskip

\section*{Introduction}

  In \cite{Po-unitary}, we provided a
    generalization of the Sz.-Nagy--Foia\c s theory of $\rho$-contractions
    (see \cite{Sz1},
    \cite{SzF}, \cite{SzF-book}),
    to the multivariable setting.     An $n$-tuple
    $(T_1,\ldots, T_n)\in B(\cH)^n$ of bounded linear operators
    acting on a Hilbert space $\cH$  belongs to the class
    $\cC_\rho$,
    $\rho>0$,
    if there is a Hilbert space $\cK\supseteq \cH$ and some isometries
     $V_i\in B(\cK)$,
    \ $i=1,\ldots, n$, with orthogonal ranges such that
    $$T_\alpha =\rho P_\cH V_\alpha |_\cH\quad \text{ for any }
    \alpha\in \FF_n^+\backslash \{g_0\},
    $$
    where
   $P_\cH$ is the
orthogonal projection of $\cK$ onto $\cH$. Here,   $\FF_n^+$ stands
for
     the unital free semigroup on $n$ generators
      $g_1,\dots,g_n$, and the identity $g_0$, while
     $T_\alpha :=  T_{i_1}T_{i_2}\cdots T_{i_k}$
if $\alpha=g_{i_1}g_{i_2}\cdots g_{i_k}\in \FF_n^+$ and
$T_{g_0}:=I_\cH$, the identity on $\cH$.

According  to  the  theory of  row contractions (see \cite{SzF-book}
for the case  $n=1$, and \cite{Fr}, \cite{Bu}, \cite{Po-models},
\cite{Po-isometric}, \cite{Po-charact}, for $n\geq 2$)  we have
$$
\cC_1=[B(\cH)^n]_1^-:=\left\{(X_1,\ldots, X_n)\in B(\cH)^n:\ \|X_1
X_1^*+\cdots +X_nX_n^* \|^{1/2} \leq 1\right\}.
$$
    The results in  \cite{Po-unitary} (see Section 4)  can be seen as the unification of
     the theory of isometric dilations for row contractions \cite{Sz1},
      \cite{SzF-book},
      \cite{Fr}, \cite{Bu}, \cite{Po-models}, \cite{Po-isometric},
      \cite{Po-charact}
      (which corresponds to the case $\rho=1$)
     and Berger type dilations    for $n$-tuples $(T_1,\ldots, T_n)$ with the joint numerical
     radius $w(T_1,\ldots, T_n)\leq 1$ (which corresponds to the case $\rho=2$).

    Following the classical case (\cite{Ho1}, \cite{W}), we defined
    the {\it joint operator radius}
     $\omega_\rho: B(\cH)^{n}\to [0,\infty)$, \ $\rho>0$, by setting
     $$
     \omega_\rho(T_1,\ldots, T_n):=\inf\left\{t>0:\
     \left(\frac{1} {t}T_1,\ldots, \frac{1} {t}T_n\right)\in \cC_\rho\right\}
     $$
     and
     $\omega_\infty(T_1,\ldots, T_n):=\lim\limits_{\rho\to\infty}\omega_\rho(T_1,\ldots, T_n)$.
     In particular,  $\omega_1(T_1,\ldots, T_n)$ coincides with the norm of the row
     operator $[T_1 \,\cdots \, T_n]$, $\omega_2(T_1,\ldots, T_n)$ coincides with
      the {\it joint numerical radius}  $w(T_1,\ldots, T_n)$, and $\omega_\infty(T_1,\ldots,
      T_n)$ is equal to the   ({\it algebraic})  {\it joint spectral radius}
      (see \cite{Bu}, \cite{M})
       $$
r(T_1,\ldots, T_n):=\lim_{k\to \infty}\left\|\sum_{|\alpha|=k}
T_\alpha T_\alpha^*\right\|^{1/2k},
$$
where the length of $\alpha\in \FF_n^+$ is defined by $|\alpha|:=0$
if $\alpha=g_0$ and  by $|\alpha|:=k$ if
 $\alpha=g_{i_1}\cdots g_{i_k}$ and $i_1,\ldots, i_k\in \{1,\ldots, n\}$.
     In \cite{Po-unitary}, we considered  basic properties of the joint
      operator radius  $\omega_\rho $ and  we extended
       to the (noncommutative and commutative) multivariable setting
     several classical results obtained  by Sz.-Nagy and Foia\c s, Halmos,
       Berger and Stampfli, Holbrook, Paulsen, Badea and Cassier, and  others
     (see \cite{BC}, \cite{B},    \cite{BS}, \cite{BS1},  \cite{Ha}, \cite{Ha1},
     \cite{Ho1}, \cite{Ho2}, \cite{K}, \cite{Pa-book}, \cite{Pe}, \cite{SzF},
       and \cite{W}).

In \cite{Po-hyperbolic}, we introduced a hyperbolic   metric
$\delta$ on the open noncommutative ball $ [B(\cH)^n]_1$, which
turned out to be
   a noncommutative extension  of the Poincar\'
  e-Bergman (\cite{Be}) metric    on the open unit ball
$\BB_n:=\{ z\in \CC^n:\ \|z\|_2<1\}.$   We proved that $\delta$ is
invariant under the action of  the group $Aut([B(\cH)^n]_1)$ of all
free holomorphic automorphisms of $[B(\cH)^n]_1$, and  showed that
the $\delta$-topology and the usual operator norm topology coincide
on $[B(\cH)^n]_1$.  Moreover, we proved that   $[B(\cH)^n]_1$ is a
complete metric space with respect to the hyperbolic metric and
obtained an explicit formula for $\delta$ in terms of the
reconstruction operator.  A Schwarz-Pick lemma for  bounded free
holomorphic functions on $[B(\cH)^n]_1$, with respect to the
hyperbolic metric, was also obtained. In \cite{Po-hyperbolic2},  we
continued to study   the noncommutative hyperbolic geometry on the
unit ball of $B(\cH)^n$, its connections with multivariable dilation
theory, and its implications
 to noncommutative function theory.
The results from \cite{Po-hyperbolic} and \cite{Po-hyperbolic2} make
connections between noncommutative function theory (see
\cite{Po-poisson}, \cite{Po-holomorphic}, \cite{Po-automorphism},
\cite{Po-pluriharmonic}) and   classical results in hyperbolic
complex analysis  (see \cite{Ko1}, \cite{Ko2}, \cite{Kr},
\cite{Ru}, \cite{Zhu}).

The present paper is an attempt to extend the results
\cite{Po-hyperbolic} concerning the noncommutative hyperbolic
geometry of the unit ball  $[B(\cH)^n]_1$ to the more  general
setting  of \cite{Po-unitary}. We study the noncommutative balls
$$
[\cC_\rho]_{<1}=\left\{(X_1,\ldots, X_n)\in B(\cH)^n:\
\omega_\rho(X_1,\ldots, X_n)< 1\right\},\qquad \rho\in (0,\infty],
$$
and the Harnach parts of $\cC_\rho$, $\rho>0$,  as metric spaces
with respect to a hyperbolic (resp. Carath\'eodory) type metric that
will be introduced.  We provide mapping theorems for free
holomorphic functions on these noncommutative balls, extending
  classical results from complex analysis and hyperbolic geometry.

In Section 1, we consider  some preliminaries on  free holomorphic
(resp. pluriharmonic) functions on the open unit ball
$[B(\cH)^n]_1$, and present several characterizations for the
$n$-tuples of operators of class $\cC_\rho$, $\rho\in (0,\infty)$.
We introduce a free pluriharmonic functional calculus for the class
$\cC_\rho$ and show that a von Neumann type inequality characterizes
this class. In particular, we prove that an $n$-tuple of operators
$(T_1,\ldots, T_n)\in B(\cH)^n$ is of class $\cC_\rho$ if and only
if
$$ \|p(T_1,\ldots, T_n)\|\leq\|\rho p(S_1,\ldots, S_n) +(1-\rho) p(0)\| $$
   for any noncommutative  polynomial with matrix-valued coefficients  $p\in \CC[Z_1,\ldots,
Z_n]\otimes M_m$, $m\in \NN$,  where $S_1,\ldots, S_n$ are the left
creation operators on the full Fock space with $n$ generators.

 In Section 2,
 we  introduce  a preorder relation $\overset{H}{\prec} $
on the class $\cC_\rho$. If  $A:=(A_1,\ldots, A_n)$ and
$B:=(B_1,\ldots, B_n)$ are  in the class $\cC_\rho\subset B(\cH)^n$,
we say that $A$ is Harnack dominated by $B$ (denote
$A\overset{H}{\prec}\, B$) if there exists $c>0$ such that
$$\Re p(A_1,\ldots, A_n)+ (\rho-1)\Re p(0)\leq c^2 \left[\Re p(B_1,\ldots, B_n)+
 (\rho-1)\Re p(0)\right] $$
 for any noncommutative polynomial with
matrix-valued coefficients $p\in \CC[X_1,\ldots, X_n]\otimes M_{m}$,
$m\in \NN$, such that  $\Re p(X):=\frac{1}{2}[p(X)^*+p(X)]\geq 0$
for any $X\in [B(\cK)^n]_1$, where $\cK$ is  an infinite dimensional
Hilbert space.
 When we
want to emphasize the constant $c$, we write
$A\overset{H}{{\underset{c}\prec}}\, B$. We provide
    several   characterizations  for the Harnack domination
    on the noncommutative ball $\cC_\rho$ (see Theorem \ref{equivalent}), and
    determine   the  set of all elements in  $\cC_\rho$ which
are Harnack dominated  by $0$.  The results of this section will
play
    a major role in the next sections.

The relation $\overset{H}{\prec} $  induces an equivalence relation
$\overset{H}\sim$ on  the class $\cC_\rho$. More precisely,  two
$n$-tuples $A$ and $B$ are Harnack equivalent (and denote
$A\overset{H}{\sim}\, B$) if and only if there exists $c>  1$ such
that $A\overset{H}{{\underset{c}\prec}}\, B$ and
$B\overset{H}{{\underset{c}\prec}}\, A$ (in this case we denote
$A\overset{H}{{\underset{c}\sim}}\, B$). The equivalence classes
with respect to $\overset{H}\sim$ are called Harnack parts of
$\cC_\rho$. In Section 3, we  provide a Harnack type double
inequality   for positive free pluriharmonic functions on the
noncommutative ball $\cC_\rho$ and use it to prove that  the Harnack
part of $\cC_\rho$ which contains $0$ coincides  with the open
noncommutative ball
$$
[\cC_\rho]_{<1}:=\{(X_1,\ldots, X_n)\in B(\cH)^n:\
\omega_\rho(X_1,\ldots, X_n)<1\}.
$$
We introduce  a  hyperbolic metric $\delta_\rho:\Delta\times \Delta
\to \RR^+$ on any Harnack part $\Delta$ of $\cC_\rho$, by setting
\begin{equation*}
 \delta_\rho(A,B):=\ln \inf\left\{ c > 1: \
A\,\overset{H}{{\underset{c}\sim}}\, B   \right\},\qquad A,B\in
\Delta.
\end{equation*}
A concrete formula for the hyperbolic distance on any
  Harnack  part of $\cC_\rho$ is obtained.
   When $\Delta=[\cC_\rho]_{<1}$,  we prove that
\begin{equation*}
\delta_\rho(A,B)=\ln \max \left\{ \left\|C_{\rho,A} C_{\rho,B}^{-1}
\right\|,
  \left\|C_{\rho,B} C_{\rho,A}^{-1} \right\|\right\},\qquad A, B\in
  [\cC_\rho]_{<1},
\end{equation*}
where \begin{equation*} \begin{split} C_{\rho,X}&:=
\Delta_{\rho,X}(I-R_X)^{-1},\\
\Delta_{\rho,X}&:=  \left[\rho I+(1-\rho)(R_{X}^*+ R_{X})+
(\rho-2)R_{X}^* R_{X}\right]^{1/2},
\end{split}
\end{equation*}
 and $R_X:=X_1^*\otimes R_1+\cdots + X_n^*\otimes R_n$
 is the {\it reconstruction operator} associated with the right
creation operators $R_1,\ldots, R_n$ on the full Fock space with $n$
generators, and $X:=(X_1,\ldots, X_n)\in [\cC_\rho]_{<1}$.
We recall that the reconstruction  operator  has played an important role in  noncommutative multivariable operator theory. It
appeared  as a building block in the characteristic function associated to a row contraction (see \cite{Po-charact},  \cite{Po-charact-inv}) and also as a quantized variable (associated with the $n$-tuple $X$)  in the noncommutative  Cauchy, Poisson, and Berezin  transform, respectively (see \cite{Po-poisson}, \cite{Po-holomorphic}, \cite{Po-pluriharmonic}, \cite{Po-unitary}).

In Section 4, we study the stability of the ball $\cC_\rho$ under
contractive free holomorphic functions and provide mapping theorems,
von Neumann inequalities, and Schwarz type lemmas,
 with respect
to the hyperbolic metric $\delta_\rho$ and the operator radius
$\omega_\rho$, $\rho\in (0,\infty]$.

Let   $f:=(f_1,\ldots, f_m)$ be  a contractive  free holomorphic
function with $\|f(0)\|<1$ such that the boundary functions
$\widetilde f_1,\ldots, \widetilde f_m$ are in the noncommutative
disc algebra $\cA_n$ (see \cite{Po-von}, \cite{Po-disc}).  If an
$n$-tuple  of operators $(T_1,\ldots, T_n)\in B(\cH)^n$  is of class
$\cC_\rho$, $\rho>0$, then we prove that, under the free
pluriharmonic  functional calculus, the $m$-tuple $f(T_1,\ldots,
T_n) \in B(\cH)^m$ is of class $\cC_{\rho_f}$, where $\rho_f>0$ is
given in terms of $\rho$ and $f(0)$.

 One of the main results of this section is  the
following {\it spectral von Neumann inequality} for $n$-tuples of
operators. If
$f:=(f_1,\ldots, f_m)$  satisfies the conditions above and
 $(T_1,\ldots,
T_n)\in B(\cH)^n$  has the joint spectral radius $r(T_1,\ldots,
T_n)<1$, then $ r(f(T_1,\ldots, T_n))<1. $

 If, in addition, $f(0)=0$ and
  $\delta_\rho :\Delta\times \Delta\to [0,\infty)$
 is the hyperbolic metric on  a Harnack part $\Delta$ of $\cC_\rho$,
 then we prove that
$$
\delta_{\rho}(f(A),f(B))\leq  \delta_\rho(A,B),\qquad A,B\in \Delta.
$$
In particular, this holds when $\Delta$ is the open ball
$[\cC_\rho]_{<1}$. Moreover, in this setting,  we   show that
$$
\omega_{\rho}(f(T_1,\ldots, T_n))<1,\qquad (T_1,\ldots, T_n)\in
[\cC_\rho]_{<1},
$$
for any $\rho>0$. The general  case when $f(0)\neq 0$ is also
discussed.

In Section 5, we introduce a Carath\' eodory type metric on
 the set of all $n$-tuples of operators with
joint spectral radius strictly less then 1, i.e.,
$$[\cC_\infty]_{<1}:=\{(X_1,\ldots, X_n)\in  B(\cH)^n: \
r(X_1,\ldots, X_n)<1\},
$$
by setting
$$
d_K(A,B)=\sup_p \|\Re  p(A)-\Re p(B)\|,
$$
where the supremum is taken over all noncommutative polynomials with
matrix-valued coefficients $p\in \CC[X_1,\ldots, X_n]\otimes M_{m}$,
$m\in \NN$, with $\Re p(0)=I$ and $\Re p(X) \geq 0$ for all $X\in
[B(\cK)^n]_1$.

 We obtain
a concrete formula for $d_K$
  in terms of the free pluriharmonic kernel on the open unit ball
  $[\cC_\infty]_{<1}$. More precisely, we show that
$$
d_K(A,B)=\|P(A,R)-P(B,R)\|,\qquad A,B\in   [\cC_\infty]_{<1},
$$
where $$ P(X,R) := \sum_{k=1}^\infty \sum_{|\alpha|=k} X_{
\alpha}\otimes R_{\tilde\alpha}^* +\rho I\otimes I+
  \sum_{k=1}^\infty \sum_{|\alpha|=k} X_{\alpha}^* \otimes R_{\tilde
  \alpha},\qquad X\in   [\cC_\infty]_{<1},
$$
and  $\tilde \alpha$ is the reverse of $\alpha\in \FF_n^+$.  This is
used to prove that the metric $d_K$ is complete on
$[\cC_\infty]_{<1}$ and its topology coincides with the operator
norm topology.  We also prove that if  $f:=(f_1,\ldots, f_m)$ is a
contractive  free holomorphic function with $\|f(0)\|<1$ such that
the boundary functions $\widetilde f_1,\ldots, \widetilde f_m$ are
in the noncommutative disc algebra $\cA_n$, then
$$
d_K(f(A),f(B))\leq \frac{1+\|f(0)\|}{1-\|f(0)\|} d_K(A,B),\qquad
A,B\in [\cC_\infty]_{<1}.
$$
  As a consequence, we deduce that the map
$$
[\cC_\infty]_{<1}\ni (X_1,\ldots, X_n)\mapsto f(X_1,\ldots, X_n)\in
[\cC_\infty]_{<1}
$$
is continuous in the operator norm topology.

In Section 6, we compare the hyperbolic metric $\delta_\rho$ with
the Carath\' eodory metric $d_K$, and the operator metric,
respectively, on Harnack parts of the unit ball $\cC_\rho$,
$\rho>0$. In particular, we  prove that  the hyperbolic metric
$\delta_\rho$ is complete on the open unit unit ball
$[\cC_\rho]_{<1}$, while the other two metrics, mentioned above, are
not complete. On the other hand, we show  the
$\delta_\rho$-topology, the $d_K$-topology, and the operator norm
topology coincide on $[\cC_\rho]_{<1}$.

In Section 7,  we consider the single variable case ($n=1$) and show
that  our Harnack domination  for  $\rho$-contractions is equivalent
to the one introduced and studied  by G. Cassier and N. Suciu in
\cite{CaSu} and \cite{CaSu2}. Consequently, we recover some of their
results and, moreover, we obtain some results which seem to be new
even in the single variable case.

 Finally, we want to acknowledge
that we were  influenced in writing this paper by the work of  C. Foia\c s (\cite{Fo}), I. Suciu
(\cite{Su2}), and
 G. Cassier and N. Suciu (\cite{CaSu},  \cite{CaSu2}) concerning the  Harnack domination  and  the hyperbolic distance between two $\rho$-contractions.
It will be interesting to see to which extent the results of this
paper, concerning the hyperbolic geometry on  noncommutative balls,
can be extended to the Hardy algebras of Muhly and Solel (see
\cite{MuSo1}, \cite{MuSo2}, \cite{MuSo3}).

\bigskip

\section{ The noncommutative ball $\cC_\rho$ and  a free pluriharmonic  functional
calculus}

In this section,  we consider  some preliminaries on  free
holomorphic (resp. pluriharmonic) functions on the unit ball
$[B(\cH)^n]_1$, and several characterizations for the $n$-tuples of
operators of class $\cC_\rho$.  We introduce a free pluriharmonic
functional calculus for the class $\cC_\rho$ and show that a von
Neumann type inequality characterizes  the class $\cC_\rho$.

 Let $H_n$ be an $n$-dimensional complex  Hilbert space with
orthonormal
      basis
      $e_1$, $e_2$, $\dots,e_n$, where $n=1,2,\dots$, or $n=\infty$.
       The full Fock space  of $H_n$  is defined by
      $$F^2(H_n):=\CC1\oplus \bigoplus_{k\geq 1} H_n^{\otimes k},$$
      where  $H_n^{\otimes k}$ is the (Hilbert)
      tensor product of $k$ copies of $H_n$.
      We define the left  (resp.~right) creation
      operators  $S_i$ (resp.~$R_i$), $i=1,\ldots,n$, acting on the full
       Fock space  $F^2(H_n)$  by
      setting
      $$
       S_i\varphi:=e_i\otimes\varphi, \qquad  \varphi\in F^2(H_n),
      $$
       (resp.~$
       R_i\varphi:=\varphi\otimes e_i, \quad  \varphi\in F^2(H_n)
      $).
We recall that the noncommutative disc algebra $\cA_n$
(resp.~$\cR_n$) is the norm closed algebra generated by the left
(resp.~right) creation operators and the identity. The
noncommutative analytic Toeplitz algebra $F_n^\infty$
(resp.~$\cR_n^\infty$)
 is the  weakly
closed version of $\cA_n$ (resp.~$\cR_n$). These algebras were
introduced in \cite{Po-von} in connection with a  von Neumann type
inequality \cite{vN}, as noncommutative analogues of the disc
algebra $A(\DD)$ and the Hardy space $H^\infty(\DD)$. For more
information on theses  noncommutative algebras we refer the reader
to
 \cite{Po-multi}, \cite{Po-funct}, \cite{Po-analytic},
\cite{Po-disc},  \cite{D},  and the references therein.

 Let $\cH$ be a Hilbert space  and let $B(\cH)$ be the algebra of all bounded linear
  operators on $\cH$.
We  identify $M_m(B(\cH))$, the set of $m\times m$ matrices with
entries from $B(\cH)$, with
$B( \cH^{(m)})$, where $\cH^{(m)}$ is the direct sum of $m$ copies
of $\cH$.
 If $\cX$ is an operator space, i.e., a closed subspace of
$B(\cH)$, we consider $M_m(\cX)$ as a subspace of $M_m(B(\cH))$ with
the induced norm.
Let $\cX, \cY$ be operator spaces and $u:\cX\to \cY$ be a linear
map. Define the map
$u_m:M_m(\cX)\to M_m(\cY)$ by
$$
u_m ([x_{ij}]):=[u(x_{ij})].
$$
We say that $u$ is completely bounded  if
$$
\|u\|_{cb}:=\sup_{m\ge1}\|u_m\|<\infty.
$$
If $\|u\|_{cb}\leq1$
(resp. $u_m$ is an isometry for any $m\geq1$) then $u$ is completely
contractive (resp. isometric), and if $u_m$ is positive for all $m$,
then $u$ is called
 completely positive. For basic results concerning  completely bounded maps
 and operator spaces we refer to \cite{Pa-book}, \cite{Pi}, and \cite{ER}.

A few more notations and definitions are necessary. If $\omega,
\gamma\in \FF_n^+$, we say that  $\omega
>_{l}\gamma$ if there is $\sigma\in
\FF_n^+\backslash\{g_0\}$ such that $\omega= \gamma \sigma$ and set
$\omega\backslash_l \gamma:=\sigma$.
 We denote by
$\tilde\alpha$  the reverse of $\alpha\in \FF_n^+$, i.e.,
  $\tilde \alpha= g_{i_k}\cdots g_{i_1}$ if
   $\alpha=g_{i_1}\cdots g_{i_k}\in\FF_n^+$.
An operator-valued  positive semidefinite kernel on the free
semigroup $\FF_n^+$ is a map
$ K:\FF_n^+\times\FF_n^+\to B(\cH) $ with the property that for each
$k\in\NN$, for each choice of vectors
 $h_1,\dots,h_k$ in $\cH$, and
$\sigma_1,\dots,\sigma_k$ in $\FF_n^+$, the inequality
$$
\sum\limits_{i,j=1}^k\langle K(\sigma_i,\sigma_j)h_j,h_i\rangle\ge 0
$$
holds. Such a kernel is called multi-Toeplitz if it has the
following properties:
$K(\alpha,\alpha)=I_\cH$  for  any $\alpha\in \FF_n^+$,
and
$$
K(\sigma,\omega)=
\begin{cases}K(g_0,\omega \backslash_l \sigma)\ &\text{ if } \omega>_l\sigma  \\
 K(\sigma\backslash_l\omega,g_0) \ &\text{ if }\sigma>_l\omega  \\
0 &\text { otherwise. } \end{cases}
$$

 An  $n$-tuple of operators $(T_1,\ldots, T_n)$, \ $T_i\in B(\cH)$,
   belongs to the
   class $\cC_\rho$, $\rho>0$, if there exist a Hilbert
  space $\cK\supseteq \cH$
 and isometries  $V_i\in B(\cK)$, $i=1,\ldots, n$, with orthogonal
  ranges,  such that
  \begin{equation*}
  T_\alpha =\rho P_\cH V_\alpha |_\cH,\qquad
\alpha\in \FF_n^+\backslash\{g_0\},
\end{equation*}
where
  $P_\cH$ is the orthogonal projection of $\cK$ onto $\cH$. If
  $\cK=\cK_T:=\bigvee_{\alpha\in \FF_n^+}
V_\alpha \cH$, then  the $n$-tuple $(V_1,\ldots, V_n)$ is the
minimal isometric dilation of $(T_1,\ldots, T_n)$, which is unique
up to an isomorphism. Note that if $(T_1,\ldots, T_n)\in \cC_\rho$,
then
   the   joint spectral radius  $r (T_1,\ldots, T_n)\leq
1$, where
$$
r(T_1,\ldots, T_n):=\lim_{k\to \infty}\left\|\sum_{|\alpha|=k}
T_\alpha T_\alpha^*\right\|^{1/2k}.
$$
 We recall (see Corollary 1.36 from \cite{Po-unitary}) that
$\bigcup\limits_{\rho>0}\cC_\rho$ is dense (in the operator norm
topology) in the set of all $n$-tuples of operators with joint
spectral radius $r (T_1,\ldots, T_n)\leq 1$. Moreover, any $n$-tuple
of operators with $r(T_1,\ldots, T_n)<1$ is of class $\cC_\rho$ for
some $\rho>0$. We should add that (see Theorem 5.9 from \cite{Po-similarity})
$(T_1,\ldots, T_n)\in B(\cH)^n$ has the joint spectral radius
$r(T_1,\ldots, T_n)<1$ if and only if  it is uniformly stable, i.e.,
$\|\sum_{|\alpha|=k} T_\alpha T_\alpha^*\|\to 0$, as $k\to\infty$.

Since  the joint spectral radius of $n$-tuples of operators plays an
important role in the present paper, we recall (see \cite{Bu},
\cite{M}) some of its properties. The joint right spectrum
  $\sigma_r(T_1,\ldots, T_n)$ of an   $n$-tuple
 $(T_1,\ldots, T_n)$ of operators
   in $B(\cH)$ is the set of all $n$-tuples
    $(\lambda_1,\ldots, \lambda_n)$  of complex numbers such that the
     right ideal of $B(\cH)$  generated by the operators
     $\lambda_1I-T_1,\ldots, \lambda_nI-T_n$ does
      not contain the identity operator.
      We know that  $\sigma_r(T_1,\ldots, T_n)$ is included in the
      closed ball of $\CC^n$ of radius $r(T_1,\ldots, T_n)$.

 If  we assume that $T_1,\ldots, T_n\in B(\cH)$ are mutually commuting operators and
  $\cB$ is a closed subalgebra of $B(\cH)$ containing $T_1,\ldots, T_n$, and the
  identity,
   then  the Harte spectrum $\sigma(T_1,\ldots, T_n)$ is the set of all
 $(\lambda_1,\ldots, \lambda_n)\in \CC^n$ such that
 $$(\lambda_1I-T_1)X_1+\cdots +(\lambda_n I-T_n)X_n\neq I
 $$
 for all $X_1,\ldots, X_n\in \cB$. In this case, we have
 $$
 r(T_1,\ldots, T_n)=\max\{\|(\lambda_1,\ldots, \lambda_n)\|_2:\
 (\lambda_1,\ldots, \lambda_n)\in \sigma(T_1,\ldots, T_n)\}.
 $$
According to \cite{M}, the latter  formula   remains true if the
Harte spectrum  is replaced by   the Taylor's spectrum for commuting
operators.

According to Theorem 4.1 from   \cite{Po-posi} and  Theorems 1.34 and 1.39 from \cite{Po-unitary}, we have
       the following  characterizations
    for the  $n$-tuples of operators of class
     $\cC_\rho$.
We denote by $\CC[Z_1,\ldots, Z_n]$ the set of all noncommutative
     polynomials  in $n$ noncommuting indeterminates.

\begin{theorem}\label{ro-contr}
Let $T_1,\ldots, T_n\in B(\cH)$ and let $\cS\subset C^*(S_1,\ldots,
S_n)$
  be the operator system defined by
  \begin{equation*}
  \cS:=\{p(S_1,\ldots, S_n)+q(S_1,\ldots, S_n)^*:\ p,q\in \CC[Z_1,\ldots, Z_n]\}.
  \end{equation*}
 Then the following statements are equivalent:
 \begin{enumerate}
 \item[(i)] $(T_1,\ldots, T_n)\in \cC_\rho.$

\item[(ii)]
   The map $\Psi:\cS\to B(\cH)$ defined by
 \begin{equation*}
 \begin{split}
 \Psi\left(p(S_1,\ldots, S_n)+q(S_1,\ldots, S_n)^*\right):=
 p(T_1,\ldots, T_n)&+q(T_1,\ldots,T_n)^*\\
&+ (\rho-1)(p(0)+\overline{q(0)})I
 \end{split}
 \end{equation*}
 is completely positive.
 \item[(iii)] The joint spectral radius $r(T_1\ldots, T_n)\leq 1$ and
  the $\rho$-pluriharmonic kernel
    defined by
  \begin{equation*}
 P_\rho(rT,R) := \sum_{k=1}^\infty \sum_{|\alpha|=k}r^{|\alpha|}
   T_{  \alpha}\otimes R_{\tilde\alpha}^* +\rho I\otimes I+
  \sum_{k=1}^\infty \sum_{|\alpha|=k}r^{|\alpha|}
  T_{\alpha}^* \otimes R_{\tilde \alpha}
  \end{equation*}
  is positive for any $0<r<1$, where the convergence is in the operator norm topology.
 \item[(iv)] The spectral radius $r(T_1,\ldots, T_n)\leq 1$ and
 $$
 \rho I\otimes I+(1-\rho)r\sum_{i=1}^n ( T_i\otimes R_i^*+  T_i^*\otimes R_i)+
 (\rho-2)r^2\left( \sum_{i=1}^n T_iT_i^*\otimes I\right)\geq 0
 $$
 for any $0<r<1$.
 \item[(v)]
 The multi-Toeplitz kernel $K_{ \rho, T}:\FF_n^+\times \FF_n^+\to
B(\cH)$ defined
    by
   $$
   K_{ \rho, T}(\alpha, \beta):=
   \begin{cases}
  \frac {1} {\rho}   T_{\beta\backslash_l \alpha}
   &\text{ if } \beta>_l\alpha\\
   I  &\text{ if } \alpha=\beta\\
   \frac {1} {\rho} (T_{\alpha\backslash_l \beta})^*
     &\text{ if } \alpha>_l\beta\\
    0\quad &\text{ otherwise}
   \end{cases}
   $$
is positive semidefinite.
\end{enumerate}
\end{theorem}

Consider  $1\leq m<n$ and let $(R_1',\ldots,R_m')$  and
$(R_1,\ldots, R_n)$ be the right creation operators on $F^2(H_m)$
and  $F^2(H_n)$, respectively. According to the Wold type
decomposition for isometries with orthogonal ranges
\cite{Po-isometric}, the $m$-tuple $(R_1,\ldots, R_m)$ is unitarily
equivalent to $(R_1'\otimes I_\cE,\ldots,R_m'\otimes I_\cE)$, where
$\cE$ is equal to $F^2(H_n)\ominus F^2(H_m)$. Consequently, using
Theorem \ref{ro-contr}, one can easily deduce the following result.

\begin{corollary} \label{T,0}
Let $\rho>0$, $1\leq m<n$,  and  consider an $m$-tuple $(T_1,\ldots,
T_m)\in B(\cH)^m$ and  its extension   $(T_1,\ldots, T_m,0,\ldots,
0) \in B(\cH)^n$. Then the following statements hold:
\begin{enumerate}
\item[(i)]
$(T_1,\ldots, T_m)\in \cC_\rho$  if and only if $ (T_1,\ldots,
T_m,0,\ldots, 0)\in \cC_\rho$;
\item[(ii)]
$\omega_\rho(T_1,\ldots, T_m)=\omega_\rho (T_1,\ldots, T_m,0,\ldots,
0));$
\item[(iii)] $r(T_1,\ldots, T_m)=r (T_1,\ldots, T_m,0,\ldots,
0)$.
\end{enumerate}
\end{corollary}

  Throughout this paper, we assume that
 $\cE$ is a  separable Hilbert space.
We recall \cite{Po-holomorphic} that a  mapping
$F:[B(\cH)^n]_{1}\to B( \cH)\bar \otimes_{min}B(\cE)$ is called
  {\it free
holomorphic function} on  $[B(\cH)^n]_{1}$  with coefficients in
$B(\cE)$ if there exist $A_{(\alpha)}\in B(\cE)$, $\alpha\in
\FF_n^+$, such that $\limsup_{k\to \infty} \left\| \sum_{|\alpha|=k}
A_{(\alpha)}^* A_{(\alpha)}\right\|^{1/2k}\leq 1$ and
$$
F(X_1,\ldots, X_n)=\sum\limits_{k=0}^\infty \sum\limits_{|\alpha|=k}
X_\alpha\otimes  A_{(\alpha)},
$$
where the series converges in the operator  norm topology  for any
$(X_1,\ldots, X_n)$ in the open unit ball $
[B(\cH)^n]_{1}:=\{(X_1,\ldots, X_n):\ \|X_1X_1^*+\cdots+
X_nX_n\|<1\}$. The set of all free holomorphic functions on
$[B(\cH)^n]_1$ with coefficients in $B(\cE)$ is denoted by $H_{\bf
ball}(B(\cE))$.
Let $H_{\bf ball}^\infty(B(\cE))$  denote the set of  all elements
$F$ in $H_{\bf ball}(B(\cE))$  such that
$$
\|F\|_\infty:=\sup  \|F(X_1,\ldots, X_n)\|<\infty,
$$
where the supremum is taken over all $n$-tuples  of operators
$(X_1,\ldots, X_n)\in [B(\cH)^n]_1$ and any Hilbert space $\cH$.
According to \cite{Po-holomorphic} and \cite{Po-pluriharmonic},
$H_{\bf ball}^\infty(B(\cE))$ can be identified to the operator
algebra $ F_n^\infty\bar \otimes B(\cE)$ (the weakly closed algebra
generated by the spatial tensor product), via the noncommutative
Poisson transform. Due to the fact that a free holomorphic function
is uniquely determined by its representation on an infinite
dimensional Hilbert space, we  identify, throughout this paper,  a
free holomorphic function   with its representation on a   separable
infinite dimensional Hilbert space.

We say that a map $u:[B(\cH)^n]_1\to B(\cH)\bar\otimes_{min} B(\cE)$
is a self-adjoint {\it free pluriharmonic function} on
$[B(\cH)^n]_1$ if $u=\Re f:=\frac{1}{2}(f^*+f)$ for some free
holomorphic function $f$. A free pluriharmonic function on
$[B(\cH)^n]_1$ has the form $H:=H_1+ iH_2$, where $H_1,H_2$ are
self-adjoint free pluriharmonic  functions on $[B(\cH)^n]_1$. We
recall \cite{Po-pluriharmonic} that  if
$$f(Z_1,\ldots, Z_n)=\sum_{k=1}^\infty \sum_{|\alpha|=k}
Z_\alpha^*\otimes B_{(\alpha)} +
    I\otimes A_{(0)}+\sum_{k=1}^\infty \sum_{|\alpha|=k}   Z_\alpha\otimes A_{(\alpha)}$$
  is a free  pluriharmonic  function on  $[B(\cH)^n]_1$ with
  coefficients in $B(\cE)$
   and  $(T_1,\ldots, T_n)\in B(\cH)^n$ is any $n$-tuple of operators
   with joint spectral radius
$r(T_1,\ldots, T_n)<1$, then $ f(T_1,\ldots, T_n) $ is  a bounded
linear operator, where the corresponding series converge in norm.
Moreover $\lim_{r\to 1} f(rT_1,\ldots, rT_n)=f(T_1,\ldots T_n)$ in
the  operator norm  topology. We refer to \cite{Po-pluriharmonic}
for more results  on free pluriharmonic functions.

We denote by $Har_{\bf ball}^c(B(\cE))$ the set of all
  free pluriharmonic functions on $[B(\cH)^n]_1$ with operator-valued
  coefficients in $B(\cE)$, which
 have continuous extensions   (in the operator norm topology) to
the closed ball $[B(\cH)^n]_1^-$. We assume that $\cH$ is an
infinite dimensional Hilbert space. According to Theorem 4.1 from
\cite{Po-pluriharmonic}, we can identify $Har_{\bf ball}^c(B(\cE))$
with  the operator space  $\overline{\cA_n(\cE)^*+
\cA_n(\cE)}^{\|\cdot\|}$, where $\cA_n(\cE):=\cA_n\bar\otimes_{min}
B(\cE)$ and $\cA_n$ is the noncommutative disc algebra. More
precisely, if $u:[B(\cH)^n]_1\to  B( \cH)\bar \otimes_{min}B(\cE)$,
then
 the following statements are equivalent:
\begin{enumerate}
\item[(a)] $u$ is a free pluriharmonic function on $[B(\cH)^n]_1$ which
 has a continuous extension  (in the operator norm topology) to
the closed ball $[B(\cH)^n]_1^-$;
\item[(b)]
there exists $f\in \overline{\cA_n(\cE)^*+\cA_n(\cE)}^{\|\cdot\|}$
such that $u(X)=(P_X\otimes \text{\rm id})(f)$ for $X\in
[B(\cH)^n]_1$, where $ P_X$ is the noncommutative Poisson transform
at $X$;
\item[(c)] $u$ is a free pluriharmonic function on $[B(\cH)^n]_1$
such that \ $u(rS_1,\ldots, rS_n)$ converges in the operator norm
topology, as $r\to 1$.
\end{enumerate}
In this case, $  f=\lim\limits_{r\to 1}u(rS_1,\ldots, rS_n),$ where
the convergence is in the operator norm. Moreover, the map $ \Phi:
Har_{\bf ball}^c(B(\cE))\to
\overline{\cA_n(\cE)^*+\cA_n(\cE)}^{\|\cdot\|}\quad \text{ defined
by } \quad \Phi(u):=f $ is a  completely   isometric isomorphism of
operator spaces. We call  $ f$  the {\it model boundary function} of
$u$.

Now, we introduce  a  free pluriharmonic functional calculus for the
class $\cC_\rho$.

\begin{theorem}\label{funct-calc} Let $T:=(T_1,\ldots, T_n)\in B(\cH)^n$
be of class  $\cC_\rho$, and let $u\in Har_{\bf ball}^c(B(\cE))$
have the  standard representation
$$u(X_1,\ldots, X_n)=\sum_{k=1}^\infty \sum_{|\alpha|=k}
X_\alpha^*\otimes B_{(\alpha)} +
    I\otimes A_{(0)}+\sum_{k=1}^\infty \sum_{|\alpha|=k}
      X_\alpha\otimes A_{(\alpha)},\qquad X\in [B(\cH)^n]_1,$$
      for some $A_{(\alpha)}, B_{(\alpha)}\in B(\cE)$,
where the series  converge in the operator norm  topology. Then
$$
u(T_1,\ldots, T_n):=\lim_{r\to 1} u(rT_1,\ldots, rT_1)
$$
exists in the operator norm and
$$
\| u(T_1,\ldots, T_n)\|\leq \|\rho u +(1-\rho) u(0)\|_\infty.
$$
\end{theorem}

\begin{proof}
 Since
$T:=(T_1,\ldots, T_n)\in B(\cH)^n$ is an $n$-tuple  of class
$\cC_\rho$, there is a minimal isometric dilation $V:=(V_1,\ldots,
V_n)$ of $T$ on a Hilbert space $\cK_T\supseteq \cH$, satisfying the
following properties: $V_i^*V_j=\delta_{ij} I$ for $i,j=1,\ldots,n$,
and
 $T_\alpha =\rho P_\cH V_\alpha |_{\cH}$  for any  $
    \alpha\in \FF_n^+\backslash \{g_0\}$, and
     $\cK_T=\bigvee_{\alpha\in \FF_n^+} V_\alpha \cH$.
      Taking into account that $u\in Har_{\bf ball}^c(B(\cE))$, we have
$$u(rV_1,\ldots, rV_n)=\sum_{k=1}^\infty \sum_{|\alpha|=k}
r^{|\alpha|}V_\alpha^*\otimes B_{(\alpha)} +
    I\otimes A_{(0)}+\sum_{k=1}^\infty \sum_{|\alpha|=k}
      r^{|\alpha|}V_\alpha\otimes A_{(\alpha)},
      $$
where the convergence is in the operator norm. Hence, and due to the
fact that
$$\sum_{|\alpha|=k} r^{|\alpha|}T_\alpha^*\otimes B_{(\alpha)}
=\rho (P_\cH\otimes I)\left(\sum_{|\alpha|=k}
r^{|\alpha|}V_\alpha^*\otimes B_{(\alpha)}\right)|_{\cH\otimes
\cE},\qquad k=1,2,\ldots,
$$
 we deduce that
\begin{equation*}
\begin{split}
u(rT_1,\ldots, rT_n)&:= \sum_{k=1}^\infty \sum_{|\alpha|=k}
r^{|\alpha|}T_\alpha^*\otimes B_{(\alpha)} +
    I\otimes A_{(0)}+\sum_{k=1}^\infty \sum_{|\alpha|=k}
      r^{|\alpha|}T_\alpha\otimes A_{(\alpha)}\\
      &=\rho (P_\cH\otimes I)u(rV_1,\ldots, rV_n)|_{\cH\otimes \cE}
      -(\rho-1)u(0).
      \end{split}
      \end{equation*}
 exists  in the operator norm topology.
 Now, taking into account that  $ \lim_{r\to 1}
u(rV_1,\ldots, rV_1) $ exists in the  operator norm,  we deduce that
$ \lim_{r\to 1} u(rT_1,\ldots, rT_1)$ exists in the same topology.
Consequently, we can define
$$
u(T_1,\ldots, T_n):=\lim_{r\to 1} u(rT_1,\ldots, rT_1).
$$
 Using the considerations above, and the noncommutative von Neumann inequality,
  we obtain
$$\|u(T_1,\ldots, T_n)\|\leq\|\rho u +(1-\rho) u(0)\|_\infty \leq  (\rho+ |\rho-1|) \| u\|_\infty
$$
for any $(T_1,\ldots, T_n)\in \cC_\rho$.
\end{proof}

We will refer to the map
$$Har_{\bf ball}^c(B(\cE)) \ni  u \mapsto
u(T_1,\ldots, T_n)\in B(\cH)\bar \otimes _{min} B(\cE) $$
 as the
free pluriharmonic functional calculus for the class $\cC_\rho$.
Since there is a completely   isometric isomorphism of operator
spaces $\overline{\cA_n(\cE)^*+\cA_n(\cE)}^{\|\cdot\|} \ni f\mapsto
u\in Har_{\bf ball}^c(B(\cE))$,  given by $u=(P_X\otimes \text{\rm
id})(f)$ for $X\in [B(\cH)^n]_1$,
 we  also use the notation $ f(T_1,\ldots, T_n)$ for
$u(T_1,\ldots, T_n)$.

Now, we show that the von Neumann type inequality of Theorem
\ref{funct-calc} characterizes the class $\cC_\rho$. Denote
$$\cP(S_1,\ldots, S_n):=\{p(S_1,\ldots, S_n) :\ p\in \CC[Z_1,\ldots,
Z_n]\},
$$
where $S_1,\ldots, S_n$ are the left creation operators on the full
Fock space $F^2(H_n)$.
\begin{theorem}\label{new-caract}
Let $T:=(T_1,\ldots, T_n)\in B(\cH)^n$ be an $n$-tuple of operators.
Then the following statements are equivalent:

\begin{enumerate}
\item[(i)]
$T$  is of class $\cC_\rho$;
\item[(ii)] the von Neumann type inequality
$$ \|p(T_1,\ldots, T_n)\|\leq\|\rho p(S_1,\ldots, S_n) +(1-\rho) p(0)\| $$
holds  for any noncommutative  polynomial  $p\in \CC[Z_1,\ldots,
Z_n]\otimes M_m$, $m\in \NN$;
\item[(iii)]
the map $\Psi_T:\cA_n\to B(\cH)$ defined by
 \begin{equation*}
 \Psi_T\left(q(S_1,\ldots, S_n)\right):=
 \frac{1}{\rho}q(T_1,\ldots, T_n)+
 \left(1-\frac{1}{\rho}\right)q(0)I, \qquad  q(S_1,\ldots, S_n)\in \cP(S_1,\ldots,
 S_n),
 \end{equation*}
 is completely contractive.

\end{enumerate}
\end{theorem}

\begin{proof} The implication $(i)\implies (ii)$ follows, in
particular, from Theorem \ref{funct-calc}. To prove the implication
$(ii)\implies (iii)$, note that setting $p:=\frac{1}{\rho}q+
 \left(1-\frac{1}{\rho}\right)q(0)I$, where $q\in \CC[Z_1,\ldots, Z_n]\otimes M_m$,
$m\in \NN$,
 we have
 \begin{equation*}
 \begin{split}
\left\|\Psi_T\left(q(S_1,\ldots,
S_n)\right)\right\|&=\|p(T_1,\ldots, T_n)\|\\
&\leq\|\rho p(S_1,\ldots, S_n) +(1-\rho) p(0)\|\\
&=\|q(S_1,\ldots, S_n)\|,
\end{split}
\end{equation*}
which proves  that $\Psi_T$ is completely contractive on the set of
all polynomials $\cP(S_1,\ldots, S_n)$ and, consequently,  extends
uniquely to a completely contractive map on the noncommutative disc
algebra $\cA_n$. It remains to prove that $(iii)\implies (i)$. Due
to Arveson's  extension theorem, item (iii) implies the existence of
a unique completely positive  extension $\widetilde{\Psi}_T:\cA_n^*+
\cA_n\to B(\cH)$ of $ \Psi_T$. Note that
$$
\widetilde\Psi_T(r(S_1,\ldots, S_n)+q(S_1,\ldots, S_n)^*)=
\frac{1}{\rho}(r(T_1,\ldots, T_n)+q(T_1,\ldots,T_n)^*)+
 \left(1-\frac{1}{\rho}\right)(r(0)+\overline{q(0)}I
 $$
for any polynomials $r(S_1,\ldots, S_n)$ and $q(S_1,\ldots, S_n)$ in
$\cP(S_1,\ldots, S_n)$. Applying Theorem \ref{ro-contr}  (the
equivalence $(i)\leftrightarrow (ii)$), we complete the proof.
\end{proof}

\bigskip

\bigskip

\section{Harnack   domination on  noncommutative balls}

 We introduce a
preorder relation $\overset{H}{\prec} $ on the noncommutative ball
$\cC_\rho$, $\rho\in (0,\infty)$,  and provide
    several   characterizations. We determine   the elements of  $\cC_\rho$ which
are Harnack dominated  by $0$. These results will play a crucial
role in the next sections.

First, we consider  some preliminaries on noncommutative Poisson
transforms.
 Let $C^*(S_1,\ldots, S_n)$ be the Cuntz-Toeplitz
$C^*$-algebra generated by the left creation operators (see
\cite{Cu}). The noncommutative Poisson transform at
 $T:=(T_1,\ldots, T_n)\in [B(\cH)^n]_1^-$ is the unital completely contractive  linear map
 $P_T:C^*(S_1,\ldots, S_n)\to B(\cH)$ defined by
 \begin{equation*}
 P_T[f]:=\lim_{r\to 1} K_{rT}^* (I_\cH \otimes f)K_{rT}, \qquad f\in C^*(S_1,\ldots,
 S_n),
\end{equation*}
 where the limit exists in the operator  norm topology of $B(\cH)$.
Here, the noncommutative Poisson  kernel $ K_{rT} :\cH\to
\overline{\Delta_{rT}\cH} \otimes  F^2(H_n)$, $ 0< r\leq 1$, is
defined by
\begin{equation*}
K_{rT}h:= \sum_{k=0}^\infty \sum_{|\alpha|=k} r^{|\alpha|}
\Delta_{rT} T_\alpha^*h\otimes  e_\alpha,\qquad h\in \cH,
\end{equation*}
where $\{e_\alpha\}_{\alpha\in \FF_n^+}$ is the orthonormal basis
for $F^2(H_n)$, defined by
  $e_\alpha:=
e_{i_1}\otimes\cdots \otimes  e_{i_k}$  if $\alpha=g_{i_1}\cdots
g_{i_k}\in \FF_n^+$ and $e_{g_0}:=1$, and
$\Delta_{rT}:=(I_\cH-r^2T_1T_1^*-\cdots -r^2 T_nT_n^*)^{1/2}$.
 We
recall that
 $
 P_T[S_\alpha S_\beta^*]=T_\alpha T_\beta^*$, $ \alpha,\beta\in \FF_n^+.
 $
 When $T:=(T_1,\ldots, T_n)$  is a pure row contraction, i.e.,
 $ \text{\rm SOT-}\lim\limits_{k\to\infty} \sum_{|\alpha|=k}
T_\alpha T_\alpha^*=0$, then
   we have $$P_T[f]=K_T^*(I_{\cD_{T}}\otimes f)K_T, \qquad f\in C^*(S_1,\ldots,
 S_n)\ \text{ or } \ f\in F_n^\infty,
   $$
   where $\cD_T:=\overline{\Delta_T \cH}$. We refer to \cite{Po-poisson}, \cite{Po-curvature}, and
\cite{Po-unitary} for more on noncommutative Poisson transforms on
$C^*$-algebras generated by isometries.

 A free pluriharmonic function $u$  on $[B(\cK)^n]_1$  with operator
 valued coefficients  is called positive, and denote $u\geq 0$,  if
$u(X_1,\ldots,X_n)\geq 0$  for any $(X_1,\ldots, X_n)\in
[B(\cK)^n]_1$, where $\cK$ is an infinite dimensional Hilbert space.
We mention that it is enough to assume that  the positivity
condition holds for  any finite dimensional Hilbert space $\cK$.
Indeed, for each $m\in \NN$, consider $R^{(m)}:=(R_1^{(m)},\ldots,
R_n^{(m)})$, where $R_i^{(m)}$  is the compression of the right
creation operator $R_i$ to the subspace $\cP_m:=\text{\rm
 span}\,\{e_\alpha:\ \alpha\in \FF_n^+, |\alpha|\leq m\}$ of $F^2(H_n)$.
 We recall from \cite{Po-pluriharmonic} the following result.
\begin{lemma}\label{L1}
\label{positive} Let $u$ be a free pluriharmonic function on
$[B(\cK)^n]_1$ with operator-valued coefficients.  Then
$u(X_1,\ldots, X_n)\geq 0$ for any  $(X_1,\ldots, X_n)\in
[B(\cK)^n]_1$ if and only if  $u(R_1^{(m)},\ldots, R_n^{(m)})\geq 0$
for any $m\in \NN$.
\end{lemma}

   Let
$A:=(A_1,\ldots, A_n)$ and $B:=(B_1,\ldots, B_n)$ be  $n$-tuples of
operators in $\cC_\rho\subset B(\cH)^n$. We say that $A$ is Harnack
dominated by $B$, and denote $A\overset{H}{\prec}\, B$, if there
exists $c>0$ such that
$$\Re p(A_1,\ldots, A_n)+ (\rho-1)\Re p(0)\leq c^2 \left[\Re p(B_1,\ldots, B_n)+
 (\rho-1)\Re p(0)\right] $$
 for any noncommutative polynomial with
matrix-valued coefficients $p\in \CC[X_1,\ldots, X_n]\otimes M_{m}$,
$m\in \NN$, such that $\Re p\geq 0$.
 When we
want to emphasize the constant $c$, we write
$A\overset{H}{{\underset{c}\prec}}\, B$.

According to Theorem \ref{funct-calc}, we can associate with each
$n$-tuple $T:=(T_1,\ldots, T_n)\in \cC_\rho$ the completely positive
map $\varphi_T: \overline{\cA_n^*+ \cA_n}^{\|\cdot\|}\to B(\cH)$
defined by
\begin{equation}
\label{var}
 \varphi_T(g):=\frac{1}{\rho} g(T_1,\ldots,
T_n)+\left(1-\frac{1}{\rho}\right) g(0),
\end{equation}
where $g(T_1,\ldots, T_n)$ is defined by the free pluriharmonic
functional calculus for the class $\cC_\rho$.

Now, we present  several  characterizations for the  Harnack
domination in $ \cC_\rho$.

\begin{theorem}
\label{equivalent} Let $A:=(A_1,\ldots, A_n)\in B(\cH)^n$ and
$B:=(B_1,\ldots, B_n)\in B(\cH)^n$ be in  the class $ \cC_\rho$ and
let $c>0$. Then the following statements are equivalent:
\begin{enumerate}
\item[(i)]
$A\overset{H}{{\underset{c}\prec}}\, B$;

\item[(ii)]
$P_\rho( rA, R)\leq c^2 P_\rho(rB, R)$ for any $r\in [0,1)$, where
$P_\rho(X,R)$ is the  multi-Toeplitz  kernel  associated with $X\in
\cC_\rho$;

\item[(iii)]
$u(rA_1,\ldots, rA_n)+(\rho-1)u(0)\leq c^2 \left[u(rB_1,\ldots,
rB_n)+ (\rho-1)u(0)\right]$ for any positive free pluriharmonic
function $u$ on $[B(\cH)^n]_1$   with operator-valued coefficients
and any $r\in [0,1)$;

 \item[(iv)]
$K_{\rho,A}\leq c^2 K_{\rho, B}$, where $K_{\rho,X}$ is the
multi-Toeplitz kernel associated with $X\in \cC_\rho$;
\item[(v)]
$  c^2{\varphi}_B-{\varphi}_A$   is a completely positive linear map
on the operator space $\overline{\cA_n^*+ \cA_n}^{\|\cdot\|}$, where
$\varphi_A$, $\varphi_B$ are the c.p. maps associated with $A$ and
$B$, respectively.
\item[(vi)]
there is an operator $L_{B,A}\in B(\cK_B, \cK_A)$ with
$\|L_{B,A}\|\leq c$ such that $L_{B,A}|_\cH=I_\cH$ and
\begin{equation*}
 L_{B,A} W_i=V_iL_{B,A},\qquad i=1,\ldots,n,
\end{equation*}
where $(V_1,\ldots, V_n)$ on $\cK_A\supset \cH$ and $(W_1,\ldots,
W_n)$ on $\cK_A\supset \cH$ are the minimal isometric dilations of
$A$ and $B$, respectively.
\end{enumerate}
\end{theorem}

\begin{proof}

First we  prove that $(i)\implies (ii)$. Since
 $R_\alpha^{(m)}=0$ for any $\alpha\in \FF_n^+$ with $|\alpha|\geq
 m+1$, we have
 $$
P_\rho( rX, R^{(m)})=\sum_{1\leq |\alpha|\leq m} r^{|\alpha|}
X_\alpha^*\otimes R_{\widetilde\alpha}^{(m)} +\rho I\otimes I
+\sum_{1\leq |\alpha|\leq m} r^{|\alpha|} X_\alpha \otimes
{R_{\widetilde\alpha}^{(m)}}^*.
$$
 Since  $X\mapsto P_1(X,R)$ is a positive free
pluriharmonic
 function on $[B(\cH)^n]_1$, with coefficients in $B(F^2(H_n))$, so is
 the map
 $$ X\mapsto P_1(rX,R^{(m)})=(I\otimes P_{\cP_m})P_1(rX,R)|_{\cH\otimes
 \cP_m}
 $$
for any $r\in [0,1)$. If $A\overset{H}{{\underset{c}\prec}}\, B$,
then we have
$$
P_1(rA, R^{(m)})+(\rho-1)P_1(0,R^{(m)})\leq c^2\left[P_1(rB,
R^{(m)})+(\rho-1)P_1(0,R^{(m)})\right]
$$
for any $m=1,2,\ldots.$
   Using  Lemma \ref{L1}, we deduce that
$$
P_1(rA, R )+(\rho-1) I\leq c^2\left[P_1(rB, R )+(\rho-1) I\right]
$$
 for any $r\in [0,1)$.  Since $P_\rho(rY,R)=P_1(rY,R)+(\rho-1)I$ for
 any $n$-tuple $Y\in B(\cH)^n$ with spectral radius $r(Y)\leq 1$ and $r\in [0,1)$, we deduce
 item (ii).

  To prove the implication $(ii)\implies (iii)$, assume that
 condition (ii) holds and let $u$ be a positive free pluriharmonic
 function  on $[B(\cH)^n]_1$ with coefficients in $B(\cE)$ of the form
$$u(Z_1,\ldots, Z_n)=\sum_{k=1}^\infty \sum_{|\alpha|=k}
Z_\alpha^*\otimes C_{(\alpha)}^* +
    I\otimes C_{(0)}+\sum_{k=1}^\infty \sum_{|\alpha|=k}   Z_\alpha\otimes C_{(\alpha)}.
    $$
 It is well-known (see e.g. \cite{Pa-book}) that if $\cS\subseteq
B(F^2(H_n))$\ is an operator system and  $\mu: \cS\to B(\cK)$ is a
completely bounded map, then there exists a completely bounded
linear map
$$\widetilde
\mu:=\mu\otimes \text{\rm id} : \cS \bar\otimes_{min} B(\cH)\to
B(\cK)\bar\otimes_{min} B(\cH)
$$
such that $ \widetilde \mu(f\otimes Y):= \mu(f)\otimes Y$ for $f\in
\cS$ and  $Y\in B(\cH)$. Moreover, $\|\widetilde
\mu\|_{cb}=\|\mu\|_{cb}$ and, if $\mu$ is completely positive, then
so is $\widetilde \mu$.

  Using Corollary 5.5
 from \cite{Po-pluriharmonic}, we find  a completely positive
 linear map $\nu: \cR_n^* + \cR_n\to B(\cE)$ such that
 $\nu(R_{\tilde \alpha})=C_{(\alpha)}^*$ if $|\alpha|\geq 1$ and
 $\nu(I)=C_{(0)}$. Note that
 \begin{equation*}
 \begin{split}
 ( \text{\rm id}\otimes \nu)&[c^2P_\rho(rB,R)-P_\rho(rA,R)] \\
&=( \text{\rm id}\otimes \nu)\left\{\sum_{k=1}^\infty
\sum_{|\alpha|=k} r^{|\alpha|} (c^2B_\alpha-A_\alpha)\otimes
R_{\tilde \alpha}^*+ \rho(c^2-1)I\otimes I+\sum_{k=1}^\infty
\sum_{|\alpha|=k}(c^2B_\alpha^*-A_\alpha^*)\otimes R_{\tilde
\alpha}\right\}\\
&= \left\{\sum_{k=1}^\infty \sum_{|\alpha|=k} r^{|\alpha|}
(c^2B_\alpha-A_\alpha)\otimes C_{(\alpha)}+ \rho(c^2-1)I\otimes
C_{(0)}+\sum_{k=1}^\infty
\sum_{|\alpha|=k}(c^2B_\alpha^*-A_\alpha^*)\otimes C_{
(\alpha)}^*\right\}\\
&=c^2 \left[u(rB_1,\ldots, rB_n)+ (\rho-1)u(0)]-[u(rA_1,\ldots,
rA_n)+(\rho-1)u(0)\right].
 \end{split}
 \end{equation*}
  Hence,  and using the fact that
 $c^2P_\rho(rB, R)-P_\rho(rA, R)\geq 0$, we deduce that
 $$
  c^2 \left[u(rB_1,\ldots, rB_n)+ (\rho-1)u(0)]-[u(rA_1,\ldots,
rA_n)+(\rho-1)u(0)\right]\geq 0,
$$
which proves (iii).

Now, we prove the implication $(iii)\implies (v)$. Let $g\in
\left(\overline{\cA_n^*+ \cA_n}^{\|\cdot\|}\right)\otimes_{min}
M_{m}$ be positive. Then, according to  Theorem 4.1 from
\cite{Po-pluriharmonic},  the map defined by
  \begin{equation*}
  g(X):=(P_X\otimes \text{\rm id})[g],\qquad X\in [B(\cH)^n]_1,
  \end{equation*}
   is a positive  free
  pluriharmonic function. Condition (iii) implies
  $$
  g(rA_1,\ldots, rA_n)+(\rho-1)g(0)\leq c^2 \left[g(rB_1,\ldots,
rB_n)+ (\rho-1)g(0)\right]
$$ for   any $r\in
  [0,1)$.  Hence, and using   relation \eqref{var},  we get
$\rho \varphi_A(g_r)\leq c^2 \rho \varphi_B(g_r)$. Taking $r\to 1$,
we deduce item (v).

  To prove the implication
$(v)\implies (i)$, let $p\in \CC[X_1,\ldots, X_n]\otimes M_{m}$,
$m\in \NN$, be a noncommutative polynomial with matrix coefficients
such that $\text{\rm Re}\,p\geq 0$. Since
$$\rho \varphi_Y(p)=p(Y_1,\ldots, Y_n)+(\rho -1)p(0)$$ for any
$Y:=(Y_1,\ldots, Y_n)\in \cC_\rho$, it is clear that (v) implies
item (i).

We prove  now that $(ii)\implies (iv)$.
 We
recall that  $e_\alpha:= e_{i_1}\otimes\cdots \otimes e_{i_k}$  if
$\alpha=g_{i_1}\cdots g_{i_k}\in \FF_n^+$ and $e_{g_0}:=1$, and that
$\{e_\alpha\}_{\alpha\in \FF_n^+}$ is an orthonormal basis for the
full Fock space  $F^2(H_n)$.
 First,  we prove   that
  \begin{equation}\label{Ar}
  \left< P_\rho(X, rR)\left( \sum_{|\beta|\leq q}  h_\beta\otimes e_\beta \right),
   \sum_{|\gamma|\leq q}  h_\gamma\otimes e_\gamma\right>
   = \rho\sum_{|\beta|, |\gamma|\leq q}\left< K_{\rho, X,r}(\gamma,
   \beta)
    h_\beta, h_\gamma\right>,
  \end{equation}
  where
  the multi-Toeplitz kernel  $K_{\rho, X,r}:\FF_n^+\times \FF_n^+\to
  B(\cH)$, $r\in (0,1)$,
  is defined  by
   $$
   K_{\rho, X,r}(\alpha, \beta):=
   \begin{cases}
  \frac{1}{\rho}r^{|\beta\backslash_l \alpha|} X_{\beta\backslash_l \alpha}
   &\text{ if } \beta>_l\alpha\\
   I  &\text{ if } \alpha=\beta\\
    \frac{1}{\rho}r^{|\alpha\backslash_l \beta|}(X_{\alpha\backslash_l \beta})^*
     &\text{ if } \alpha>_l\beta\\
    0\quad &\text{ otherwise}.
   \end{cases}
   $$
   Note that   if $\{h_\beta\}_{|\beta|\leq q}\subset \cH$, then we
   have
    \begin{equation*}
    \begin{split}
    \left< \left(\rho I\otimes I+\sum_{k=1}^\infty \sum_{|\alpha|=k} X_{\alpha}^*\otimes r^k
     R_{\tilde\alpha}\right)
     \right.& \left.\left(\sum_{|\beta|\leq q}  h_\beta\otimes e_\beta \right),
   \sum_{|\gamma|\leq q} h_\gamma\otimes e_\gamma\right>\\
   &=
     \rho\sum_{|\beta|\leq q} \|h_\beta\|^2+
     \sum_{k=1}^\infty \sum_{|\alpha|=k}\left<\sum_{|\beta|\leq q}
      X_{\alpha}^*h_\beta\otimes r^k
     R_{\tilde\alpha} e_\beta,
    \sum_{|\gamma|\leq q}  h_\gamma\otimes e_\gamma\right> \\
    &=\rho\sum_{|\beta|\leq q} \|h_\beta\|^2++\sum_{|\alpha|\geq 1}\sum_{|\beta|, |\gamma|\leq q}
    r^{|\alpha|} \left< e_{\beta  {\alpha}}, e_\gamma\right>
    \left<X_{\alpha}^* h_\beta, h_\gamma\right>\\
    &=
    \rho\sum_{|\beta|\leq q} \|h_\beta\|^2+
    \sum_{ \gamma>\beta; ~|\beta|, |\gamma|\leq q}
    r^{|\gamma\backslash_l \beta|}
     \left<X_{\gamma\backslash_l \beta}^* h_\beta, h_\gamma\right>\\
     &=
     \sum_{\gamma\geq\beta; ~|\beta|, |\gamma|\leq
     q}\left<\rho K_{\rho, X,r}
      (\gamma, \beta)h_\beta, h_\gamma\right>.
    \end{split}
    \end{equation*}
   Now, taking into account that $K_{\rho,X,r}
      (\gamma, \beta)=K_{\rho, X,r}^*
      ( \beta, \gamma)$,  we deduce relation \eqref{Ar}.
   Therefore, the condition $P_\rho( rA, R)\leq c^2 P_\rho( rB,R)$, $r\in
   [0,1)$, implies
$$[K_{\rho, A,r}(\alpha,\beta)]_{|\alpha|, |\beta|\leq q}
  \leq c^2[K_{\rho, B,r}(\alpha,\beta)]_{|\alpha|, |\beta|\leq q} $$
  for any $0<r<1$
  and $q=0,1,\ldots$.
  Taking  $r\to 1$ in the latter inequality, we   obtain item (iv).

  Assume now that (iv) holds. Since $c^2K_{\rho,B}-K_{\rho,A}$ is a positive
  semidefinite multi-Toeplitz kernel, due to Theorem 3.1 from  \cite{Po-posi} (see also
  the proof of Theorem 5.2 from \cite{Po-pluriharmonic}), we find a
  completely positive linear map $\mu:C^*(S_1,\ldots, S_n)\to
  B(\cE)$ such that
  $$
  \mu(S_\alpha)=c^2K_{\rho, B}(g_0, \alpha)-K_{\rho,A}(g_0, \alpha)=\frac{1}{\rho}
  (c^2B_\alpha-A_\alpha)
  $$
  for any $\alpha \in \FF_n^+$ with $|\alpha|\geq 1$, and $\mu(I)=(c^2-1)I$.
   Since
   $$P(rS, R):=  \sum_{k=1}^\infty \sum_{|\alpha|=k}r^k
   S_{  \alpha}\otimes R_{\tilde\alpha}^* + I\otimes I+
  \sum_{k=1}^\infty \sum_{|\alpha|=k}r^k
  S_{\alpha}^* \otimes R_{\tilde \alpha}
  \geq 0 $$
  for $r\in
  [0,1)$, we deduce that
  \begin{equation*} \begin{split}
(\mu\otimes \text{\rm
 id})[P(rS, R)]&=\sum_{k=1}^\infty\sum_{|\alpha|=k}
 \frac{1}{\rho}r^{|\alpha|}[c^2 B_\alpha^*-A_\alpha^*]\otimes R_{\widetilde\alpha}
+(c^2-1) I\otimes I+\sum_{k=1}^\infty\sum_{|\alpha|=k}
 \frac{1}{\rho}r^{|\alpha|}[c^2 B_\alpha-A_\alpha]\otimes R_{\widetilde\alpha}^*\\
&=c^2P_\rho(rB, R)-P_\rho(rA, R)\geq 0,
  \end{split}
  \end{equation*}
which implies   (ii).

Let us prove that $(iv)\implies  (vi)$. Assume that (iv) holds. Then
we have $K_{\rho, A}\leq c^2 K_{\rho,B}$, where $K_{\rho,X}$ is the
multi-Toeplitz kernel associated with $X\in \cC_\rho$. Let
$V:=(V_1,\ldots, V_n)$ be the minimal isometric dilation of
$A:=(A_1,\ldots, A_n)$. Then $\cK_A=\bigvee_{\alpha\in \FF_n^+}
V_\alpha \cH$ and $\rho P_\cH V_\alpha|_{\cH}=A_\alpha$ for any
$|\alpha|\geq 1$. Similar properties hold if $W:=(W_1,\ldots, W_n)$
is  the minimal isometric dilation of  $B:=(B_1,\ldots, B_n)$.
Hence, and taking into account that $V_1,\ldots, V_n$ and
$W_1,\ldots, W_n$ are isometries with orthogonal ranges,
respectively,  we have
\begin{equation*}
\begin{split}
\left\|\sum_{|\alpha|\leq m} V_\alpha h_\alpha\right\|^2 &=
\sum_{\alpha>_l \beta, |\alpha|, |\beta|\leq
m}\left<V_{\alpha\backslash _l \beta} h_\alpha, h_\beta\right>
+\sum_{|\alpha|\leq m} \left<h_\alpha, h_\alpha\right> +
\sum_{\beta>_l \alpha, |\alpha|, |\beta|\leq
m}\left<V_{\beta\backslash _l \alpha}^* h_\alpha, h_\beta\right>
\\
&= \sum_{\alpha>_l \beta, |\alpha|, |\beta|\leq
m}\left<\frac{1}{\rho}A_{\alpha\backslash _l \beta} h_\alpha,
h_\beta\right> +\sum_{|\alpha|\leq m} \left<h_\alpha,
h_\alpha\right> + \sum_{\beta>_l \alpha, |\alpha|, |\beta|\leq
m}\left<\frac{1}{\rho}A_{\beta\backslash _l \alpha}^* h_\alpha,
h_\beta\right>
\\
&=\sum_{|\alpha|\leq m, |\beta|\leq m} \left<K_{\rho,A}(\beta,
\alpha) h_\alpha, h_\beta\right> =
\left<\left[K_{\rho,A}(\beta,\alpha)\right]_{|\alpha|, |\beta|\leq
m}{\bf h}_m, {\bf h}_m\right>
\end{split}
\end{equation*}
for any $m\in \NN$  and ${\bf h}_m:=\oplus_{|\alpha|\leq m} h_\alpha
\in \oplus_{|\alpha|\leq m} \cH_{\alpha}$, where each $\cH_{\alpha}$
is a copy of $\cH$. Similarly, we obtain
$$
\left\|\sum_{|\alpha|\leq m} W_\alpha h_\alpha\right\|^2  =
\left<\left[K_{\rho, B}(\beta, \alpha)\right]_{|\alpha|, |\beta|\leq
m}{\bf h}_m, {\bf h}_m\right>.
$$
Taking into account that  $K_{\rho,A}\leq c^2 K_{\rho, B}$, we
deduce that
$$
\left\|\sum_{|\alpha|\leq m} V_\alpha h_\alpha\right\|\leq c
\left\|\sum_{|\alpha|\leq m} W_\alpha h_\alpha\right\|.
$$
Therefore, we can define an  operator $L_{B,A}:\cK_B\to \cK_A$ by
setting
\begin{equation}
\label{LBA} L_{B,A}\left(\sum_{|\alpha|\leq m} W_\alpha
h_\alpha\right):=\sum_{|\alpha|\leq m} V_\alpha h_\alpha
\end{equation}
for any $m\in \NN$ and $h_\alpha\in \cH$, $\alpha\in \FF_n^+$. Note
that $L_{B,A}$ is a bounded  operator with $\|L_{B,A}\|\leq c$.
Since $L_{B,A}|_\cH=I_\cH$, we have $\|L_{B,A}\|\geq 1$. It is easy
to see that $L_{B,A} W_i=V_iL_{B,A}$ for $ i=1,\ldots,n$. Therefore
item (vi) holds.

Conversely, assume that there is an operator $L_{B,A}\in B(\cK_B,
\cK_A)$ with norm $\|L_{B,A}\|\leq c$  such that
$L_{B,A}|_\cH=I_\cH$ and $ L_{B,A} W_i=V_iL_{B,A}$, $ i=1,\ldots,n$.
Then, we deduce that $ L_{B,A}\left(\sum_{|\alpha|\leq m} W_\alpha
h_\alpha\right)=\sum_{|\alpha|\leq m} V_\alpha h_\alpha$ for any
$m\in \NN$ and $h_\alpha\in \cH$, $\alpha\in \FF_n^+$. The condition
$\|L_{B,A}\|\leq c$ implies
$$
\left\|\sum_{|\alpha|\leq m} V_\alpha h_\alpha\right\|^2 \leq
c^2\left\|\sum_{|\alpha|\leq m} W_\alpha h_\alpha\right\|^2,
$$
which is equivalent to the inequality
$$
\left<\left[K_{\rho, A}(\beta, \alpha)\right]_{|\alpha|, |\beta|\leq
m}{\bf h}_m, {\bf h}_m\right> \leq c^2 \left<\left[K_{\rho,
B}(\beta, \alpha)\right]_{|\alpha|, |\beta|\leq m}{\bf h}_m, {\bf
h}_m\right>
$$
for any $m\in \NN$  and ${\bf h}_m:=\oplus_{|\alpha|\leq m}
h_\alpha\in \oplus_{|\alpha| \leq m} \cH_{\alpha }$. Consequently,
we deduce item (iv). The proof is complete.
\end{proof}

A closer look at the proof of Theorem \ref{equivalent} reveals that
one can assume that $u(0)=I$ in part (iii), and one can also assume
that $\Re p(0)=I$ in the definition of the Harnack domination
$A\overset{H}{{\prec}}\, B$. We also remark  that, due to Theorem
\ref{funct-calc}, we can add an equivalence to Theorem
\ref{equivalent}, namely, $A\overset{H}{{\underset{c}\prec}}\,
 B$ if and only if
$$u(A_1,\ldots, A_n)+(\rho-1)u(0)\leq c^2 \left[u(B_1,\ldots,
B_n)+ (\rho-1)u(0)\right] $$
 for any positive free pluriharmonic
function $u\in  Har_{\bf ball}^c(B(\cE))$.

\begin{corollary}\label{LBA-inf} If  $A, B\in \cC_\rho$ and
 $A\overset{H}{{\prec}}\, B$, then
 \begin{equation*}
 \begin{split}
 \|L_{B,A}\|&=\inf\{c> 1:\ A\overset{H}{{\underset{c}\prec}}\,
 B\}\\
 &=\inf\{c> 1:\  P_\rho( rA, R)\leq c^2 P_\rho(rB, R)\quad \text{ for any }
\quad r\in [0,1)\}.
 \end{split}
 \end{equation*}
 Moreover,
 $A\overset{H}{{\prec}}\, B$ if and only if \
    $\sup_{r\in [0, 1)} \|L_{rA,rB}\|<\infty$.
In this case,
$$\|L_{A,B}\|=\sup_{r\in [0, 1)} \|L_{rA,rB}\|
$$
and  the mapping $r\mapsto  \|L_{rA,rB}\|$ is increasing on $[0,1)$.
\end{corollary}

\begin{proof} Assume that $A\overset{H}{{\prec}}\, B$.  Then, due to
Theorem \ref{equivalent}, $A\overset{H}{{\underset{c}\prec}}\, B$ if
and only if there is an operator $L_{B,A}\in B(\cK_B, \cK_A)$ with
$\|L_{B,A}\|\leq c$ such that $L_{B,A}|_\cH=I_\cH$ and
$
 L_{B,A} W_i=V_iL_{B,A}$ for  $ i=1,\ldots,n.
$ Consequently, taking $c=\|L_{B,A}\|$, we deduce that
$A\overset{H}{{\underset{\|L_{B,A}\|}\prec}}\, B$, which  is
equivalent to
$$
P_\rho(rA,R)\leq \|L_{B,A}\|^2 P_\rho(rB,R)
$$
for any $r\in [0,1)$. Hence, we have
$tA\overset{H}{{\underset{\|L_{B,A}\|}\prec}}\, tB$ for any $t\in
[0,1)$. Applying  again Theorem \ref{equivalent} to the operators
$tA$ and $tB$, we deduce that $\|L_{tA,tB}\|\leq \|L_{B,A}\|$.

Conversely, suppose that $c:=\sup_{r\in [0, 1)}
\|L_{rA,rB}\|<\infty$. Since $\|L_{rA,rB}\|\leq c$, Theorem
\ref{equivalent} implies $rA\overset{H}{{\underset{c}\prec}}\, rB$
for any $r\in [0,1)$ and, therefore,  $P_\rho(rtA,R)\leq c^2
P_\rho(rtB,R)$ for any $t,r\in [0,1)$. Hence,
$A\overset{H}{{\underset{c}\prec}}\, B$ and, consequently,
$\|L_{B,A}\|\leq c$. Therefore, $\|L_{A,B}\|=\sup_{r\in [0, 1)}
\|L_{rA,rB}\|$. The fact that $r\mapsto  \|L_{rA,rB}\|$ is an
increasing function on $[0,1)$ follows  from the latter relation.
This completes the proof.
\end{proof}

We remark that if $1\leq m<n$ and $u$ is a positive free
pluriharmonic function on $[B(\cK)^n]_1$,  then the map
$$
(X_1,\ldots, X_m)\mapsto u(X_1,\ldots, X_m, 0,\ldots, 0)
$$
is a positive free pluriharmonic function on $[B(\cK)^m]_1$.
Moreover, if $g$ is a positive free pluriharmonic function on
$[B(\cK)^m]_1$, then the map
$$
(X_1,\ldots, X_n)\mapsto g(X_1,\ldots, X_m, 0,\ldots, 0)
$$
is a positive free pluriharmonic function on $[B(\cK)^n]_1$.
Consequently, using Corollary \ref{T,0},  one can easily deduce the
following result.

\begin{corollary} \label{T,02}
Let $c>0$, $\rho>0$,  and $1\leq m<n$. Consider two  $n$-tuples $
(A_1,\ldots, A_m)\in B(\cH)^m$ and $ (B_1,\ldots, B_m)\in B(\cH)^m$
 in  the class $ \cC_\rho$ and  let $(A_1,\ldots, A_m,0,\ldots, 0)$
and $(B_1,\ldots, B_m,0,\ldots, 0)$   be their extensions in
$B(\cH)^n$, respectively.
   Then $(A_1,\ldots, A_m)\overset{H}{{\underset{c}\prec}}\, (B_1,\ldots,
   B_m)$ in $\cC_\rho\subset B(\cH)^m$  if and only if
   $$(A_1,\ldots,
A_m,0,\ldots, 0)\overset{H}{{\underset{c}\prec}}\,(B_1,\ldots,
B_m,0,\ldots, 0)\quad \text{ in }\ \cC_\rho\subset B(\cH)^n.
$$
\end{corollary}

 We recall  (e.g.
\cite{Po-similarity})  that  if $(T_1,\ldots T_n)$ is an $n$-tuple
of operators, then the joint spectral radius $r(T_1,\ldots, T_n)<1$
if and only if $\lim_{k\to\infty} \left\|\sum_{|\alpha|=k} T_\alpha
T_\alpha^*\right\|=0$.

In  what follows, we characterize  the elements of  $\cC_\rho$ which
are Harnack dominated  by $0$.

\begin{theorem}
\label{A<0} Let $A:=(A_1,\ldots, A_n)$   be in $\cC_\rho$. Then
$A\overset{H}{{ \prec}}\, 0$ if and only if the joint spectral
radius  $r(A_1,\ldots, A_n)<1$.
\end{theorem}
\begin{proof}
Note that the map $X\mapsto P_\rho(X,R)$ is a positive  free
pluriharmonic function on $ [B(\cH)^n]_1$ with coefficients in
 $B(F^2(H_n))$ and has the factorization
\begin{equation}\label{facto}
\begin{split}
P_\rho(X,R)&= (I-R_X)^{-1}+(\rho - 2)I+ (I-R_X^*)^{-1}\\
&=(I-R_X^*)^{-1}\left[ I-R_X+(\rho -2)(I-R_X^*)(I-R_X)+ I-R_X^*
\right](I-R_X)^{-1}\\
&= (I-R_X^*)^{-1}\left[  \rho I+(1-\rho)(R_X^*+ R_X)+ (\rho-2)R_X^*
R_X \right](I-R_X)^{-1},
\end{split}
\end{equation}
where $R_X:=X_1^*\otimes R_1+\cdots +X_n^*\otimes R_n$ is the
reconstruction operator associated with  the $n$-tuple
$X:=(X_1,\ldots, X_n)\in  [B(\cH)^n]_1$. We remark that,  due to the
fact that the spectral radius of $R_X$ is equal
 to the joint spectral radius $r(X_1,\ldots, X_n)$,
the factorization above  holds  for any $X\in \cC_\rho$ with
$r(X_1,\ldots, X_n)<1$.

Now, using
 Theorem \ref{equivalent} part (ii) and the  above-mentioned factorization,
 we deduce that $A\overset{H}{{ \prec}}\, 0$
if and only if  there exists $c>0$ such that
$$
(I-R_{rA}^*)^{-1}\left[  \rho I+(1-\rho)(R_{rA}^*+ R_{rA})+
(\rho-2)R_{rA}^* R_{rA} \right](I-R_{rA})^{-1}\leq \rho c^2I
$$
for any $r\in [0,1)$. Similar inequality holds if we replace the
right creation operators by the left creation operators. Then,
applying the noncommutative Poisson transform $ \text{\rm id}\otimes
P_{e^{i\theta}R }$, where $R:=(R_1,\ldots, R_n)$,   we obtain
 \begin{equation}\label{rho-ine}
 \rho I+(1-\rho)(e^{-i\theta}R_{rA}^*+ e^{i\theta}R_{rA})+
(\rho-2)R_{rA}^* R_{rA} \leq \rho
c^2(I-re^{-i\theta}R_A^*)(I-re^{i\theta}R_A)
\end{equation}
for any $r\in [0,1)$ and $\theta\in  \RR$.

On the other hand, since  $A:=(A_1,\ldots, A_n)\in \cC_\rho$, we
have $r(A_1,\ldots, A_n)\leq 1$. Suppose that $r(A_1,\ldots, A_n)=
1$. Taking into account that  $r(R_A)=r(A_1,\ldots, A_n)$, we can
find $\lambda_0\in \TT$ in the approximative spectrum of $R_A$.
Consequently, there is a sequence $\{h_m\}$ in $\cH\otimes F^2(H_n)$
such that $\|h_m\|=1$ and
\begin{equation} \label{assu} \lambda_0h_m-R_Ah_m \to 0\quad \text{
as } m\to \infty.
\end{equation}
In particular, relation \eqref{rho-ine} implies
\begin{equation}
\label{rho-ine2} \begin{split} \rho \|h_m\|^2&+ (1-\rho)
\left[\left<\lambda_0 R_{rA}^* h_m, h_m\right>+\left<\bar \lambda_0
R_{rA} h_m, h_m\right>\right]+(\rho-2)\|R_{rA} h_m\|^2\\
&\leq \rho c^2\|h_m-\bar \lambda_0 R_{rA}h_m\|^2
\end{split}
\end{equation}
for any $r\in (0,1)$ and $m\in \NN$. Note that due to \eqref{assu}
and the fact that $|\lambda_0|=1$, we have
$$
\left<\bar \lambda_0 R_{A} h_m, h_m\right>=\bar
\lambda_0\left<R_Ah_m-\lambda_0 h_m, h_m\right>+1\to 1,\quad \text{
as }\
 m\to \infty.
 $$
Since
\begin{equation*}
\begin{split}
\|h_m-\bar\lambda_0R_{rA} h_m\|&\leq \|h_m-\bar\lambda_0R_{A} h_m\|+
\|\bar \lambda_0(R_Ah_m-R_{rA}h_m)\|\\
&=\|\bar\lambda_0 h_m-R_A h_m\|+(1-r)\|R_A h_m\|
\end{split}
\end{equation*}
and due to the fact that $\|R_A h_m\|\to 1$ as $m\to\infty$, we
deduce that
$$
\limsup_{m\to\infty}\|h_m-\bar \lambda_0 R_{rA}h_m\|\leq 1-r
$$
for any $r\in (0,1)$. Now, since $R_{rA}=rR_A$ and taking
$m\to\infty$ in  relation \eqref{rho-ine2}, we obtain
$$
\rho+ 2(1-\rho)r+(\rho-2)r^2\leq c^2\rho(1-r)^2
$$
for any $r\in (0,1)$. Setting $r=1-\frac{1}{m}$, $m\geq 2$,
straightforward calculations imply $ 2m\leq \rho c^2-\rho+2$ for any
$m\in \NN$, which is a contradiction. Therefore, we  must have
$r(A_1,\ldots, A_n)< 1$.

Conversely, assume that  $A:=(A_1,\ldots, A_n)\in \cC_\rho$ has the
joint spectral radius  $r(A_1,\ldots, A_n)< 1$.  Since
$r(A_1,\ldots, A_n)=r(R_A)$, one can see that $ M:=\sup_{r\in (0,1)}
\|(I-rR_A)^{-1}\| $ exists and $M\geq 1$. Hence
 \begin{equation}
 \label{MRR}
M^2(I-R_{rA}^*)(I-R_{rA})\geq I\geq I-R_{rA}^*R_{rA}
 \end{equation}
for any $r\in (0,1)$. Now we consider the case $\rho\geq 1$. Note
that  relation \eqref{MRR} implies
$$
I-R_{rA}^*R_{rA} +(\rho-1)(I-R_{rA}^*)(I-R_{rA})\leq \rho
M^2(I-R_{rA}^*)(I-R_{rA}).
$$
The latter inequality is equivalent  to
$$
\rho I +(1-\rho)(R_{rA}^*+ R_{rA})+(\rho-2)R_{rA}^*R_{rA}\leq \rho
M^2 (I-R_{rA}^*)(I-R_{rA}),
$$
which, due to the factorization  \eqref{facto},  is equivalent to
$$
P_\rho(rA,R)\leq \rho M^2 =M^2 P_\rho (0,R)
$$
for any $r\in [0,1)$.  According to Theorem \ref{equivalent}, we
deduce that $A\overset{H}{{ \prec}}\, 0$.

Now, consider the case  when $\rho\in (0,1)$. Since $\|R_{rA}\|\leq
r\rho$ and $\delta-2<0$, we have
\begin{equation*}
\begin{split}
\rho I+(1-\rho)(R_{rA}^*+ R_{rA})+(\rho-2)R_{rA}^*R_{rA}&\leq \rho
I+(1-\rho)(R_{rA}^*+ R_{rA})\\
&\leq \rho I+2(1-\rho)r\rho\leq (3\rho-2\rho^2)I.
\end{split}
\end{equation*}
Using  again the factorization \eqref{facto}, we deduce that
$$
P_\rho(rA,R)\leq (3\rho-2\rho^2)(I-R_{rA}^*)^{-1}(I-R_{rA})^{-1}
$$
for any $r\in (0,1)$. Hence and using the fact that
$(I-R_{rA}^*)^{-1}(I-R_{rA})^{-1}\leq M^2 I$, we obtain
$$
P_\rho(rA,R)\leq (3-2\rho) M^2 P_\rho(0,R)
$$
for any $r\in (0,1)$. Using again Theorem \ref{equivalent}, we get
 $A\overset{H}{{ \prec}}\, 0$.
The proof is complete.
\end{proof}

We mention that in the particular case  when $n=1$ we can recover a
result obtained by Ando, Suciu, and Timotin   \cite{AST}, when
$\rho=1$, and by G. Cassier and N. Suciu \cite{CaSu}, when $\rho\neq
1$.

\bigskip

\section{ Hyperbolic metric  on  Harnack parts of the  noncommutative ball $\cC_\rho$}

The relation $\overset{H}{\prec} $ induces an equivalence relation
$\overset{H}\sim$ on  the class $\cC_\rho$. We  provide a Harnack
type double inequality for positive free pluriharmonic functions on
the noncommutative ball $\cC_\rho$ and use it to prove that the
Harnack part of $\cC_\rho$ which contains $0$ coincides  with the
open noncommutative ball $ [\cC_\rho]_{<1}.$  We introduce  a
hyperbolic metric
  on any Harnack part
  of $\cC_\rho$ and
 obtain a concrete formula in terms of
 the reconstruction operator.

   Since $\overset{H}{\prec} $ is a preorder
relation on $\cC_\rho$, it induces an equivalence relation
$\overset{H}\sim$ on $\cC_\rho$, which we call Harnack equivalence.
The equivalence classes with respect to $\overset{H}\sim$ are called
Harnack parts of $\cC_\rho$. Let $A:=(A_1,\ldots, A_n)$ and
$B:=(B_1,\ldots, B_n)$ be in $\cC_\rho$. We say  that $A$ and $B$
are Harnack equivalent (we denote $A\overset{H}{\sim}\, B$) if and
only if there exists $c\geq  1$ such that
\begin{equation*}
\begin{split}
 \frac{1}{c^2}\left[\Re p(B_1,\ldots, B_n)+(\rho-1)\Re p(0)
 \right]&\leq
\Re p(A_1,\ldots, A_n)+(\rho-1)\Re p(0)\\
&\leq c^2\left[ \Re p(B_1,\ldots, B_n)+(\rho-1)\Re p(0)\right]
\end{split}
\end{equation*}
for any noncommutative polynomial with matrix-valued coefficients
$p\in \CC[X_1,\ldots, X_n]\otimes M_{m}$, $m\in \NN$, such that $\Re
p(X)\geq 0$ for any $X\in [B(\cH)^n]_1$. We also use the notation
$A\overset{H}{{\underset{c}\sim}}\, B$  when
$A\overset{H}{{\underset{c}\prec}}\, B$ and
$B\overset{H}{{\underset{c}\prec}}\, A$. We remark that Theorem
\ref{equivalent} can be used to  provide several characterizations
for the Harnack parts of $\cC_\rho$.

The  first result is an extension of   Harnack   inequality to
positive free pluriharmonic functions on
 the noncommutative ball   $\cC_\rho$, $\rho>0$.

\begin{theorem}
\label{Harnack} If \,$u$ is a  positive free pluriharmonic function
on $[B(\cH)^n]_1$  with operator-valued  coefficients in $B(\cE)$
and $0\leq r<1$, then
$$
  u(0) \,\frac{1-r(2\rho-1)}{1+r}\leq u(rX_1,\ldots, rX_n) \leq  u(0)
   \,\frac{1+r(2\rho-1)}{1-r}
$$
  for any  $ (X_1,\ldots, X_n)\in \cC_\rho$.
\end{theorem}
\begin{proof}
Let $$u(Z_1,\ldots, Z_n)=\sum_{k=1}^\infty \sum_{|\alpha|=k}
Z_\alpha^*\otimes A_{(\alpha)}^* +
    I\otimes A_{(0)}+\sum_{k=1}^\infty \sum_{|\alpha|=k}   Z_\alpha\otimes A_{(\alpha)}$$
  be  a positive  free  pluriharmonic  function on  $[B(\cH)^n]_1$ with
  coefficients in $B(\cE)$. According to  Theorem 1.4 from \cite{Po-hyperbolic}, for
  any $Y\in [B(\cH)^n]_1^-$ and $r\in [0,1)$, we have
  \begin{equation}
  \label{u-ine}
  u(0) \,\frac{1-r}{1+r}\leq u(rY_1,\ldots, rY_n) \leq  u(0)
   \,\frac{1+r}{1-r}.
\end{equation}
On the other hand, let $ (X_1,\ldots, X_n)\in \cC_\rho$ and let
$(V_1,\ldots, V_n)$ be the minimal  isometric dilation  of
$(X_1,\ldots, X_n)$ on  a  Hilbert space $\cK_T\supseteq \cH$. Since
$X_\alpha =\rho P_\cH V_\alpha |_{\cH}$  for any  $
    \alpha\in \FF_n^+\backslash \{g_0\}$, and using the free
    pluriharmonic functional calculus, we have

    \begin{equation*}
    \begin{split}
&u(rX_1,\ldots, rX_n)=\sum_{k=1}^\infty \sum_{|\alpha|=k}
r^{|\alpha|}X_\alpha^*\otimes A_{(\alpha)}^* +
    I\otimes A_{(0)}+\sum_{k=1}^\infty \sum_{|\alpha|=k}   r^{|\alpha|}X_\alpha\otimes A_{(\alpha)}
    \\
    &=\rho (P_\cH\otimes I_\cE) \left[\sum_{k=1}^\infty \sum_{|\alpha|=k}
r^{|\alpha|}V_\alpha^*\otimes A_{(\alpha)}^* \right]|_{\cH\otimes
\cE}+ I_\cH\otimes A_{(0)}+\rho (P_\cH\otimes I_\cE)
\left[\sum_{k=1}^\infty \sum_{|\alpha|=k}
     r^{|\alpha|}V_\alpha\otimes A_{(\alpha)}\right]|_{\cH\otimes\cE}\\
&=\rho (P_\cH\otimes I_\cE) u(rV_1,\ldots, rV_n)|_{\cH\otimes\cE}+
(1-\rho)u(0),
\end{split}
    \end{equation*}
where the convergence is  in the operator norm topology. Due to
\eqref{u-ine}, we have
$$
  u(0) \,\frac{1-r}{1+r}\leq u(rV_1,\ldots, V_n) \leq  u(0)
   \,\frac{1+r}{1-r}.
$$
 Consequently, we deduce that
$$
  u(0) \left[\frac {\rho(1-r)}{1+r}+(1-\rho)\right]\leq
   \rho (P_\cH\otimes I_\cE) u(rV_1,\ldots, rV_n)|_{\cH\otimes\cE}+
(1-\rho)u(0) \leq  u(0) \left[\frac{\rho(1+r)}{1-r}
+(1-\rho)\right].
$$
Since
$$u(rX_1,\ldots, rX_n)=\rho (P_\cH\otimes I_\cE) u(rV_1,\ldots,
rV_n)|_{\cH\otimes\cE}+ (1-\rho)u(0),
$$
the result follows.
\end{proof}

Now, we can determine the Harnack part of $\cC_\rho$ which contains
$0$.

\begin{theorem}
\label{foias2} Let $A:=(A_1,\ldots, A_n)$   be in $\cC_\rho$. Then
the following statements are equivalent:
\begin{enumerate}
\item[(i)] $\omega_\rho(A_1,\ldots, A_n)<1$;
\item[(ii)]
$A\overset{H}{\sim}\, 0$;
\item[(iii)]
 $r(A_1,\ldots, A_n)<1$ and $P_\rho(A, R)\geq aI$ for some constant
 $a>0$.
\end{enumerate}
\end{theorem}

\begin{proof}
First, we prove that $(i)\implies (ii)$.  Let $A:=(A_1,\ldots, A_n)$
be in $\cC_\rho$ and assume that $\omega_\rho(A)<1$. Then there is
$r_0\in (0,1)$ such that $\omega_\rho(\frac{1}{r_0}A)=\frac{1}{r_0}
\omega_\rho (A)<1$. Consequently, $\frac{1}{r_0}A\in \cC_\rho$.

 According to
Theorem \ref{Harnack}, we have
 $$
  \Re p(0) \,\frac{1-r_0(2\rho-1)}{1+r_0}\leq \Re p(A_1,\ldots, A_n) \leq  \Re p(0)
   \,\frac{1+r_0(2\rho-1)}{1-r_0}
$$
 for any noncommutative polynomial with
matrix-valued coefficients $p\in \CC[X_1,\ldots, X_n]\otimes M_{m}$,
$m\in \NN$, such that $\Re p\geq 0$ on $[B(\cH)^n]_1$. Hence, we
deduce that $A\overset{H}{\sim}\, 0$.

To prove that $(ii)\implies (iii)$, assume that
$A\overset{H}{\sim}\, 0$. Due to Theorem \ref{A<0}, we have
$r(A)<1$. Using now Theorem \ref{equivalent}, we deduce that there
exists $c>0$ such that
\begin{equation}
\label{Pro} P_\rho(rA,R)\geq \frac{1}{c^2}
P_\rho(0,R)=\frac{\rho}{c^2} I
\end{equation}
for any $r\in [0,1)$. Since $r(A)<1$, one can prove that $\lim_{r\to
1}P_\rho(rA,R)=P_\rho(A,R)$ in the operator norm topology.
Consequently, taking $r\to 1$ in  relation \eqref{Pro}, we obtain
 item (iii).

  It remains to show that $(iii)\implies (i)$.
Assume that  $r(A_1,\ldots, A_n)<1$ and $P_\rho(A, R)\geq aI$ for
some constant
 $a>0$. Note that there exists $t_0\in (0, 1)$ such that the map
 $$
 t\mapsto \left(I-\sum_{i=1}^n A_i^* \otimes tR_i\right)^{-1}+ (\rho-2)I+
 \left(I-\sum_{i=1}^n A_i \otimes tR_i^*\right)^{-1}
 $$
 is well-defined and continuous on $[0,1+t_0]$ in the operator norm
 topology. In particular,  there is $\epsilon_0\in (0, t_0)$ such
 that
 $$
 \|P_\rho(A,R)-P_\rho(A,tR)\|<\frac{a}{2}
 $$
 for any $t\in(1-\epsilon_0, 1+\epsilon_0)$.
 Consequently, if $\gamma_0\in (1,1+\epsilon_0)$, then
 $$
 P_\rho(\gamma_0 A,R)\geq
 P_\rho(A,R)-\|P_\rho(A,R)-P_\rho(\gamma_0A,R)\|I\geq \frac{a}{2}
 I>0.
 $$
Due to   Theorem \ref{ro-contr},  we have $\gamma_0 A\in
  \cC_\rho$, which implies $\omega(\gamma_0 A)\leq 1$. Therefore,
  $\omega(A)\leq \frac{1}{\gamma_0}<1$ and
     item (i)
holds. The proof is complete.
\end{proof}

We remark that, when $n=1$,  we recover a  result
  obtain by Foia\c s \cite{Fo} if $\rho=1$,  and by Cassier and Suciu  \cite{CaSu} if
  $\rho>0$.

Given $A,B\in \cC_\rho$, $\rho>0$,  in the same Harnack part of
$\cC_\rho$, i.e., $A\,\overset{H}{\sim}\, B$, we introduce
\begin{equation}
\label{La} \Lambda_\rho(A,B):=\inf\left\{ c > 1: \
A\,\overset{H}{{\underset{c}\sim}}\, B   \right\}.
\end{equation}
Note that, due to Theorem \ref{equivalent},
  $A\overset{H}{\sim}\, B$ if and only
if the operator $L_{B,A}$ is invertible. In this case,
$L_{B,A}^{-1}=L_{A,B}$ and
$$
\Lambda_\rho(A,B)= \max\left\{ \|L_{A,B}\|, \|L_{B,A}\|\right\}.
$$
To prove the latter equality,
  assume that
$A\overset{H}{{\underset{c}\sim}}\, B $ for some $c\geq 1$. Due to
the same theorem, we have $\|L_{B,A}\|\leq c$ and $\|L_{A,B}\|\leq
c$. Consequently,
\begin{equation}\label{Max}
\max\left\{ \|L_{A,B}\|, \|L_{B,A}\|\right\}\leq \inf\left\{ c\geq
1: \ A\overset{H}{{\underset{c}\sim}}\, B
\right\}=\Lambda_\rho(A,B).
\end{equation}
On the other hand, setting  $c_0:=\|L_{B,A}\|$ and
$c_0':=\|L_{A,B}\|$, Theorem \ref{equivalent} implies
$A\overset{H}{{\underset{c_0}\prec}}\, B$ and
$B\overset{H}{{\underset{c_0'}\prec}}\, A$. Hence, we deduce that
$A\overset{H}{{\underset{d}\sim}}\, B $, where
$d:=\max\{c_0,c_0'\}$. Consequently, $\Lambda_\rho(A,B)\leq d$,
which together with relation \eqref{Max} imply $\Lambda_\rho(A,B)=
\max\left\{ \|L_{A,B}\|, \|L_{B,A}\|\right\}$, which proves our
assertion.

Now, we can introduce  a  hyperbolic ({\it Poincar\'e-Bergman} type)
metric $\delta_\rho:\Delta\times \Delta \to \RR^+$ on any Harnack
part $\Delta$ of $\cC_\rho$, by setting
\begin{equation}
\label{hyperbolic}
 \delta_\rho(A,B):=\ln \Lambda_\rho(A,B),\qquad A,B\in \Delta.
\end{equation}
Due to our discussion above, we also have
 \begin{equation*}
\begin{split}
\delta_\rho(A,B)&=
  \ln \max \left\{ \left\| L_{A,B} \right\|,
  \left\| L_{A,B}^{-1}\right\|\right\}.
\end{split}
\end{equation*}

\begin{proposition}\label{delta}
  $\delta_\rho$ is a metric on  any Harnack
part of $\cC_\rho$.
\end{proposition}
\begin{proof} The proof is similar to that of  Proposition 2.2 from
\cite{Po-hyperbolic}, but uses  $\rho$-pluriharmonic kernels.
\end{proof}

We remark that, according to Theorem \ref{foias2}, the set
$$[\cC_\rho]_{<1}:=\{(X_1,\ldots, X_n)\in B(\cH)^n:\
\omega_\rho(X_1,\ldots, X_n)<1\}$$
 is the Harnack part of $\cC_\rho$
containing $0$.

In what follows we calculate the norm of $L_{Y,X}$ with  $X,Y\in
[\cC_\rho]_{<1}$,  in terms of the reconstruction operators.

\begin{theorem}\label{LCC}
If $X,Y\in [\cC_\rho]_{<1}$, then $
\|L_{Y,X}\|=\|C_{\rho,X}C_{\rho,Y}^{-1}\|, $ where
 \begin{equation*} \begin{split} C_{\rho,X}&:=
\Delta_{\rho,X} (I-R_X)^{-1},\\
\Delta_{\rho,X}&:=  \left[\rho I+(1-\rho)(R_{X}^*+ R_{X})+
(\rho-2)R_{X}^* R_{X}\right]^{1/2}.
\end{split}
\end{equation*}
Moreover, if $X,Y\in \cC_\rho$ is such that $X\overset{H}{\prec}\,
Y$, then \ $
\|L_{Y,X}\|=\sup_{r\in[0,1)}\|C_{\rho,rX}C_{\rho,rY}^{-1}\|. $
\end{theorem}

\begin{proof} Since $X,Y\in [\cC_\rho]_{<1}$,  Theorem \ref{foias2}
implies  $X\overset{H}{\sim}\, Y$, $r(X)<1$, and $r(Y)<1$.
   Let  $c> 1$ and  assume that
$P_\rho(rX,R)\leq c^2 P_\rho(rY,R)$ for any $r\in[0,1)$. Since
$r(X)<1$ and $r(Y)<1$, we can take the limit, as $r\to 1$, in the
operator norm topology, and obtain
 $ P_\rho(X,R)\leq c^2 P_\rho(Y,R)$. Conversely, if the latter inequality
 holds, then $P_\rho(X,S)\leq c^2 P_\rho(Y,S)$,
where $S:=(S_1,\ldots, S_n)$ is  the $n$-tuple of left creation
operators.  Applying the noncommutative Poisson transform $\text{\rm
id}\otimes P_{rR }$, $r\in [0,1)$, and taking into account that it
is a positive map, we deduce that $  P_\rho(rX,R)\leq c^2
P_\rho(rY,R) $ for any $r\in [0,1)$.

 Therefore, due to   Theorem
\ref{equivalent},  we have
\begin{equation}
\label{PPL}
 P_\rho(X,R)\leq c^2 P_\rho(Y,R) \quad \text{ if and only if }
\quad \|L_{Y,X}\|\leq c.
\end{equation}
 We
recall that the free pluriharmonic  kernel $P_\rho(X,R)$ with $X\in
[\cC_\rho]_{<1}$, has the factorization $P(X, R)=C_{\rho,X}^*
C_{\rho,X}$. Due to Theorem \ref{foias2}, $P_\rho(X,R)$ is
invertible and, consequently, so is
 $C_{\rho,X}$.     Consequently,
$$P_\rho(X,R)\leq c^2 P_\rho(Y,R)\quad \text{  if and only if }\quad
{C_{\rho,Y}^*}^{-1} C_{\rho,X}^*C_{\rho,X}C_{\rho,Y}^{-1}\leq c^2I.
$$
Setting $c_0:=\|C_{\rho,X}C_{\rho,Y}^{-1}\|$, we have
$P_\rho(X,R)\leq c_0^2 P_\rho(Y,R)$. Now, due to relation
\eqref{PPL}, we obtain
$$\|L_{Y,X}\|\leq c_0=\|C_{\rho,X}C_{\rho,Y}^{-1}\|.
$$
Setting  $c_0':=\|L_{Y,X}\|$ and using again \eqref{PPL}, we obtain
$P_\rho(X,R)\leq {c_0'}^2 P_\rho(Y,R)$. Hence, we deduce that
${C_{\rho,Y}^*}^{-1} C_{\rho,X}^*C_{\rho,X}C_{\rho,Y}^{-1}\leq
{c_0'}^2I$, which implies
$$\|C_{\rho,X}C_{\rho,Y}^{-1}\|\leq
c_0'=\|L_{Y,X}\|.
$$
Therefore, $\|L_{Y,X}\|=\|C_{\rho,X}C_{\rho,Y}^{-1}\|$. The last
part of the  theorem  is now obvious.
\end{proof}

Combining Theorem \ref{LCC} with the remarks preceding Proposition
\ref{delta},  we obtain a concrete formula for the hyperbolic metric
$\delta_\rho$ on $ [\cC_\rho]_{<1}$ in terms of the reconstruction
operator, which is the main result of this section.

\begin{theorem}\label{P-B} Let $\delta_\rho :[\cC_\rho]_{<1}\times [\cC_\rho]_{<1}\to [0,\infty)$
 be the hyperbolic metric.
  If    $X,Y\in [\cC_\rho]_{<1}$, then
\begin{equation*}
\delta_\rho(X,Y)=\ln \max \left\{ \left\|C_{\rho,X} C_{\rho,Y}^{-1}
\right\|,
  \left\|C_{\rho,Y} C_{\rho,X}^{-1} \right\|\right\},
\end{equation*}
where \begin{equation*} \begin{split} C_{\rho,X}&:=
\Delta_{\rho,X}(I-R_X)^{-1},\\
\Delta_{\rho,X}&:=  \left[\rho I+(1-\rho)(R_{X}^*+ R_{X})+
(\rho-2)R_{X}^* R_{X}\right]^{1/2},
\end{split}
\end{equation*}
 and $R_X:=X_1^*\otimes R_1+\cdots + X_n^*\otimes R_n$
 is the reconstruction operator associated with the right
creation operators $R_1,\ldots, R_n$ and $X:=(X_1,\ldots, X_n)\in
[\cC_\rho]_{<1}$.
\end{theorem}

Using Theorem \ref{equivalent}, one can easily obtain the following
result. Since the proof is similar to that of Lemma 2.6  from
\cite{Po-hyperbolic}, we shall omit it.

\begin{lemma}
\label{OMr}
  Let   $X:=(X_1,\ldots, X_n)$ and $Y:=(Y_1,\ldots,
Y_n)$ be   in $\cC_\rho$. Then the following properties hold.
\begin{enumerate}
\item[(i)] $X\overset{H}{\sim}\, Y$ if and only if $rX\overset{H}{\sim}\,
rX$ for any $r\in [0,1)$ and  \  $\sup_{r\in [0,
1)}\Lambda_\rho(rX,rY)<\infty$. In this case,
$$
\Lambda_\rho(X,Y)=\sup_{r\in [0, 1)}\Lambda_\rho(rX,rY)\quad \text{
and } \quad \delta_\rho(X,Y)=\sup_{r\in [0, 1)}\delta_\rho(rX,rY).
$$
\item[(ii)]  If  $X\overset{H}{\sim}\, Y$, then the functions
 $r\mapsto \Lambda_\rho(rX,rY)$ and $r\mapsto
\delta_\rho(rX,rY)$ are increasing on $[0,1)$.
\end{enumerate}
\end{lemma}

Putting together Theorem \ref{P-B} and Lemma \ref{OMr}, we deduce
the following result.

\begin{theorem}\label{formula}
Let   $X:=(X_1,\ldots, X_n)$ and $Y:=(Y_1,\ldots, Y_n)$ be   in
$\cC_\rho$ such that $X\overset{H}{\sim}\, Y$. Then
  the metric $\delta_\rho$ satisfies the relation
\begin{equation*}
\begin{split}
\delta_\rho(X,Y)=\ln \max \left\{\sup_{r\in [0,1)}\left\|C_{\rho,
rX} C_{\rho,rY}^{-1} \right\|, \sup_{r\in [0,1)} \left\|C_{\rho,rY}
C_{\rho,rX}^{-1} \right\|\right\},
\end{split}
\end{equation*}
 where   $C_{\rho,X}:=\Delta_{\rho,X}(I-R_X)^{-1}$ and $R_X:=X_1^*\otimes
R_1+\cdots + X_n^*\otimes R_n$ is the reconstruction operator.
\end{theorem}

Using the Harnack type inequality of Theorem \ref{Harnack}, we
obtain an upper bound for the hyperbolic distance $\delta_\rho$ on
$[\cC_\rho]_{<1}$.  First, we need the following result.

\begin{proposition}\label{Harnack2}
Let $f$ be in the noncommutative disc algebra $\cA_n$  such that
$\Re f\geq 0$ and let $X:=(X_1,\ldots, X_n)\in \cC_\rho$ be with
$\omega_\rho(X)<1$. Then
$$
\rho\frac{1-\omega_\rho(X)}{1+\omega_\rho(X)}\Re f(0)\leq \Re
f(X_1,\ldots, X_n) +(\rho -1)\Re f(0)\leq
\rho\frac{1+\omega_\rho(X)}{1-\omega_\rho(X)}.
$$
\end{proposition}
\begin{proof} Let $r:=\omega_\rho (X)$ and define $Y:=\frac{1}{r} X$.
Since $\omega_\rho (Y)=\frac{1}{r}\omega_\rho(X)=1$, we deduce that
$Y\in \cC_\rho$. Applying Theorem \ref{Harnack} to $Y$, we obtain
$$
\frac{1-\omega_\rho(X)(2\rho-1)}{1+\omega_\rho(X)}\Re f(0)\leq \Re
f(X_1,\ldots, X_n)  \leq
\rho\frac{1+\omega_\rho(X)(2\rho-1)}{1-\omega_\rho(X)}.
$$
It is easy to see that the latter inequality is equivalent to the
one from the proposition.
\end{proof}

Now, we can deduce   the following  upper bound for the hyperbolic
distance on $ [\cC_\rho]_{<1}$.

\begin{corollary} \label{delta-ine} For any $X,Y\in [\cC_\rho]_{<1}$,
$$\delta_\rho(X,Y)\leq  \frac{1}{2} \ln
\frac{
(1+\omega_\rho(X))(1+\omega_\rho(Y))}{(1-\omega_\rho(X))(1-\omega_\rho(Y))}.
$$
\end{corollary}
\begin{proof}
Using Theorem \ref{equivalent} and the inequality of Proposition
\ref{Harnack2}, we deduce that
 $$ \Lambda_\rho(X,0)\leq  \left(
 \frac{1+\omega_\rho(X)}{1-\omega_\rho(X)}\right)^{1/2}.
 $$
On the other hand, since $\delta_\rho$ is a metric on
$[\cC_\rho]_{<1}$, we have $\delta(X,Y)\leq
\delta(X,0)+\delta_\rho(Y,0)$. Taking into account that $\delta_\rho
(X, Y)=\ln \Lambda_\rho(X,Y)$, the result follows.
\end{proof}

We remark that when $\rho=1$, the inequality of Corollary
\ref{delta-ine} is sharper then the one obtained in Corollary 2.5
from \cite{Po-hyperbolic}.

 Using Corollary \ref{T,02}, on can easily obtain
 the following result.

\begin{corollary} \label{T,03}
Let  $\rho>0$,  and $1\leq m<n$. Consider two  $n$-tuples $
A:=(A_1,\ldots, A_m)\in B(\cH)^m$ and $ B:=(B_1,\ldots, B_m)\in
B(\cH)^m$
 in  the class $ \cC_\rho$ and  their extensions $\widetilde A:=(A_1,\ldots, A_m,0,\ldots, 0)$
and $\widetilde B:=(B_1,\ldots, B_m,0,\ldots, 0)$    in $B(\cH)^n$,
respectively.
   Then
   $$ A\overset{\, H}{{ \sim}}\,  B \quad \text{  if and only
   if }\quad
    \widetilde A\overset{\, H}{{ \sim}}\, \widetilde B.
$$
  Moreover, in this case,
$$
\delta_\rho(A,B)=\delta_\rho(\widetilde A, \widetilde B).
$$
\end{corollary}

In what follows we provide a few properties for the map $\rho\mapsto
\delta_\rho(A,B)$.

\begin{lemma}
\label{ro-ro'}
 Let $A:=(A_1,\ldots, A_n)\in B(\cH)^n$ and
$B:=(B_1,\ldots, B_n)\in B(\cH)^n$ be in  the class $ \cC_\rho$ and
let $c>0$ and $0<\rho\leq \rho'$. Then the following statements
hold.
\begin{enumerate}
\item[(i)]
if $A\overset{H}{{\underset{c}\prec}}\, B$ in $\cC_\rho$, then
$A\overset{H}{{\underset{c}\prec}}\, B$ in $\cC_{\rho'}$;
\item[(ii)] if $A\,\overset{H}{{\underset{c}\sim}}\, B$  in
$\cC_\rho$, then  if $A\,\overset{H}{{\underset{c}\sim}}\, B$  in
$\cC_\rho$ and
$$ \delta_{\rho'}(A,B)\leq \delta_{\rho}(A,B).
$$
\end{enumerate}
\end{lemma}
\begin{proof} First recall that $\cC_\rho\subseteq \cC_{\rho'}$.
If $A\overset{H}{{\underset{c}\prec}}\, B$ in $\cC_\rho$, then
$$\Re p(A_1,\ldots, A_n)+ (\rho-1)\Re p(0)\leq c^2 \left[\Re p(B_1,\ldots, B_n)+
 (\rho-1)\Re p(0)\right] $$
 for any noncommutative polynomial with
matrix-valued coefficients $p\in \CC[X_1,\ldots, X_n]\otimes M_{m}$,
$m\in \NN$, such that $\Re p(X)\geq 0$ for any $X\in [B(\cH)^n]_1$.
Hence, $c\geq 1$ and, consequently,  the inequality above holds when
we replace $\rho$ with $\rho' \geq \rho$. This shows that
$A\overset{H}{{\underset{c}\prec}}\, B$ in $\cC_{\rho'}$. Part (ii)
is a clear consequence of (i) and the definition of the hyperbolic
metric.
\end{proof}

If $A:=(A_1,\ldots, A_n)\in B(\cH)^n$ is a nonzero $n$-tuple of
operators such that $A\in [C_\infty]_{<1}$, i.e., the joint spectral
radius $r(A)<1$, then
$$
\rho_A:=\inf \{\rho>0: \  A\in \cC_\rho \}>0.
$$
Indeed,   if $\rho,\rho'\in (0,\infty]$, $\rho\leq\rho'$,  then
$\cC_\rho\subseteq \cC_{\rho'}$ and, moreover, we have
$$
\omega_{\rho'}(A)\leq \omega_\rho(A),\quad r(A)=\lim_{\rho\to
\infty} \omega_\rho(A),\qquad A\in B(\cH)^n.
$$
Consequently, there exists $\rho>0$ such that $\omega_{\rho'}(A)<1$,
for any $\rho'\geq \rho$. Assume now that $\rho_A=0$. Then $T\in
\cC_\rho$, i.e.,  $\omega_\rho (A)\leq 1$ for any $\rho>0$. On the
other hand, we know that $\|A\|\leq \rho\omega_\rho(A)$. Taking
$\rho\to 0$, we deduce that $A=0$, which is a contradiction. This
proves our assertion.

Note that if $A,B\in [\cC_\infty]_{<1}$, then
$$
\rho_{A,B}:=\inf \{\rho>0: \  A, B \in \cC_\rho \}=\max\{\rho_A,
\rho_B\}.
$$

\begin{proposition}\label{invariant-curve} If  $A,B\in
[\cC_\infty]_{<1}$, then the map
$$
[\rho_{A,B},\infty)\ni \rho\mapsto \delta_\rho(A,B)\in \RR^+
$$
is continuous, decreasing, and
$$
\lim_{\rho\to\infty} \delta_\rho(A,B)=0.
$$
\end{proposition}
\begin{proof}
Using Theorem \ref{P-B} and Lemma \ref{ro-ro'}, one can easily
deduce that the map $\rho\mapsto \delta_\rho(A,B)$ is continuous and
decreasing. To prove the last part of the proposition, note that
since $\delta_\rho(A,B)\leq \delta_\rho(A,0) +\delta_\rho(0,B)$, it
is enough to show that $\lim_{\rho\to \infty}\delta_\rho(A,0)=0$. To
this end, note that Theorem \ref{P-B}, implies
\begin{equation}
\label{fo} \delta_\rho(A,0)=\ln \max \left\{ \left\|C_{\rho,A}
C_{\rho,0}^{-1} \right\|,
  \left\|C_{\rho,0} C_{\rho,A}^{-1} \right\|\right\},
\end{equation}
where \begin{equation*}
 \begin{split}
C_{\rho,A}C_{\rho,0}^{-1}
  =\frac{1}{\sqrt{\rho}}  \left[\rho I+(1-\rho)(R_{A}^*+ R_{A})+
(\rho-2)R_{A}^* R_{A}\right]^{1/2}(I-R_A)^{-1}.
\end{split}
\end{equation*}
Hence, we deduce that
\begin{equation*}
\begin{split}
\lim_{\rho\to \infty}\|C_{\rho,A}C_{\rho,0}^{-1}\| &=   \left\|
\left[ I-(R_{A}^*+ R_{A})+ R_{A}^*
R_{A}\right]^{1/2}(I-R_A)^{-1}\right\|\\
&=
 \left\|(I-R_A^*)^{-1}
\left[ I-(R_{A}^*+ R_{A})+ R_{A}^*
R_{A}\right](I-R_A)^{-1}\right\|\\
&=
 \left\|(I-R_A^*)^{-1}
( I-R_{A}^*)(I- R_{A})(I-R_A)^{-1}\right\|\\
&=1
\end{split}
\end{equation*}
Similarly,  we have  $\lim_{\rho\to
\infty}\|C_{\rho,0}C_{\rho,A}^{-1}\|=1$. Using now relation
\eqref{fo}, we complete the proof.
\end{proof}

\bigskip

\section{ Mapping theorems for free holomorphic functions on noncommutative balls}

In this section, we   provide mapping theorems, spectral von Neumann
inequalities,  and Schwarz type  results  for free holomorphic
functions on noncommutative balls, with respect to the hyperbolic
metric  and the operator radius $\omega_\rho$, $\rho\in (0,\infty]$.

First, we prove the following mapping theorem for the classes
$\cC_\rho$, $\rho>0$.

\begin{theorem}
\label{rho-f} Let $f:=(f_1,\ldots, f_m)$ be a contractive  free
holomorphic function with $\|f(0)\|<1$ such that the  boundary
functions $\widetilde f_1,\ldots, \widetilde f_m$ are in the
noncommutative disc algebra $\cA_n$. If \  $(T_1,\ldots, T_n)\in
B(\cH)^n$  is of  class $\cC_\rho$, $\rho>0$, then
$f(T_1,\ldots,T_n)$ is of class $\cC_{\rho_f}$, where
\begin{equation}
\label{rof} \rho_f:=\begin{cases}
1+(\rho-1)\frac{1-\|f(0)\|}{1+\|f(0)\|}&\quad
\text{if }\ \rho<1\\
\\
1+(\rho-1)\frac{1+\|f(0)\|}{1-\|f(0)\|}&\quad \text{if }\ \rho\geq
1.
\end{cases}
\end{equation}
\end{theorem}

\begin{proof}
Let  $p\in \CC[Z_1,\ldots, Z_m]\otimes M_k$, $k\in \NN$, be  such
that $\Re p\geq 0$ on the unit ball $[B(\cH)^m]_1$. This is
equivalent to $\Re p(S_1',\ldots S_m')\geq 0$, where $S_1',\ldots,
S_m'$ are the left creation operators on the full Fock space
$F^2(H_m)$. Applying the noncommutative Poisson transform
$P_{f(X_1,\ldots, X_n)}\otimes \text{\rm id}$, which is  a
completely positive linear map,  to the inequality $\Re
p(S_1',\ldots S_m')\geq 0$, we obtain
$$
\Re p(f(X_1,\ldots, X_n))\geq 0,\qquad X\in [B(\cH)^n]_1.
$$
Moreover, since the  boundary functions $\widetilde f_1,\ldots,
\widetilde f_m$ are in the noncommutative disc algebra $\cA_n$, we
deduce that the boundary function of the composition $p\circ f$ is
$p(\widetilde f_1,\ldots, \widetilde f_m)\in \cA_n\bar \otimes_{min}
M_k$.

Assume that $ (T_1,\ldots, T_n)\in\cC_\rho$. Using the  free
pluriharmonic   functional calculus of Theorem \ref{funct-calc} and
 Theorem \ref{ro-contr}, we deduce that
\begin{equation}
\label{Rep} \Re (p\circ f)(T_1,\ldots, T_n)+(\rho -1) \Re (p\circ
f)(0)\geq 0.
\end{equation}
On the other hand, according  to the Harnack type inequality of Theorem 1.4 from
\cite{Po-hyperbolic} applied to  the  positive  free pluriharmonic
function $\Re p$ at  the point $f(0)=(f_1(0),\ldots, f_m(0))$, we
have
\begin{equation}
\label{Rep>0}
 \Re p(0)\frac{1-\|f(0)\|}{1+\|f(0)\|}\leq \Re p(f(0))\leq
\Re p(0)\frac{1+\|f(0)\|}{1-\|f(0)\|}.
\end{equation}

Let $\gamma>0$ and note that
\begin{equation}\label{AB}
\Re p( f(T_1,\ldots,T_n)) +(\gamma-1)\Re p(0)=A+B,
\end{equation}
where
\begin{equation}
\begin{split}
A&:= \Re
p( f(T_1,\ldots,T_n)) +(\rho-1)p(f (0) ) \\
B&:=(\gamma-1)\Re p(0)-(\rho-1)p(f (0) ).
\end{split}
\end{equation}

Assume now that $\rho\geq 1$. Using the second inequality in
\eqref{Rep>0}, we obtain
\begin{equation*}
\begin{split}
B&\geq (\gamma -1)\Re p(0)-(\rho -1) \Re
p(0)\frac{1+\|f(0)\|}{1-\|f(0)\|}\\
&=\left[(\gamma -1)-(\rho -1) \frac{1+\|f(0)\|}{1-\|f(0)\|}\right]
\Re p(0),
\end{split}
\end{equation*}
which is positive if  $\gamma \geq 1+(\rho -1)
\frac{1+\|f(0)\|}{1-\|f(0)\|}.$ In this case, using relation
\eqref{AB} and \eqref{Rep}, we obtain
$$
\Re p(f(T_1,\ldots, T_n)) +(\gamma-1)\Re p(0)\geq 0
$$
for any $p\in \CC[Z_1,\ldots, Z_n]\otimes M_k$, $k\in \NN$, be  such
that $\Re p\geq 0$ on the unit ball $[B(\cH)^m]_1$. Applying Theorem
\ref{ro-contr}, we deduce that $ f(T_1,\ldots,T_n)\in \cC_\gamma$.
In  particular, we have $ f(T_1,\ldots,T_n)\in \cC_{\delta_f}$ where
$$\delta_f:=1+(\rho-1)\frac{1+\|f(0)\|}{1-\|f(0)\|}.
$$

Now, we consider the case $\rho\in (0,1)$. Using the first
inequality in  \eqref{Rep>0}, we obtain
\begin{equation*}
\begin{split}
B&\geq \left[(\gamma -1)-(\rho -1)
\frac{1-\|f(0)\|}{1+\|f(0)\|}\right] \Re p(0),
\end{split}
\end{equation*}
which is positive if  $\gamma \geq 1+(\rho -1)
\frac{1-\|f(0)\|}{1+\|f(0)\|}.$ As above,  using relations
\eqref{AB} and \eqref{Rep}, we obtain
$$
\Re p( f(T_1,\ldots,T_n)) +(\gamma-1)\Re p(0)\geq 0
$$
 for any $p\in \CC[Z_1,\ldots, Z_n]\otimes M_k$, $k\in \NN$, be  such that
$\Re p\geq 0$ on the unit ball $[B(\cH)^m]_1$.  Theorem
\ref{ro-contr} implies  $ f(T_1,\ldots,T_n)\in \cC_\gamma$. In
particular,  we have $ f(T_1,\ldots,T_n)\in \cC_{\delta_f}$ where
$$\delta_f:=1+(\rho-1)\frac{1-\|f(0)\|}{1+\|f(0)\|}.
$$
The proof is complete.
\end{proof}

Note  that under the conditions of  Theorem \ref{rho-f}, $\rho\leq
\rho_f$  and  $\rho=1\implies \rho_f=1$. Moreover,  if $\rho\neq 1$,
then $\rho_f=\rho$ if and only if $f(0)=0$. On can also show that
$\rho_f\leq 1$ if $\rho\leq 1$.

We   remark that, under the conditions of Theorem \ref{rho-f}, there
exists $T:=(T_1,\ldots,T_n)\in B(\cH)^n$ such that if $\rho>0$ is
the smallest positive number such that $(T_1,\ldots,T_n)\in
\cC_\rho$, then there exists a  free holomorphic function $f$ such
that $\rho_f$ is the smallest positive number  with the property
that $f(T_1,\ldots,T_n)\in \cC_{\rho_f}$. Indeed, if $n\leq m$, take
$f(X_1,\ldots, X_n)=(X_1,\ldots, X_n,0,\ldots,0)$ and use Corollary
\ref{T,02}. When $n>m$, take $f(X_1,\ldots, X_n)=(X_1,\ldots, X_m )$
and $T:=(T_1,\ldots,T_n,0,\ldots,0)$ with $(T_1,\ldots,T_n)\in
\cC_\rho$.

\begin{corollary}\label{Ome} Let $f:=(f_1,\ldots, f_m)$ be a bounded free
holomorphic function with $\|f(0)\|<\|f\|_\infty$ such that the
boundary functions $\widetilde f_1,\ldots, \widetilde f_m$ are in
the noncommutative disc algebra $\cA_n$. If  $(T_1,\ldots, T_n)\in
B(\cH)^n$  is of  class $\cC_\rho$, $\rho>0$, then
$$
\omega_{\rho_f}( f(T_1,\ldots,T_n))\leq \|f\|_\infty,
$$
 where $\rho_f$ is given by relation \eqref{rof}. In particular, if
$f(0)=0$ and $(T_1,\ldots, T_n)\in \cC_\rho$, then
$$
\omega_{\rho}(f(T_1,\ldots,T_n))\leq \|f\|_\infty.
$$
\end{corollary}
\begin{proof} Applying Theorem \ref{rho-f}  the function
$\frac{1}{\|f\|_\infty} f$, we deduce that $\frac{1}{\|f\|_\infty}
f(T_1,\ldots,T_n)$ is in the class $\cC_{\rho_f}$, which is
equivalent to $\omega_{\rho_f}\left( \frac{1}{\|f\|_\infty}
f(T_1,\ldots,T_n)\right)\leq 1$, and  the first inequality of the
theorem follows. Hence, and using the fact that $\rho_f=\rho$ when
$f(0)=0$,  we  complete the proof.
\end{proof}

A simple consequence of Corollary \ref{Ome} is the following power
inequality.
\begin{corollary} If  \ $(T_1,\ldots, T_n)\in B(\cH)^n$ is non-zero,
$\rho\in (0,\infty)$, and $k\geq 1$,  then
$$
\omega_\rho(T_\alpha:\ \alpha\in \FF_n^+, |\alpha|=k)\leq
\omega_\rho(T_1,\ldots, T_n).
$$
\end{corollary}
\begin{proof} Since $\|(T_1,\ldots, T_n)\|\leq \rho \omega_\rho(T_1,\ldots, T_n)$, we have
$\omega_\rho(T_1,\ldots, T_n)\neq 0$. Applying  the second part of
Corollary \ref{Ome} to the $n$-tuple  of operators
$\left(\frac{1}{\omega_\rho(T_1,\ldots, T_n)} T_1,\ldots,
\frac{1}{\omega_\rho(T_1,\ldots, T_n)} T_n\right)\in \cC_\rho$ and
to the free holomorphic function
$$
f(X_1,\ldots, X_n):=(X_\alpha:\ \alpha\in
\FF_n^+, |\alpha|=k),\qquad (X_1,\ldots, X_n)\in [B(\cH)^n]_1,
$$
 we complete the proof.
\end{proof}

\begin{theorem} \label{sup}
Let $f:=(f_1,\ldots, f_m)$ be a bounded free holomorphic function
with $\|f(0)\|<\|f\|_\infty$ such that the boundary functions
$\widetilde f_1,\ldots, \widetilde f_m$ are in the noncommutative
disc algebra $\cA_n$. Then, for each $r\in [0,1)$,
$$
\sup_{T\in \cC_\rho, \ \omega_\rho(T)\leq r}
\omega_{\rho_f}(f(T_1,\ldots, T_n))\leq \|f(rS_1,\ldots, rS_n)\|,
$$
where $S_1,\ldots, S_n$ are  the left creation operators.
\end{theorem}
\begin{proof}
Consider the free holomorphic function $f_r$, defined by
$$f_r(X_1,\ldots, X_n):=f(rX_1,\ldots, rX_n),\qquad (X_1,\ldots,
X_n)\in [B(\cH)^n]_1
$$
and recall that $\|f_r\|_\infty=\|f(rS_1,\ldots, rS_n)\|$. Applying
Corollary \ref{Ome} to $f_r$, we have
 \begin{equation}
 \label{Oro}
\omega_{\rho_{f_r}}(f_r(A_1,\ldots,A_n))\leq \|f_r\|_\infty,\quad
(A_1,\ldots,A_n)\in \cC_\rho
\end{equation}
Since $f(0)=f_r(0)$, we have $\rho_{f}=\rho_{f_r}$. Consequently, if
we assume that $\omega_\rho(T_1,\ldots, T_n)\leq r<1$, then $
\left(\frac{1}{r} T_1,\ldots, \frac{1}{r}T_n\right)\in \cC_\rho$ and
inequality \eqref{Oro} implies
$$
\omega_{\rho_f}(f(T_1,\ldots,
T_n))=\omega_{\rho_f}\left(f_r\left(\frac{1}{r} T_1,\ldots,
\frac{1}{r}T_n\right)\right)\leq \|f_r\|_\infty,
$$
which completes the proof.
\end{proof}

\begin{corollary} Let $ (T_1,\ldots,
T_n)\in B(\cH)^n$ be such that $\omega_\rho(T_1,\ldots, T_n)<1$, and
let $f:=(f_1,\ldots, f_m)$ be a bounded free holomorphic function
with the following properties:
\begin{enumerate}
\item[(i)]
 the boundary functions $\widetilde
f_1,\ldots, \widetilde f_m$ are in the noncommutative disc algebra
$\cA_n$.
\item[(ii)]  $f_j$  has the standard
representation of the form
$$f_j(X_1,\ldots, X_n)=\sum_{|\alpha|\geq k} a_\alpha^{(j)}X_\alpha,\qquad
j=1,\ldots,m.
$$
\end{enumerate}
Then
$$
\omega_\rho (f (T_1,\ldots, T_n))\leq \omega_\rho(T_1,\ldots, T_n)^k
\|f\|_\infty.
$$
\end{corollary}
\begin{proof} Consider  the  free holomorphic function
$g:=\frac{1}{\|f\|_\infty} f$. Note that $\|g\|_\infty=1$ and
$g(0)=0$. According to the Schwarz lemma for free holomorphic
functions (see Theorem 2.4 from \cite{Po-holomorphic}), we have
$$
\|g(X_1,\ldots, X_n)\|\leq \|(X_1,\ldots, X_n)\|^k,\qquad
(X_1,\ldots, X_n)\in [B(\cH)^n]_1.
$$
Denote $r:=\omega_\rho(T_1,\ldots, T_n)<1$, $\rho>0$, and consider
$$g_r(X_1,\ldots, X_n):=g(rX_1,\ldots, rX_n),\qquad (X_1,\ldots,
X_n)\in [B(\cH)^n]_1.
$$
Note that  the inequality above implies $\|g_r\|_\infty\leq r^k$.
 Applying now
Theorem \ref{sup} to $g$, and using the latter  inequality, we
obtain
$$ \omega_\rho
(g(T_1,\ldots, T_n))\leq \|g_r\|_\infty\leq
r^k=\omega_\rho(T_1,\ldots,T_n)^k.
$$
  Hence, the result follows.
\end{proof}

\begin{corollary} Let $(T_1,\ldots,
T_n)\in B(\cH)^n$ be such that $\omega_\rho(T_1,\ldots,T_n)<1$, and
let $f:[B(\cH)^n]_1\to B(\cH)$ be a free holomorphic function  with
$\Re f\leq I$ and having the standard representation
$$f(X_1,\ldots, X_n)=\sum_{|\alpha|\geq k} a_\alpha X_\alpha,\qquad
\text{ where } \ k\geq 1.
$$
 Then
$$
\omega_\rho (f(T_1,\ldots, T_n))\leq  \frac{2 \omega_\rho
(T_1,\ldots, T_n)^k}{1- \omega_\rho (T_1,\ldots, T_n)^k}.
$$
\end{corollary}

\begin{proof} According to the
 Carath\'eodory type result for free holomorhic functions (see Theorem 5.1 from
 \cite{Po-holomorphic.II}), we have
$$
  \|f(X_1,\ldots,X_n)\|\leq
\frac{2\|\sum_{|\beta|=k}X_\beta
X_\beta^*\|^{1/2}}{1-\|\sum_{|\beta|=k}X_\beta X_\beta^*\|^{1/2}},
\qquad (X_1,\ldots,X_n)\in [B(\cH)^n]_1.
$$
Hence, we deduce that  $\|f_r\|_\infty\leq \frac{2r^k}{1-r^k}$ for
any $r\in (0,1)$.  Setting $r:=\omega_\rho(T_1,\ldots, T_n)<1$,
$\rho>0$, and applying Theorem \ref{sup}, we obtain
$$
\omega_\rho (f(T_1,\ldots,T_n))\leq \|f_r\|_\infty\leq  \frac{2
\omega_\rho (T_1,\ldots, T_n)^k}{1- \omega_\rho (T_1,\ldots,
T_n)^k}.
$$
The proof is complete.
\end{proof}

\begin{lemma}\label{H<} Let $f:=(f_1,\ldots, f_m)$ be a
contractive  free holomorphic function with $\|f(0)\|<1$ such that
the  boundary functions $\widetilde f_1,\ldots, \widetilde f_m$ are
in the noncommutative disc algebra $\cA_n$.  Let $A:=(A_1,\ldots,
A_n)\in B(\cH)^n$ and $B:=(B_1,\ldots, B_n)\in B(\cH)^n$ be in  the
class $ \cC_\rho\subset B(\cH)^n$ and let $c\geq 1$. If
$A\overset{H}{{\underset{c}\prec}}\, B$, then $f(A)$ and  $f(B)$ are
in $\cC_{\rho_f}\subset  B(\cH)^m$  and
$f(A)\overset{H}{{\underset{c}\prec}}\, f(B)$, where $\rho_f$ is
given by relation \eqref{rof}.
\end{lemma}

\begin{proof}
First, note that, due to Theorem \ref{rho-f},   $f(A),f(B)$ are in
$\cC_{\rho_f}$, where $\rho_f$ is given by relation \eqref{rof}. Let
$p\in \CC[Z_1,\ldots, Z_m]\otimes M_k$, $k\in \NN$, be  such that
$\Re p\geq 0$ on the unit ball $[B(\cH)^m]_1$. According to the
proof of Theorem \ref{rho-f}, the boundary function of the
composition $p\circ f$ is $p(\widetilde f_1,\ldots, \widetilde
f_m)\in \cA_n\bar \otimes_{min} M_k$ and $\Re (p\circ f)\geq 0$.
Using the free pluriharmonic functional calculus  for the class
$\cC_\rho$ and Theorem \ref{equivalent}, if $A,B$ are in $\cC_\rho$
and $A\overset{H}{{\underset{c}\prec}}\, B$, $c\geq 1$,  then
\begin{equation}
\label{Rep2} \Re (p\circ f)(A_1,\ldots, A_n)+(\rho-1)\Re (p\circ
f)(0) \leq c^2\left[ \Re (p\circ f)(B_1,\ldots, B_n)+(\rho-1)\Re
(p\circ f)(0)\right].
\end{equation}
Assume now  that $\rho\geq 1$. Due to the Harnack type inequality
\eqref{Rep>0}, the inequality \eqref{Rep2} implies
$$
 \Re (p\circ f)(A_1,\ldots, A_n)\leq c^2 \Re (p\circ f)(B_1,\ldots,
 B_n) + (c^2-1)(\rho-1)\Re p(0)\frac{1+\|f(0)\|}{1-\|f(0)\|},
 $$
 which is equivalent to
\begin{equation*}
 \Re (p\circ f)(A_1,\ldots, A_n)+(\rho_f-1)\Re (p\circ f)(0) \leq
c^2\left[ \Re (p\circ f)(B_1,\ldots, B_n)+(\rho_f-1)\Re (p\circ
f)(0)\right],
\end{equation*}
where  $\delta_f:=1+(\rho-1)\frac{1+\|f(0)\|}{1-\|f(0)\|}.$ Applying
Theorem \ref{equivalent}, we deduce that
$f(A)\overset{H}{{\underset{c}\prec}}\, f(B)$.

Now, we consider the case $\rho\in (0,1)$.
  The inequality \eqref{Rep2}  and the  Harnack type inequality
\eqref{Rep>0} imply
$$
 \Re (p\circ f)(A_1,\ldots, A_n)\leq c^2 \Re (p\circ f)(B_1,\ldots,
 B_n) + (c^2-1)(\rho-1)\Re p(0)\frac{1-\|f(0)\|}{1+\|f(0)\|}.
 $$
As above, we deduce that $f(A)\overset{H}{{\underset{c}\prec}}\,
f(B)$ in $\cC_{\rho_f}$, where
$\delta_f:=1+(\rho-1)\frac{1-\|f(0)\|}{1+\|f(0)\|}.$ This completes
the proof.
\end{proof}

\begin{theorem}\label{Schwarz-Hyp} Let $\delta_\rho
:\Delta\times \Delta\to [0,\infty)$
 be the hyperbolic metric on a Harnack part $\Delta$ of $\cC_\rho$, and
let $f:=(f_1,\ldots, f_m)$ be a contractive  free holomorphic
function with $\|f(0)\|<1$ such that the boundary functions
$\widetilde f_1,\ldots, \widetilde f_m$ are in the noncommutative
disc algebra $\cA_n$.Then
$$
\delta_{\rho_f}(f(A),f(B))\leq  \delta_\rho(A,B),\qquad A,B\in
\Delta,
$$
 where $\rho_f$ is given by relation \eqref{rof}.
\end{theorem}

\begin{proof} Let $A,B\in \Delta\subset \cC_\rho$, i.e., there is $c\geq 1$ such that
 $A\overset{H}{{\underset{c}\sim}}\, B$. According to Theorem \ref{rho-f} and  Lemma
 \ref{H<},  $f(A)$ and  $f(B)$ are
in $\cC_{\rho_f}$,  and $f(A)\overset{H}{{\underset{c}\sim}}\, f(B)$
in
 $\cC_{\rho_f}$, where $\rho_f$ is given by relation
\eqref{rof}. Hence and taking into account that
$$\delta_\rho(A,B):=\ln \inf\left\{ c > 1: \
A\,\overset{H}{{\underset{c}\sim}}\, B   \right\},\qquad A,B\in
\Delta,$$ we deduce that
$$
\delta_{\rho_f}(f(A),f(B))\leq  \delta_\rho(A,B),\qquad A,B\in
\Delta.
$$
 The proof is complete.
\end{proof}

Now, we can deduce the following Schwarz type result.

\begin{corollary}\label{f(0)} Let $\delta_\rho
:\Delta\times \Delta\to [0,\infty)$
 be the hyperbolic metric on  a Harnack part $\Delta$ of $\cC_\rho$, and
let $f:=(f_1,\ldots, f_m)$ be a contractive  free holomorphic
function with $f(0)=0$ such that the boundary functions $\widetilde
f_1,\ldots, \widetilde f_m$ are in the noncommutative disc algebra
$\cA_n$. Then
$$
\delta_{\rho}(f(A),f(B))\leq  \delta_\rho(A,B),\qquad A,B\in \Delta.
$$
\end{corollary}

We recall that, due to Theorem \ref{foias2}, the open ball
$[\cC_\rho]_{<1}$ is the Harnack part of $\cC_\rho$ containing $0$.
Consequently, Theorem \ref{Schwarz-Hyp} and Corollary \ref{f(0)}
hold in  the particular case when $\Delta:=[\cC_\rho]_{<1}$.

Ky Fan \cite{KF1} showed that the von Neumann inequality \cite{vN}
is equivalent to the fact that if $T\in B(\cH)$ is a strict
contraction ($\|T\|<1$) and $f:\DD\to \DD$ is an analytic function,
then $\|f(T)\|<1$. A multivariable analogue of this result was
obtained in \cite{Po-holomorphic.II}. In what follows, we provide  a
spectral version of this result, when the norm is replaced by the
joint spectral radius.

\begin{theorem}\label{radius} Let $f:=(f_1,\ldots, f_m)$ be a
contractive  free holomorphic function with $\|f(0)\|<1$ such that
the  boundary functions $\widetilde f_1,\ldots, \widetilde f_m$ are
in the noncommutative disc algebra $\cA_n$. If $(T_1,\ldots, T_n)\in
B(\cH)^n$ and  the joint spectral radius $r(T_1,\ldots, T_n)<1$,
then
$$
r(f(T_1,\ldots, T_n))<1. $$
\end{theorem}

\begin{proof}
  Assume that
$ (T_1,\ldots, T_n)\in B(\cH)^n$ has the joint spectral radius
$r(T_1,\ldots, T_n)<1$. Taking into account that  $r(T_1,\ldots,
T_n)=\lim_{\rho\to \infty} \omega_\rho(T_1,\ldots, T_n)$,  we find
$\delta> 1$ such that $\omega_\rho(T_1,\ldots, T_n)<1$.  Therefore,
we have $T:=(T_1,\ldots, T_n)\in \cC_\rho$ and, due to Theorem
\ref{foias2}, the $n$-tuple $T$ is Harnack equivalent to $0$.
Consequently, $T\overset{H}{{\underset{c}\prec}}\, 0$ for some
constant $c\geq 1$. According to Theorem \ref{rho-f},  $f(T)$ and
$f(0)$ are in the class  $\cC_{\rho_f}$, where $\rho_f$ is given by
relation \eqref{rof}. On the other hand, Lemma \ref{H<} implies
$f(T)\overset{H}{{\underset{c}\prec}}\, f(0)$ in  $\cC_{\rho_f}$.
Since   $\|f(0)\|<1$, we have the joint spectral radius $r(f(0))<1$.
Applying Theorem \ref{A<0}, we deduce that
$f(0)\overset{H}{{\underset{c}\prec}}\, 0$ in $\cC_{\rho_f}$.
Therefore, we have $f(T)\overset{H}{{\underset{c}\prec}}\, 0$ in
$\cC_{\rho_f}$. Applying again Theorem \ref{A<0}, we have
$r(f(T))<1$. The proof is complete.
\end{proof}

An analogue of Theorem \ref{radius} for $n$-tuples of operators with
joint operator radius  $\omega_\rho(T_1,\ldots, T_n)<1$ is the
following.

\begin{theorem}\label{Om<} Let $f:=(f_1,\ldots, f_m)$ be a
contractive  free holomorphic function with $\|f(0)\|<1$ such that
the  boundary functions $\widetilde f_1,\ldots, \widetilde f_m$ are
in the noncommutative disc algebra $\cA_n$. If $(T_1,\ldots, T_n)\in
B(\cH)^n$ and $\omega_\rho(T_1,\ldots, T_n)<1$, then
$$
\omega_{\rho_f}(f(T_1,\ldots, T_n))<1, $$
 where $\rho_f$ is defined by relation \eqref{rof}. In particular,
 if $f(0)=0$, then \ $\omega_{\rho}(f(T_1,\ldots, T_n))<1$.
\end{theorem}

\begin{proof}If $T:=(T_1,\ldots, T_n)\in
B(\cH)^n$ and $\omega_\rho(T_1,\ldots, T_n)<1$, then $T\in
\cC_\rho$. According to Theorem \ref{foias2}, we have
$$r(T_1,\ldots,
T_n)<1\quad \text{ and } \quad P_\rho(T,R)\geq a I$$
 for some constant $a>0$. Applying
Theorem \ref{rho-f} and  Theorem \ref{radius}, we deduce that
$f(T)\in \cC_{\rho_f}$ and  $r(f(T))<1$.
Since $\omega_\rho(T)<1$, Theorem \ref{foias2} implies
$T\overset{H}{\sim}\, 0$.  In particular, we have
$0\overset{H}{{\underset{c}\prec}}\, T$ for some constant $c\geq 1$.
Applying Lemma \ref{H<}, we deduce that
$f(0)\overset{H}{{\underset{c}\prec}}\, f(T)$ in  $\cC_{\rho_f}$,
where $\rho_f$ is given by relation \eqref{rof}. Hence, and using
Theorem \ref{equivalent} (part (ii)), we get
$$
P_{\rho_f}(rf(0), R)\leq c^2 P_{\rho_f}(rf(T), R),\qquad r\in [0,1).
$$
Since $r(f(0))<1$ and  $r(f(T))<1$, the latter inequality implies
\begin{equation}
\label{Pro2}
 P_{\rho_f}(f(0), R)\leq c^2 P_{\rho_f}(f(T), R),\qquad
r\in [0,1).
\end{equation}

On the other hand,  since the mapping $X\mapsto P_1(X,R)$ is a
positive free pluriharmonic  function on $[B(\cH)^n]_1$, the Harnack
inequality  \eqref{u-ine} implies
$$
P_1(f(0),R)\geq P_1(0,R)\frac{1-\|f(0)\|}{1+\|f(0)\|}=
\frac{1-\|f(0)\|}{1+\|f(0)\|} I.
$$
Therefore, we have
\begin{equation*}
\begin{split}
P_{\rho_f}(f(0), R) &=P_1(f(0),R)+(\rho_f-1)I\\
&\geq \left(\rho_f -1 +\frac{1-\|f(0)\|}{1+\|f(0)\|}\right) I.
\end{split}
\end{equation*}
Since
\begin{equation*}
a:=\rho_f -1 +\frac{1-\|f(0)\|}{1+\|f(0)\|}
=\begin{cases}
\rho\frac{1-\|f(0)\|}{1+\|f(0)\|}&\quad
\text{if }\ \rho<1\\
\\
(\rho-1)\frac{1+\|f(0)\|}{1-\|f(0)\|}+\frac{1-\|f(0)\|}{1+\|f(0)\|}
&\quad \text{if }\ \rho\geq
1,
\end{cases}
\end{equation*}
we have $a>0$. Combining  the latter inequality with \eqref{Pro2} we
obtain
$$
P_{\rho_f}(f(T), R)\geq \frac{a}{c^2} I.
$$
Using again Theorem \ref{foias2}, we deduce that
$\omega_{\rho_f}(f(T))<1$.  The last part of the theorem follows
from Theorem \ref{rho-f}. This completes the proof.
\end{proof}

\begin{remark} If $m=1$,    all the results of this section remain true
when
 the  condition
 $\|f(0)\|<1$ is dropped  if  $f$ is
a nonconstant contractive free holomorphic function with boundary
function in the noncommutative algebra $\cA_n$.
\end{remark}

\bigskip

\section{Carath\' eodory  metric on  the open noncommutative ball
$[\cC_\infty]_{<1}$ and Lipschitz  mappings}

In this section, we introduce a Carath\' eodory type metric  $d_K$
on the open ball of all $n$-tuples of operators $(T_1,\ldots, T_n)$
with joint spectral radius $r(T_1,\ldots, T_n)<1$. We obtain a
concrete formula for $d_K$
  in terms of the free pluriharmonic kernel on the open unit ball
  $[\cC_\infty]_{<1}$. This is used to  prove that the metric $d_K$ is complete on
$[\cC_\infty]_{<1}$ and its topology coincides with  the operator
norm topology.

We need some notation. Consider the noncommutative balls
$$
[\cC_\rho]_{<1}:=\left\{ (X_1,\ldots, X_n)\in B(\cH)^n:\
\omega_\rho(X_1,\ldots, X_n)<1\right\}\quad  \text { for }\ \rho\in
(0,\infty],
$$
where $\omega_\infty (X_1,\ldots, X_n):=r(X_1,\ldots, X_n)$ is the
joint spectral radius of $(X_1,\ldots, X_n)$, and    set
$$
 [\cC_\rho]^{\prec\, 0}:=\cC_\rho\cap [\cC_\infty]_{<1} \quad  \text { for }\ \rho\in
(0,\infty).
$$
According to Theorem 1.35 from \cite{Po-unitary},   if $\rho,\rho'\in (0,\infty]$,
$\rho\leq\rho'$,  then $\cC_\rho\subseteq \cC_{\rho'}$ and,
moreover, we have
$$
\omega_{\rho'}(X)\leq \omega_\rho(X),\quad r(X)=\lim_{\rho\to
\infty} \omega_\rho(X),\qquad X\in B(\cH)^n.
$$
Consequently, we have
 $$
[\cC_\rho]^{\prec\, 0}\subseteq [\cC_{\rho'}]^{\prec\, 0},\quad
\quad [\cC_\rho]_{<1}\subseteq [\cC_{\rho'}]_{<1}.
$$
Due to Theorem \ref{A<0}  and  Theorem \ref{foias2},  one can easily
see that
$$\{X\in \cC_\rho:  X\overset{H}{\sim}\, 0\}=[\cC_\rho]_{<1}\subset [\cC_\rho]^{\prec\, 0}
=\left\{X\in  \cC_\rho:\
X\overset{H}{{ \prec}}\, 0\right\}
$$
for any $\rho\in (0,\infty)$.
Note  also that
 $$
 \bigcup_{\rho>0}
[\cC_\rho]_{<1}=\bigcup_{\rho>0} [\cC_\rho]^{\prec\, 0}=
[\cC_\infty]_{<1}.
$$
Indeed, if $X\in [\cC_\infty]_{<1}$, i.e., $r(X)<1$, then taking
into account that $r(X)=\lim_{\rho\to \infty} \omega_\rho(X)$, we
find $\rho>0$ such that $\omega_\rho(X)<1$. Thus $X\in
[\cC_\rho]_{<1}$, which proves our assertion. Note also  that
$\bigcup\limits_{\rho>0} [\cC_\rho]_{<1}$ is dense (in the norm
topology) in the set $\cC_\infty$  of all $n$-tuples of operators
$(T_1,\ldots, T_n)$  with joint spectral radius $r(T_1,\ldots,
T_n)\leq 1$.

 Now, we introduce  the map $d_K:  [\cC_\infty]_{<1}\times
  [\cC_\infty]_{<1}\to [0,\infty)$ by setting
\begin{equation}
\label{d_K0} d_K(A,B)=\sup_p \|\Re  p(A)-\Re p(B)\|,\qquad A,B\in
[\cC_\infty]_{<1},
\end{equation}
where the supremum is taken over all polynomials $p\in
\CC[X_1,\ldots, X_n]\otimes M_m$, $m\in \NN$, with $\Re p(0)=I$ and
$\Re p\geq 0$ on $[B(\cH)^n]_1$. In what follows we will prove that
$d_K$ is  a metric  and obtain a concrete formula
  in terms of the free pluriharmonic kernel on the open unit ball
  $[\cC_\infty]_{<1}$.

First, we need the following result.

\begin{lemma} \label{ine-P} Let  $G$ be  a free pluriharmonic
function on $[B(\cH)^n]_1$  with coefficients in $B(\cE)$, such that
$G(0)=I$ and $G\geq 0$. If $A,B\in  [\cC_\infty]_{<1}$, then
$$
\|G(A)-G(B)\|\leq \|P_1(A,R)-P_1(B,R)\|,
$$
where where $P_1(X,R)$ is the free pluriharmonic Poisson kernel
defined by
\begin{equation*}
 P_1(X,R) := \sum_{k=1}^\infty \sum_{|\alpha|=k}
   X_{  \alpha}\otimes R_{\tilde\alpha}^* +I\otimes I+
  \sum_{k=1}^\infty \sum_{|\alpha|=k}
  X_{\alpha}^* \otimes R_{\tilde \alpha}, \qquad X\in
  [\cC_\infty]_{<1},
  \end{equation*}
   and  the convergence is in the operator norm topology.
\end{lemma}
\begin{proof} Since $G$ is a positive free pluriharmonic
function of $[B(\cH)^n]_1$  it  has a unique representation of the
form
$$G(X_1,\ldots, X_n)=\sum_{k=1}^\infty \sum_{|\alpha|=k}
X_\alpha^*\otimes A_{(\alpha)}^* +
    I\otimes I+\sum_{k=1}^\infty \sum_{|\alpha|=k}
      X_\alpha\otimes A_{(\alpha)},\qquad X\in [B(\cH)^n]_1,$$
      for some $A_{(\alpha)}\in B(\cE)$,
where the series  converge in the operator norm  topology. Applying
Theorem 5.2 from \cite{Po-pluriharmonic} to $G$, we  find a
completely positive linear map $\mu:\cR_n^*+ \cR_n\to B(\cE)$ with
$\mu(I)=I$ and $\mu(R_{\tilde \alpha}^*)=A_{(\alpha)}$ if
$|\alpha|\geq 1$.

 Since   $A,B\in [\cC_\rho]_{<1}$, we have  $r(A)<1$ and
$r(B)<1$. According  to the free pluriharmonic functional calculus,
$P_\rho(A,R)$, $P_\rho(B,R)$, $G(A)$, and $G(B)$ are well-defined
and the corresponding series converge  in the operator norm
topology. Consequently, we have
$$ G(A)=(\text{\rm id}\otimes \mu)(P_1(A, R))\quad   \text{ and
}\quad  G(A)=(\text{\rm id}\otimes \mu)(P_1(A, R)).$$
   Taking into account that $\mu$ is completely positive linear map
with $\mu(I)=I$, we have
\begin{equation*}
\begin{split}
\|G(A)-G(B)\|\leq \|\mu\|\|P_1(A,R)-P_1(B,R)\| = \|P_1(A,R)-P_1(B,
R)\|.
\end{split}
\end{equation*}
The proof is complete.
\end{proof}

  According to Lemma \ref{ine-P},  it makes sense
to define the map $d_K':  [\cC_\infty]_{<1}\times
  [\cC_\infty]_{<1}\to [0,\infty)$ by setting
$$d_K'(A,B):=\sup_u \|u(A)-u(B)\| <\infty,
$$
where the supremum is taken over all free pluriharmonic functions
$u$ on $[B(\cH)^n]_1$ with coefficients in $ B(\cE)$,
  such that $  u(0)=I$ and  $u\geq 0$.

 Using the the free pluriharmonic functional calculus for  for $n$-tuples of operators
 $ (T_1,\ldots, T_n)$ with the joint spectral radius $r(T_1,\ldots, T_n)<1$, one can extend
  Proposition 3.1 from  \cite{Po-hyperbolic} and show  that
for any $A,B\in   [\cC_\infty]_{<1} $,
$$d_K'(A,B)=d_K(A,B),
$$
where $d_K$ is defined by relation \eqref{d_K0}.
  Since  the proof is essentially the
same, we shall omit it.

\begin{proposition}\label{d_K}
 $d_K$  is a metric  on $ [\cC_\infty]_{<1}$
satisfying relation
 \begin{equation*}
 d_K
(A,B)=\|P_1(A,R)-P_1(B,R)\|,\qquad A,B\in   [\cC_\infty]_{<1}.
\end{equation*}
 In addition, the map \ $[0,1)\ni r\mapsto d_K(rA,rB)\in \RR^+$ \
is increasing and
$$
d_K(A,B)=\sup_{r\in [0,1)}d_K(rA,rB).
$$
\end{proposition}
\begin{proof} Using Lemma  \ref{ine-P} we  deduce that  $d_K(A,B)\leq
\|P_1(A,R)-P_1(B,R)\|$. The rest of the proof is similar to that of
Proposition 3.2 from \cite{Po-hyperbolic}, so we shall omit it.
\end{proof}

Now, we can prove the main result of this section.

\begin{theorem}
\label{ker-metric} Let $d_K$ be the Carath\' eodory  metric on $
[\cC_\infty]_{<1} $. Then the following statements hold:
\begin{enumerate}
\item[(i)]  the $d_K$-topology coincides with the norm topology on
$ [\cC_\infty]_{<1}$;
\item[(ii)]  $[\cC_\rho]^{\prec\, 0}$ is
 a $d_K$-closed subset of $ [\cC_\infty]_{<1}$ for any $\rho>0$;
\item[(iii)] the metric $d_K$ is complete on $[\cC_\infty]_{<1}$.
\end{enumerate}
\end{theorem}

\begin{proof} We recall  that
the free pluriharmonic Poisson kernel is given by
\begin{equation*}
 P_1(X,R) = \sum_{k=1}^\infty \sum_{|\alpha|=k}
   X_{  \alpha}\otimes R_{\tilde\alpha}^* + I\otimes I+
  \sum_{k=1}^\infty \sum_{|\alpha|=k}
  X_{\alpha}^* \otimes R_{\tilde \alpha}, \qquad X\in  [\cC_\infty]_{<1},
  \end{equation*}
   where the convergence is in the operator norm topology.
Let  $R_A:=A_1^*\otimes R_1+\cdots +A_n^*\otimes R_n$ be the
reconstruction operator.  Note that, due to the noncommutative von
Neumann inequality, we have
\begin{equation*}
\begin{split}
\|A-B\|&= \|R_A-R_B\|\\
&=\left\|\frac{1}{2\pi} \int_0^{2\pi} e^{it} [P_1(A,e^{it}
R)-P_1(B,e^{it} R)] dt\right\|\\
&\leq \sup_{t\in[0,2\pi]}\left\|P_1(A,e^{it} R)-P_1(B,e^{it}
R)\right\|\\
&\leq \|P_1(A,R)-P_1(B,R)\|.
\end{split}
\end{equation*}
Now,  Proposition  \ref{d_K} implies
\begin{equation}
\label{A-B} \|A-B\|\leq d_K(A,B), \qquad A,B\in
   [\cC_\infty]_{<1},
\end{equation}
which shows that the $d_K$-topology is stronger then the norm
topology on $ [\cC_\infty]_{<1}$. Conversely, to prove that the norm
topology on $ [\cC_\infty]_{<1}$ is stronger than the
$d_K$-topology, note that since $r(R_A)=r(A)<1$ and $r(R_B)=r(B)<1$,
the operators
 $I-R_A$ and $I-R_B$ are  invertible.
 Thus
\begin{equation*}
d_K(A,B)=\|P_1(A,R)-P_1(B,R)\|\leq 2\|(I-R_A)^{-1}-(I-R_B)^{-1}\|
\end{equation*}
 for any $A,B\in
 [\cC_\infty]_{<1}$.
Hence and due to the continuity of the maps $X\mapsto I-R_X$ on
$B(\cH)^n$ and $Y\mapsto Y^{-1}$ on the group of invertible elements
in $B(\cH\otimes F^2(H_n))$, in the operator norm topology, we
deduce our assertion. In conclusion,  the $d_K$-topology coincides
with the norm topology on $ [\cC_\infty]_{<1}$.

 Now, to prove (ii), let
$\{A^{(k)}:=(A_1^{(k)},\ldots, A_n^{(k)})\}_{k=1}^\infty$ be a
$d_K$-Cauchy sequence  in
 $[\cC_\rho]^{\prec\, 0}\subset \cC_\rho$. Due to inequality \eqref{A-B}, we  deduce that
  $\{A^{(k)}\}_{k=1}^\infty$ is a Cauchy
sequence in the norm topology of $B(\cH)^n$. Since $\cC_\rho$ is
closed in the operator norm topology,
 there exists $T:=(T_1,\ldots, T_n)$ in $\cC_\rho$ such that
$\|T-A^{(k)}\|\to 0$, as $k\to\infty$.

Now let us prove that the joint spectral radius $r(T)<1$. Since
$\{A^{(k)}\}_{k=1}^\infty$  is a $d_K$-Cauchy sequence, there exists
$k_0\in \NN$ such that $d_K(A^{(k)},A^{(k_0)})\leq 1$ for any $k\geq
k_0$. On the other hand, since $A^{(k_0)}\in
 [\cC_\rho]^{\prec\, 0}$,  i.e.,  $A^{(k_0)}
\overset{H}{{ \prec}}\, 0$,  Theorem \ref{equivalent} shows that
there is  $c\geq 1$ such that $P_\rho(rA^{(k_0)},R)\leq c^2\delta$
for any $r\in [0,1)$. Hence,  and due to the noncommutative von
Neumann inequality,  we deduce that
\begin{equation}
\label{ine-rho}
\begin{split}
P_\rho(rA^{(k)}, R)&\leq \left( \|P_\rho(rA^{(k)},
R)-P_\rho(rA^{(k_0)}, R)\|
+\|P_\rho(rA^{(k_0)}, R)\|\right)I\\
&\leq \left(  d_K(A^{(k)},  A^{(k_0)}) +\|P_\rho(rA^{(k_0)},
R)\|\right)I \leq (1+c^2\delta)I
\end{split}
\end{equation}
for any $k\geq k_0$ and $r\in [0,1)$.

We show now that $\lim_{k\to\infty}P_\rho(rA^{(k)}, R)=P_\rho(rT,
R)$ in  the operator norm topology. First, one can easily see that,
since $T,A^{(k)}\in \cC_\rho$, we have
\begin{equation*}
\sum_{|\alpha|=p} T_\alpha T_\alpha^*\leq \rho^2 I\quad \text{ and }
\quad  \sum_{|\alpha|=p}A^{(k)}_\alpha A^{(k)}_\alpha\leq \rho^2 I
\end{equation*}
for any $p,k=1,2,\ldots$. Given $\epsilon>0$ and $r\in (0,1)$,  let
$m\in \NN$ be such that $\sum_{p=m}^\infty \rho
r^p<\frac{\epsilon}{2}$. Note that
\begin{equation*}
\begin{split}
&\|P(rA^{(k)}, R)-P(rT,R)\|\\
&\leq 2\sum_{p=1}^{m-1}\left\|\sum_{|\alpha|=p}r^{|\alpha|}
(A_\alpha^{(k)}-T_\alpha)\otimes R_{\tilde \alpha}^*\right\| +
2\sum_{p=m}^{\infty}\left\|\sum_{|\alpha|=p}r^{|\alpha|}
A_\alpha^{(k)}\otimes R_{\tilde \alpha}^*\right\|   +
2\sum_{p=m}^{\infty}\left\|\sum_{|\alpha|=p}r^{|\alpha|}
T_\alpha\otimes R_{\tilde \alpha}^*\right\|\\
&= 2\left\|\sum_{p=1}^{m-1} r^p \sum_{|\alpha|=p}
(A_\alpha^{(k)}-T_\alpha)(A_\alpha^{(k)}-T_\alpha)^*\right\| +
2\left\|\sum_{p=1}^{m-1} r^p \sum_{|\alpha|=p}
A_\alpha^{(k)}{A_\alpha^{(k)}}^*\right\|  +2\left\|\sum_{p=1}^{m-1}
r^p \sum_{|\alpha|=p} T_\alpha {T_\alpha}^*\right\|\\
&\leq 2\left\|\sum_{p=1}^{m-1} r^p \sum_{|\alpha|=p}
(A_\alpha^{(k)}-T_\alpha)(A_\alpha^{(k)}-T_\alpha)^*\right\|
+\epsilon
\end{split}
\end{equation*}
for any $k=1,2,\ldots$. Since $A^{(k)}\to T$  in the norm topology,
  as $k\to\infty$, and using the results above, one can easily deduce that
   $\lim_{k\to\infty}P_\rho(rA^{(k)}, R)=P_\rho(rT,R)$ for each
   $r\in [0,1)$.
Now, taking $k\to\infty$  in  inequality \eqref{ine-rho}, we obtain
$P_\rho(rT, R)\leq (1+c^2 \delta)I$ for $r\in [0,1)$. Applying
Theorem \ref{equivalent}, we deduce that $T \overset{H}{{ \prec}}\,
0$. Now,  Theorem \ref{A<0} implies $r(T)<1$, which shows that $T$
is in $[\cC_\rho]^{\prec\, 0}$ and, therefore, in
$[\cC_\infty]_{<1}$, which  proves part (ii).

It remains to  prove part (iii). To this end, let
$\{A^{(k)}:=(A_1^{(k)},\ldots, A_n^{(k)})\}_{k=1}^\infty$ be a
$d_K$-Cauchy sequence  in
 $[\cC_\infty]_{<1}$. Given $\epsilon>0$, there exists $k_0\geq 1$
 such that  $d_K(A^{(k)}, A^{(j)})<\epsilon $ for any $k,j\geq k_0$.
 Then we have
 \begin{equation}
 \label{c}
 d_K(A^{(k)},0)\leq c:= d_K(A^{(k_0)},0)+\epsilon \quad  \text{ for
 any }\ k\geq k_0.
 \end{equation}
Hence, and due to the definition of $d_K$, we have
$\|u(A^{(k)})-u(0)\|\leq c$ and, consequently,
$$
u(A^{(k)})\leq (\|u(A^{(k)}-u(0)\|+1)I\leq (c+1)u(0)\quad \text{ for
any } \ k\geq k_0
$$
and  for any positive free pluriharmonic function $u$ on
$[B(\cH)^n]_1$ with coefficients in $B(\cE)$  such that $u(0)=I$.

Now,  for each $k\geq k_0$,  fix $\rho_k\geq 1$ such that
$A^{(k)}\in [\cC_{\rho_k}]^{\prec\, 0}$. Note that the inequality
above implies
$$
u(A^{(k)})+(\rho_k-1)u(0)\leq  \rho_k(c+1)u(0)
$$
for all $k\geq k_0$. Applying Theorem \ref{equivalent} and using
relation  \eqref{c}, we obtain
$$
\|L_{0,A^{(k)}}\|^2\leq d_K(A^{(k_0)},0)+\epsilon +1,\qquad k\geq
k_0.
$$
Consequently, we have
\begin{equation}
\label{epsi0} 1\leq \epsilon_0:=\sup_{k\geq k_0} \|L_{0,A^{(k)}}
\|^2<\infty.
\end{equation}
Since $\{A^{(k)}\}$ is a $d_K$-Cauchy sequence, there exists
$m_0\geq k_0$ such  that $d_K(A^{(m')},
A^{(m)})<\frac{1}{2\epsilon_0}$ for any $m,m'\geq m_0$. Using now
relation  \eqref{epsi0}, we obtain
\begin{equation}
\label{DL} d_K(A^{(m)},
A^{(m_0)})<\frac{1}{2\|L_{0,A^{(m_0)}}\|^2},\qquad k\geq m_0.
\end{equation}
Since $A^{(m_0)}\in [\cC_{\rho_{m_0}}]^{\prec\, 0}$, Theorem
\ref{A<0} implies $r(A^{(m_0)})<1$. On the other hand, since
$\lim_{\rho\to \infty}\omega_\rho(A^{(m_0)})=r(A^{(m_0)})<1$, there
exists $\rho_{m_0}>0$ such that $\omega_{\rho_{m_0}}(A^{(m_0)})<1$
for any $\rho\geq \rho_{m_0}$. We can assume that
\begin{equation}
\label{rho-LL} \rho_{m_0}\geq
\frac{\|L_{A^{(m_0)},0}\|^2}{\|L_{0,A^{(m_0)}}\|^2}.
\end{equation}
Using Proposition \ref{d_K} and relation \eqref{DL}, we deduce that
\begin{equation}
\label{PPL2}
 P_{\rho_{m_0}}(A^{(m_0)},R)\leq P_{\rho_{m_0}}(A^{(k)},R)
 +\frac{1}{2\|L_{0,A^{(k)}}\|^2} I,\qquad k\geq m_0.
\end{equation}
On the other hand, since $\omega_{\rho_{m_0}}(A^{(m_0)})<1$, Theorem
\ref{foias2} implies $A^{(m_0)}\overset{H}{\sim}\, 0$ in
$\cC_{\rho_{m_0}}$. Consequently, we have  $0\overset{H}{{ \prec}}\,
A^{(m_0)}$, which due to Theorem \ref{equivalent}, implies
\begin{equation*}
\rho_{m_0} I=P_{\rho_{m_0}}(0,R)\leq \|L_{A^{(m_0)}, 0}\|^2
P_{\rho_{m_0}}(A^{(m_0)},R).
\end{equation*}
Combining this with  relation \eqref{rho-LL}, we get
$$
P_{\rho_{m_0}}(A^{(m_0)},R)\geq \frac{1}{\|L_{0,A^{(m_0)}}\|^2}I.
$$
Hence, and due to \eqref{PPL2}, we have
$$
P_{\rho_{m_0}}(A^{(k)},R)\geq  \frac{1}{2\|L_{0,A^{(m_0)}}\|^2}I\geq
\frac  {1}{2\epsilon_0} I.
$$
Applying  Theorem \ref{foias2}, we deduce that
$A^{(k)}\overset{H}{\sim}\, 0$ and $A^{(k)}\in \cC_{\rho_{m_0}}$.
Therefore, $A^{(k)}\in [\cC_{\rho_{m_0}}]^{\prec\, 0}$ for all
$k\geq m_0$ and the sequence $\{A^{(k)}\}_{k\geq m_0}$ is a
$d_K$-Cauchy sequence  in $[\cC_{\rho_{m_0}}]^{\prec\, 0}$. Due to
part (ii), there exists $A\in [\cC_{\rho_{m_0}}]^{\prec\, 0}\subset
[\cC_\infty]_{<1}$ such that $d_K(A^{(k)}, A)\to 0$, as $k\to
\infty$, which proves that $d_K$ is a complete metric on
$[\cC_\infty]_{<1}$. The proof is complete.
\end{proof}

We can provide now a class of Lipschitz  functions  with respect to
the Carath\' eodory metric on $[\cC_\infty]_{<1}$.

\begin{theorem}\label{Lipsc} Let $f:=(f_1,\ldots, f_m)$ be a
contractive  free holomorphic function with $\|f(0)\|<1$ such that
the boundary functions $\widetilde f_1,\ldots, \widetilde f_m$ are
in the noncommutative disc algebra $\cA_n$. Then
$$
d_K(f(A),f(B))\leq \frac{1+\|f(0)\|}{1-\|f(0)\|} d_K(A,B)
$$
for any  $n$-tuples $A:=(A_1,\ldots, A_n)$ and $B:=(B_1,\ldots,
B_n)$ in $[\cC_\infty]_{<1}$.
\end{theorem}

\begin{proof} According to the maximum principle for free
holomorphic functions with operator-valued coefficients (see Proposition 5.2 from \cite{Po-automorphism}), the
condition $\|f(0)\|<1$ implies that $\|f(X)\|<1$, $X\in
[B(\cH)^n]_1$. If $u$ is a  free pluriharmonic function  on
$[B(\cH)^m]_1$, then Theorem 1.1 from \cite{Po-holomorphic.II} shows
that $u\circ f$ is a free pluriharmonic function on $[B(\cH)^n]_1$.
If, in addition, $u$ is positive, then $u\circ f$ is  also positive.

Assume now that $A$ and $B$ are in $[\cC_\infty]_{<1}$.  Due to
Theorem \ref{radius}, $f(A)$ and $f(B)$ are in $[\cC_\infty]_{<1}$.
Let  $p\in \CC[X_1,\ldots, X_m]\otimes M_k$, $k\in \NN$, be a
matrix-valued noncommutative polynomial with $\Re p(0)=I$ and $\Re
p\geq 0$ on $[B(\cH)^m]_1$. According  to the Harnack type
inequality  \eqref{Rep>0}, we have
\begin{equation*}
  \frac{1-\|f(0)\|}{1+\|f(0)\|}I\leq \Re p(f(0))\leq
 \frac{1+\|f(0)\|}{1-\|f(0)\|}I.
\end{equation*}
Since $\|f(0)\|<1$, we deduce that $\Re p(f(0))$ is a positive
invertible operator of the form $I_\cH\otimes A$ for some $A\in
M_k$. Define the mapping $h:[B(\cH)^n]_1\to
B(\cH)\bar\otimes_{min}M_k$ by setting
$$
h(X):=[\Re p(f(0))]^{-1/2} \Re p(f(X))[\Re p(f(0))]^{-1/2},\qquad
X\in [B(\cH)^n]_1.
$$
Note that $h$ is a positive free pluriharmonic function on $
[B(\cH)^n]_1$ with coefficients in $M_k$  with the property  that
$h(0)=I$. Now, using the above-mentioned Harnack type inequality, we
have
\begin{equation*}
\begin{split}
&\|\Re  p(f(A))-\Re p(f(B))\|\\
&\qquad \leq \|[\Re p(f(0))]^{1/2}\| \left\|[\Re
p(f(0))]^{-1/2}\left(\Re p(f(A))-\Re p(f(B))\right)[\Re
p(f(0))]^{1/2}\right\|\|[\Re p(f(0))]^{1/2}\|\\
&\qquad \leq \|[\Re p(f(0))]\|\|h(A)-h(B)\|\\
&\qquad \frac{1+\|f(0)\|}{1-\|f(0)\|} d_K(A,B).
\end{split}
\end{equation*}
Taking the supremum over all polynomials  $p\in \CC[X_1,\ldots,
X_m]\otimes M_k$, $k\in \NN$, with $\Re p(0)=I$ and $\Re p\geq 0$ on
$[B(\cH)^m]_1$, we obtain
$$
d_K(f(A),f(B))\leq \frac{1+\|f(0)\|}{1-\|f(0)\|} d_K(A,B),
$$
which completes the proof.
\end{proof}

\begin{corollary} Let $f:=(f_1,\ldots, f_m)$ be a
contractive  free holomorphic function with $f(0)=0$  such that the
boundary functions $\widetilde f_1,\ldots, \widetilde f_m$ are in
the noncommutative disc algebra $\cA_n$. Then
$$
d_K(f(A),f(B))\leq  d_K(A,B)
$$
for any $A,B\in [\cC_\infty]_{<1}$.
\end{corollary}

We remark that, using Corollary \ref{T,0} and the remarks preceding
Corollary \ref{T,02}, one can easily obtain
 the following result, which  provides a simple example when
  the inequality of Theorem \ref{Lipsc} is   an equality.

\begin{corollary} \label{T,04}
 If $1\leq m<n$, let  $ A:=(A_1,\ldots, A_m)\in B(\cH)^m$ and $
B:=(B_1,\ldots, B_m)\in B(\cH)^m$ be
 in  $[\cC_\infty]_{<1}$ and   let  $\widetilde A:=(A_1,\ldots, A_m,0,\ldots, 0)$
and $\widetilde B:=(B_1,\ldots, B_m,0,\ldots, 0)$ be their
extensions    in $B(\cH)^n$, respectively.
   Then
$$
d_K(A,B)=d_K(\widetilde A, \widetilde B).
$$
\end{corollary}

 According to Theorem \ref{ker-metric},    the $d_K$-topology coincides with the norm
topology on $ [\cC_\infty]_{<1}$. Due to Theorem \ref{Lipsc}, we
deduce the following result.

\begin{corollary}  Let $f:=(f_1,\ldots, f_m)$ be a
contractive  free holomorphic function with $\|f(0)\|<1$ such that
the boundary functions $\widetilde f_1,\ldots, \widetilde f_m$ are
in the noncommutative disc algebra $\cA_n$. Then the map
$$
[\cC_\infty]_{<1}\ni (T_1,\ldots, T_n)\mapsto f(T_1,\ldots, T_n)\in
[\cC_\infty]_{<1}
$$
is continuous in the operator norm topology, where
$[\cC_\infty]_{<1}$ is the corresponding ball in $B(\cH)^n$ and
$B(\cH)^m$, respectively.
\end{corollary}

\bigskip

\section{Three metric topologies on  Harnack parts of $\cC_\rho$}

In this section we study the relation between the
$\delta_\rho$-topology, the $d_K$-topology, and the operator norm
topology on  Harnack parts of $\cC_\rho$. We prove that the
hyperbolic metric  $\delta_\rho$ is a complete metric on certain
Harnack parts of $\cC_\rho$, and that all the three  topologies
coincide on $[\cC_\rho]_{<1}$. In particular, we  prove that  the
hyperbolic metric $\delta_\rho$ is complete on the open unit unit
ball $[\cC_\rho]_{<1}$, while the other two metrics are not
complete.

First, we mention   another formula for the hyperbolic distance that
will be used to prove the main result of this section.
  If $f\in \cA_n\bar \otimes_{min} M_m$, $m\in
\NN$, then  we call  $\Re f$  strictly positive and denote $\Re f>0$
if there exists a constant $a>0$ such that $\Re \,f\geq aI$. We
remark that, in this case, if
 $(T_1,\ldots, T_n)\in \cC_\rho$, then, using the  functional calculus
 for the class $\cC_\rho$, we deduce that
 $$
 \Re f(T_1,\ldots, T_n)+(\rho
-1)\Re f(0)\geq \rho aI.
$$

The proof of the next result is similar to that of Proposition  3.5
from \cite{Po-hyperbolic}, but uses  the  functional calculus
 for the class $\cC_\rho$ and Theorem \ref{equivalent} of the
present paper. We shall omit it.

\begin{proposition}\label{delta-form}
Let   $A:=(A_1,\ldots, A_n)$ and $B:=(B_1,\ldots, B_n)$ be   in $
\cC_\rho$ such that $A\overset{H}{\sim}\, B$. Then

\begin{equation}
\label{de-sup} \delta_\rho(A,B)=\frac{1}{2}\sup\left|
\ln\frac{\left<[\Re f(A_1,\ldots, A_n)+(\rho -1)\Re
f(0)]x,x\right>}{\left<[\Re f(B_1,\ldots, B_n)+(\rho -1)\Re
f(0)]x,x\right>}\right|,
\end{equation}
where the supremum is taken over all $f\in \cA_n\otimes M_m$, $m\in
\NN$, with $\Re f>0$ and $x\in \cH\otimes \CC^m$ with $x\neq 0$.

\end{proposition}

We remark that, under the conditions of Proposition
\ref{delta-form}, one can  also prove that relation \eqref{de-sup}
holds if the supremum is taken over all noncommutative polynomials
  $f\in
\CC[X_1,\ldots, X_n]\otimes M_m$, $m\in \NN$, with $\Re  f> 0$, and
$x\in \cH\otimes \CC^m$ with $x\neq 0$.

The main result of this section is the following.

 \begin{theorem}
\label{topology} Let  $\delta_\rho$, $\rho>0$,  be the  hyperbolic
metric on
  a Harnack part $\Delta$ of
 $[\cC_\rho]^{\prec\, 0}$. Then  the following properties hold:
 \begin{enumerate}
 \item[(i)]
 $\delta_\rho$ is complete on
 $\Delta$;
\item[(ii)]
the $\delta_\rho$-topology is stronger then the $d_K$-topology on
$\Delta$;
\item[(iii)]
the $\delta_\rho$-topology, the $d_K$-topology, and the operator
norm topology coincide on  $[\cC_\rho]_{<1}$;
\item[(iv)] $[\cC_\rho]_{<1}$ is complete relative the  hyperbolic
metric, but not complete with respect to the  Carath\' eodory metric
$d_K$ and the operator metric.
\end{enumerate}
\end{theorem}

\begin{proof} Let   $A:=(A_1,\ldots, A_n)$ and $B:=(B_1,\ldots, B_n)$
be $n$-tuples  in a Harnack part $\Delta$ of  $[\cC_\rho]^{\prec\,
0}$. Then $A$ is Harnack equivalent to $B$ and \begin{equation*} \Re
f(A_1,\ldots, A_n)+(\rho -1)\Re f(0)\leq \Lambda_\rho(A,B)^2[\Re
f(B_1,\ldots, B_n)+(\rho -1)\Re f(0)]
\end{equation*}
 for any $f\in \cA_n\bar \otimes_{min} M_m$  with $\Re f\geq 0$, where $\Lambda_\rho(A,B)$ is defined by
 \eqref{La}.
 Hence, we deduce that
 \begin{equation}
 \label{Re-omega}
\Re f(A_1,\ldots, A_n)-\Re f(B_1,\ldots, B_n)\leq [\Lambda_\rho
(A,B)^2-1][\Re  f(B_1,\ldots, B_n)+(\rho -1)\Re f(0)].
\end{equation}
Since $B\overset{H}{{ \prec}}\, 0$, we have  the joint spectral
radius $r(B)<1$, so the $\rho$-pluriharmonic kernel  $P_\rho(B,R)$
makes sense. Due to the fact that the noncommutative Poisson
transform $\text{\rm id}\otimes P_{rR}$ is completely positive, and
$P_\rho(B,S)\leq \|P_\rho(B,R)\| I$, one can easily see that
\begin{equation*}
\begin{split}
P_\rho(rB,R)&=(\text{\rm id}\otimes P_{rR})[P_\rho(B, S)]\leq \|P_\rho(B,R)\| I\\
&=\frac{1}{\rho}\|P_\rho(B,R)\| P_\rho(0,R)
\end{split}
\end{equation*}
for any $r\in [0,1)$. Using the equivalence $(ii)\leftrightarrow
(iii)$ of Theorem \ref{equivalent}, when
$c^2=\frac{1}{\rho}\|P_\rho(B,R)\|$,   we obtain $ \Re
f(rB_1,\ldots, rB_n)+(\rho -1)\Re f(0)\leq \|P_\rho(B,R)\|\Re  f(0)
$ for any $r\in [0,1)$. Letting   $r\to 1$, in the operator norm
topology, we deduce that
\begin{equation*}
 \Re f(B_1,\ldots, B_n)+(\rho -1)\Re f(0)\leq \|P_\rho(B,R)\|\Re f(0).
\end{equation*}
Hence, and using relation  \eqref{Re-omega}, we obtain
$$
\Re f(A_1,\ldots, A_n)-\Re f(B_1,\ldots, B_n)\leq
[\Lambda_\rho(A,B)^2-1]\|P_\rho(B,R)\|\Re  f(0).
$$
We can obtain a similar inequality  if we interchange $A$ with $B$.
If, in addition, we assume that $\Re f(0)=I$, then we obtain
$$
-tI\leq \Re f(A_1,\ldots, A_n)-\Re f(B_1,\ldots, B_n)\leq tI,
$$
where
$t:=[\Lambda_\rho(A,B)^2-1]\max\{\|P_\rho(A,R)\|,\|P_\rho(B,R)\|\}$.
Since   $\Re f(A_1,\ldots, A_n)-\Re f(B_1,\ldots, B_n)$ is a
self-adjoint operator, we get $\| \Re f(A_1,\ldots, A_n)-\Re
f(B_1,\ldots, B_n)\|\leq t$. Hence, we deduce that
 $d_K(A,B)\leq s$.  As a consequence, we obtain
 \begin{equation}
 \label{ine-dh}
 d_K (A,B)\leq \max\{\|P_\rho(A,R)\|,\|P_\rho(B,R)\|\}
 \left(e^{2\delta_\rho(A,B)}-1\right).
\end{equation}

Let us prove that $\delta_\rho$ is a complete metric on $\Delta$. To
this end, let $\{A^{(k)}:=(A_1^{(k)},\ldots,
A_n^{(k)})\}_{k=1}^\infty\subset \Delta$ be a $\delta_\rho$-Cauchy
sequence. First, we prove that the sequence $\{\|P_\rho(A^{(k)},
R)\|\}_{k=1}^\infty$ is bounded.  Given $\epsilon>0$, there exists
$k_0\in \NN$ such that
\begin{equation}
\label{de} \delta_\rho(A^{(k)}, A^{(p)})<\epsilon\quad \text{ for
any } \ k,p\geq k_0.
\end{equation}

Let $f\in \cA_n\bar \otimes_{min} M_m$  with $\text{\rm Re}\,f\geq
0$. Since $A^{(k_0)}\overset{H}{{ \prec}}\, 0$ and
$$
P_\rho (rA^{(k_0)}, R)\leq \frac{1}{\rho} \|P_\rho (rA^{(k_0)}, R)\|
P_\rho(0,R),
$$
Theorem \ref{equivalent} implies
$$
 \Re
f(A^{(k_0)})+(\rho -1)\Re f(0)\leq \frac{1}{\rho} \|P_\rho
(rA^{(k_0)}, R)\|[{\Re f(0)}+(\rho -1)\Re f(0)].
$$
On the other hand, since
 $A^{(k)}\overset{H}{\sim}\, A^{(k_0)}$, Theorem  \ref{equivalent} implies
 $$
\Re f(A^{(k)})+(\rho -1)\Re f(0)\leq \Lambda_\rho(A^{(k)},
A^{(k_0)})^2[\Re f(A^{(k_0)})+(\rho -1)\Re f(0)].
$$
Combining these inequalities, we obtain
\begin{equation}
\label{an-ine} \begin{split}
 \Re f(A^{(k)})+(\rho
-1)\Re f(0) \leq c^2\frac{1}{\rho}[\Re f(0)+(\rho -1)\Re f(0)],
\end{split}
\end{equation}
where  $c:=\|P_\rho(A^{(k_0)},R)\|^{1/2}\Lambda_\rho (A^{(k)},
A^{(k_0)})$, for any   $f\in \cA_n\otimes M_m$  with $\Re f\geq 0$.

 Consequently, due to Theorem \ref{equivalent}, we have
  $ \|P_\rho(A^{(k)},R)\|\leq c^2$ for any $k\geq k_0$. Combining
this with  relation \eqref{de}, we obtain

$$
\|P_\rho(A^{(k)},R)\|\leq \|P_\rho(A^{(k_0)},R)\| e^{2\epsilon}
$$
for any $k\geq k_0$. This shows that the sequence
$\{\|P_\rho(A^{(k)}, R)\|\}_{k=1}^\infty$ is bounded. Consequently,
 inequality \eqref{ine-dh} implies that $\{A^{(k)}\}$ is a
$d_K$-Cauchy sequence. Due to Theorem \ref{ker-metric}, there exists
$A:=(A_1,\ldots, A_n)\in [\cC_\rho]^{\prec\, 0}$ such that
\begin{equation}
\label{conv1} d_K(A^{(k)}, A)\to 0\quad \text{ as } \ k\to \infty.
\end{equation}
In what follows, we prove that $A\in \Delta$.
 Let $f\in \cA_n\otimes
M_m$ with $\Re f\geq 0$ and $\Re f(0)=I$. Taking into account
relations \eqref{an-ine} and \eqref{de}, we have
\begin{equation}
\begin{split}
\label{Re-ine} \Re f(A^{(k)})+(\rho -1)\Re f(0)&\leq \Lambda_\rho
(A^{(k)}, A^{(k_0)})^2 [\Re
f(A^{(k_0)})+(\rho -1)\Re f(0)]\\
&\leq e^{2\epsilon}[\Re f(A^{(k_0)})+(\rho -1)\Re f(0)]
\end{split}
\end{equation}
for $k\geq k_0$. According to  relation \eqref{conv1} and the
definition of $d_K$, $\Re f(A^{(k)})\to \Re f(A)$, as $k\to\infty$,
in the operator norm topology. Consequently, relation \eqref{Re-ine}
implies
\begin{equation}
\label{ine-Re} \Re f(A)+(\rho -1)\Re f(0)\leq e^{2\epsilon}[\Re
f(A^{(k_0)})+(\rho -1)\Re f(0)].
\end{equation}

Such an  inequality can be deduced in  the more general case when
$f\in \cA_n\otimes M_m$  with $\Re f\geq 0$. Indeed, for each
$\epsilon'>0$ let $g:=\epsilon'I+f$, $Y:=\Re g(0)$, and
$\varphi:=Y^{-1/2} gY^{-1/2}$. Since
 $\Re \varphi\geq 0$ and $\Re
\varphi(0)=I$, we can apply inequality \eqref{ine-Re} to $\varphi$
and deduce that
$$
\rho\epsilon' I+\Re  f(A)+(\rho -1)\Re f(0)\leq
e^{2\epsilon}\left[\rho\epsilon'  I+\Re f(A^{(k_0)})+(\rho -1)\Re
f(0)\right]
$$
for any $\epsilon'>0$. Letting $\epsilon'\to 0$, we get

\begin{equation}
\label{ine-Re2} \Re f(A)+(\rho -1)\Re f(0)\leq e^{2\epsilon}[\Re
f(A^{(k_0)})+(\rho -1)\Re f(0)]
\end{equation}
for any $f\in \cA_n\otimes M_m$  with $\Re f\geq 0$. Therefore,
\begin{equation}
\label{<1} A\overset{H}{{ \prec}}\, A^{(k_0)}.
\end{equation}
On the other hand, since $A^{(k_0)}\overset{H}{{ \prec}}\, A^{(k)}$
for any $k\geq k_0$, Theorem \ref{equivalent} and relation
\eqref{de}, imply
\begin{equation*}
\begin{split}
\Re p(A^{(k_0)})+(\rho -1)\Re p(0)&\leq \Lambda_\rho (A^{(k_0)},
A^{(k)})^2 [\Re
p(A^{(k)})+(\rho -1)\Re p(0)]\\
&\leq e^{2\epsilon}[\Re  p(A^{(k)})+(\rho -1)\Re (0)]
\end{split}
\end{equation*}
for $k\geq k_0$ and any polynomial $p\in \CC[X_1,\ldots, X_n]\otimes
M_m$, $m\in \NN$, with $\Re p\geq 0$. According to Theorem
\ref{ker-metric}, the $d_K$-topology coincides with the norm
topology on $[\cC_\rho]^{\prec\, 0}$. Therefore, relation
\eqref{conv1} implies $A^{(k)}\to A\in [\cC_\rho]^{\prec\, 0}$ in
the operator norm topology. Taking  the limit, as $k\to\infty$,  in
the inequality above, we deduce that
\begin{equation}
\label{Re3} \Re p(A^{(k_0)})+(\rho -1)\Re p(0)\leq e^{2\epsilon}[\Re
p(A )+(\rho -1)\Re p(0)]
\end{equation}
for any  $p\in \CC[X_1,\ldots, X_n]\otimes M_m$  with $\Re p\geq 0$.
Consequently, we get $A^{(k_0)}\overset{H}{{ \prec}}\, A$. Hence,
and using  relation \eqref{<1}, we obtain $A\overset{H}{{ \sim}}\,
A^{(k_0)}$, which proves that $A\in \Delta$. The inequalities
\eqref{ine-Re2} and \eqref{Re3} imply  $\Lambda_\rho(A^{(k_0)},
A)\leq e^{2\epsilon}$. This shows that  $\delta_\rho(A^{(k_0)},
A)<\epsilon$, which together with  relation \eqref{de} imply
$\delta_\rho(A^{(k)}, A)<2\epsilon$ for any $k\geq k_0$. Therefore,
$\delta_\rho(A^{(k)}, A)\to 0$, as $k\to \infty$, which proves that
$\delta_\rho$ is a complete metric on the Harnack part  $\Delta$.
Note that we have also proved part (ii) of this theorem.

In what follows, we prove part (iii). To this end, assume that $A$
and $B$ are $n$-tuples of operators in $ [\cC_\rho]_{<1}$. Due to
Theorem \ref{foias2}, $P_\rho(B,R)$ is a positive invertible
operator. Since $P_\rho(B,R)^{-1}\leq \| P_\rho(B,R)^{-1}\|$, we
have $I\leq \| P_\rho(B,R)^{-1}\| P_\rho(B,R)$, which,  applying the
noncommutative Poisson transform, implies $I\leq \|
P_\rho(B,R)^{-1}\| P_\rho(rB,R)$ for any $r\in [0,1)$. By Theorem
\ref{equivalent}, we deduce that $ 0\overset{H}{{ \prec}}\, B$ and
$$
\Re f(0)\leq \| P_\rho (B,R)^{-1}\|\, [\Re  f(B)+(\rho -1)\Re f(0)]
$$
for any $f\in \cA_n\otimes M_m$  with $\Re f\geq 0$. If, in
addition,  $\Re f(0)=I$, then the latter inequality implies
\begin{equation*}
\begin{split}
\frac{\left<[\Re  f(A)+(\rho -1)\Re f(0)]x,x\right>}{\left<[\Re
f(B)+(\rho -1)\Re f(0)]x,x\right>}-1 &\leq
\frac{\|P_\rho(B,R)^{-1}\|}{\|x\|}\left<\left(\Re  f(A)-\Re
f(B)\right)x,x\right>\\
&\leq \|P_\rho(B,R)^{-1}\| d_K(A,B)
\end{split}
\end{equation*}
for any $x\in \cH\otimes \CC^m$, $x\neq 0$. Consequently, we have
$$
\ln\frac{\left<[\Re  f(A)+(\rho -1)\Re f(0)]x,x\right>}{\left<[\Re
f(B)+(\rho -1)\Re f(0)]x,x\right>}\leq \ln
\left(1+\|P_\rho(B,R)^{-1}\| d_K(A,B)\right).
$$
A  similar inequality   can be obtained interchanging $A$ with $B$.
Combining these two inequalities, we get
\begin{equation}
\label{ln1} \left|\ln\frac{\left<[\Re f(A)+(\rho -1)\Re
f(0)]x,x\right>}{\left<[\Re f(B)+(\rho -1)\Re
f(0)]x,x\right>}\right|\leq \ln \left(1+\max\{\|P_\rho(B,R)^{-1}\|,
\|P_\rho(A,R)^{-1}\|\} d_K(A,B)\right).
\end{equation}
Now,  we consider   the general case when $g\in \cA_n\otimes M_m$
with $\Re g> 0$.  Note that   $Y:=\Re  g(0)$ is a positive
invertible operator on $\cH\otimes \CC^m$  and $f:=Y^{-1/2}
gY^{-1/2}$ has the properties $\Re f\geq 0$ and $\Re f(0)=I$.
Applying inequality \eqref{ln1} to $f$ when $x:=Y^{-1/2}y$, $y\in
\cH\otimes \CC^m$, and $y\neq 0$, we obtain
\begin{equation}
\label{2de} 2\delta_\rho(A,B)\leq  \ln
\left(1+\max\{\|P_\rho(B,R)^{-1}\|, \|P_\rho(A,R)^{-1}\|\}
d_K(A,B)\right).
\end{equation}

Consider a sequence  $\{A^{(k)}\}_{k=1}^\infty$  of elements in
$[\cC_\rho]_{<1}$ and let $A \in [\cC_\rho]_{<1}$ be such that
$d_K(A^{(k)}, A)\to 0$, as $k\to \infty$. By Proposition \ref{d_K},
we deduce that $P_\rho(A^{(k)},R)\to P_\rho(A,R)$ in the operator
norm topology.  On the other hand,  due to Theorem \ref{foias2}, the
operators $P(A^{(k)},R)$ and $ P(A,R)$ are invertible. Hence, and
using the well-known fact that the map $Z\mapsto Z^{-1}$ is
continuous on the open set of all invertible operators, we deduce
that $P_\rho(A^{(k)},R)^{-1}\to P_\rho(A,R)^{-1}$ in the operator
norm topology, as $k\to\infty$. Hence, we deduce that the sequence
$\{\|P_\rho(A^{(k)},R)^{-1}\|\}_{k=1}^\infty$ is bounded.
Consequently, there exists $M>0$ with
$\|P_\rho(A^{(k)},R)^{-1}\|\leq M$ for any $k\in \NN$. Using
inequality \eqref{2de}, we obtain
$$
 2\delta_\rho(A^{(k)},A)\leq  \ln \left(1+ M d_K(A^{(k)},A)\right),\qquad
  k\in \NN.
 $$
Since $d_K(A^{(k)}, A)\to 0$, as $k\to \infty$,  the latter
inequality implies  that $\delta_\rho (A^{(k)},A)\to 0$. Therefore,
the $d_K$-topology on $[\cC_\rho]_{<1}$ is stronger than the
$\delta_\rho$-topology. Due to the first part of this theorem, the
two topologies coincide on $[\cC_\rho]_{<1}$. Using now Theorem
\ref{ker-metric}, we complete the proof of part (iii).

Now, we prove item (iv). Since $[\cC_\rho]_{<1}$ is the Harnack part
of $0$ (see Theorem \ref{foias2}), part (i) implies its completeness
with respect to the  hyperbolic metric. To prove that
$[\cC_\rho]_{<1}$ is not complete with respect to the  Carath\'
eodory metric $d_K$ and the operator metric, we consider the
following example. Let $(T_1,\ldots, T_n)\in B(\cP_1)^{n}$ be the
$n$-tuple of
 operators defined by $T_i:=P_{\cP_1} S_i|_{\cP_1}$, $i=1,\ldots,n$,
 where $\cP_1:=\text{\rm span} \{e_\alpha: \ |\alpha|\leq 1\}$.
Note that $\|[T_1,\ldots, T_n]\|=1$ and $T_\alpha=0$  for any
$\alpha\in \FF_n^+$ with  $|\alpha|\geq 2$. Set $X_i:=\rho T_i$, \
$i=1,\ldots, n$, and note that
$$X_\beta=\rho T_\beta=\rho P_{\cP_1} S_\beta|_{\cP_1},\qquad
 \beta\in \FF_n^+\backslash \{g_0\}.$$
 Therefore, $(X_1,\ldots, X_n)\in \cC_\rho$, i.e.,
$\omega_\rho(X_1,\ldots, X_n)\leq 1$, which implies
$\omega_{\rho}(T_1,\ldots, T_n)\leq \frac {1} {\rho}$. The reverse
inequality  is due to the fact that $
  \|[T_1,\ldots, T_n]\|\leq \rho \omega_\rho(T_1,\ldots, T_n).
  $
Consequently, we have
$$\omega_\rho(T_1,\ldots, T_n)=\frac{1}{\rho}, \quad \text{ for }  \rho\in
(0,\infty).
$$
On other hand,  the condition $T_\alpha=0$   if   $|\alpha|\geq 2$
 implies
 $r(T_1,\ldots, T_n)=0$.
 Therefore, we have
 $$
 \omega_\rho (X_1,\ldots, X_n)=1 \quad \text{ and } \quad r(X_1,\ldots,
 X_n)=0.
 $$

Now, let $c\in(0,1)$ and define $Y^{(k)}:= c^{1/k}(X_1,\ldots, X_n)$
for $k=1,2,\ldots.$  Since  $\omega_\rho (Y^{(k)})=c^{1/n}<1$,
Theorem \ref{foias2} implies $Y^{(k)}\overset{H}{\sim}\, 0$ in
$\cC_\rho$ and  $Y^{(k)}\in [\cC_\rho]_{<1}$. On the other hand,
since $\omega_\rho (X_1,\ldots, X_n)=1$, we have $X:=(X_1,\ldots,
X_n)\notin [\cC_\rho]_{<1}$. Now, note that
\begin{equation*}
\begin{split}
d_K(Y^{(k)}, X)&\leq 2\|(I-R_{Y^{(k)}})^{-1}-(I-R_X)^{-1}\|\\
&=2\|R_{Y^{(k)}}- R_X\|=2\|Y^{(k)}-X\|
=2\|X\|(1-c^{1/k}).
\end{split}
\end{equation*}
Consequently,  $Y^{(k)}\to X$ in the operator norm and $d_K(Y^{(k)},
X)\to 0$, as $k\to\infty$. This shows that $[\cC_\rho]_{<1}$ is
  not complete with
respect to the  Carath\' eodory metric $d_K$ and the operator
metric. The proof is complete.
\end{proof}

\begin{corollary} Let  $\delta_\rho$ be the  hyperbolic metric on
  a Harnack part $\Delta$ of
 $[\cC_\rho]^{\prec\, 0}$. Then
 $$
d_K (A,B)\leq \max\{\|P_\rho(A,R)\|,\|P_\rho(B,R)\|\}
 \left(e^{2\delta_\rho(A,B)}-1\right),\qquad A,B\in \Delta.
 $$
If, in addition  $A,B\in [\cC_\rho]_{<1}$, then
$$
2\delta_\rho(A,B)\leq  \ln \left(1+\max\{\|P_\rho(B,R)^{-1}\|,
\|P_\rho(A,R)^{-1}\|\} d_K(A,B)\right).
$$
\end{corollary}

\begin{corollary}  Let $f:=(f_1,\ldots, f_m)$ be a
contractive  free holomorphic function with $\|f(0)\|<1$ such that
the boundary functions $\widetilde f_1,\ldots, \widetilde f_m$ are
in the noncommutative disc algebra $\cA_n$. If $\Delta$ is  a
Harnack part
 of
 $[\cC_\rho]^{\prec\, 0}$, then the map
$$
 \Delta\ni (T_1,\ldots, T_n)\mapsto f(T_1,\ldots, T_n)\in [\cC_{\rho_f}]^{\prec\, 0}
$$
is continuous  with respect to the  hyperbolic metric $\delta_\rho$
on $\Delta$ and  the Carath\' eodory metric $d_K$ on
$[\cC_{\rho_f}]^{\prec\, 0}$, where $\rho_f$ is defined by relation
\eqref{rof}. In particular, tha map
$$
  [\cC_\rho]_{<1}\ni (T_1,\ldots, T_n)\mapsto f(T_1,\ldots, T_n)\in
  [\cC_{\rho_f}]_{<1}
$$
is continuous  with respect to the  hyperbolic metric.
\end{corollary}

\section{ Harnack domination and  hyperbolic metric for $\rho$-contractions (case $n=1$)}

In this section,  we consider the single variable case ($n=1$) and
show that  our Harnack domination  of $\rho$-contractions is
equivalent to the one introduced and studied  by Cassier and Suciu
in \cite{CaSu}. We recover some of their results and  obtain some
results which seem to be new  even in the single variable case.

In the particular case when $n=1$,   the free pluriharmonic Poisson
kernel $P_\rho(rY,R)$, $r\in [0,1)$,  coincides with
$$
Q_\rho(rY, U):=\sum_{k=1} r^k{Y^*}^k\otimes U^k + \rho I\otimes I+
\sum_{k=1}^\infty
 r^kY^k\otimes {U^*}^k,\qquad  Y\in \cC_\rho\subset B(\cH),
$$
where the convergence of the series is in the operator norm topology
and $U$ is the  unilateral shift acting on the Hardy space
$H^2(\TT)$. For each $\rho$-contraction $T\in B(\cH)$, consider the
operator-valued Poisson kernel  defined by
$$
K_\rho(z,T):= \sum_{k=1}^\infty z^k {T^*}^k+\rho I+\sum_{k=1}^\infty
\bar z^k T^k,\qquad z\in \DD,
$$
which was employed by Cassier and Fack in \cite{Ca}.
Using Theorem \ref{equivalent}, in the particular case when $n=1$,
we can prove  the following result.
\begin{proposition}\label{n=1}
Let $ T$ and $ T'$ be  two $\rho$-contractions in $B(\cH)$ and let
$c\geq 1$. Then the following statements are equivalent:
\begin{enumerate}
\item[(i)] $T\overset{H}{{\underset{c}\prec}}\, T'$;
\item[(ii)] $Q_\rho(rT,U)\leq c^2 Q_\rho(rT', U)$ for any $r\in [0,1)$;
\item[(iii)] $K_\rho(z,T)\leq c^2K_\rho(z,T')$ for any $z\in \DD$.
\end{enumerate}
\end{proposition}

\begin{proof} The equivalence $(i)\leftrightarrow (ii)$ follows from
Theorem \ref{equivalent},   when $n=1$. To prove the  implication
$(ii)\implies (iii)$, we apply the noncommutative Poisson transform
(when $n=1$) at $e^{it}I$
 to the inequality  of part $(ii)$.  Consequently, we obtain
 $$
 K_\rho(re^{it}, T)=(\text{\rm id}\otimes P_{e^{it}I} )[Q_\rho(rT, U)]\leq
 c^2( \text{\rm id}\otimes P_{e^{it}I} )[Q_\rho(rT', U)]= c^2 K_\rho(re^{it}, T')
 $$
for any $r\in[0,1)$ and $t\in \RR$. Now let us prove that
$(iii)\implies (ii)$. Since
$$
\left<( {T^*}^k\otimes U^k)( h_m \otimes e^{imt}), h_p\otimes
e^{ipt}\right>_{\cH\otimes H^2(\TT)}= \frac{1}{2\pi} \int_{-\pi}^\pi
\left<e^{ikt} {T^*}^k(e^{imt} h_m), e^{ipt} h_p\right>_\cH dt
$$
for any $h_m, h_p\in \cH$  and  $k,m,p\in \NN$, one can easily
obtain
\begin{equation*}
\left<\left(c^2 Q_\rho(rT',U)-Q_\rho(rT,U)\right) h(e^{it}),
h(e^{it})\right>_{\cH\otimes H^2(\TT)}= \frac{1}{2\pi}
\int_{-\pi}^\pi
\left<\left(c^2K_\rho(re^{it},T')-K_\rho(re^{it},T)\right)h(e^{it}),
h(e^{it})\right>_{\cH}
\end{equation*}
for any function $e^{it}\mapsto h(e^{it})$ in $\cH\otimes H^2(\TT)$.
Now, the implication  $(iii)\implies (ii)$ is clear. The proof is
complete.
\end{proof}

 Let $T,T'\in B(\cH)$
be  $\rho$-contractions
    such that
$T\overset{H}{\prec}\, T'$.  Due to  Proposition \ref{n=1} and
Corollary \ref{LBA-inf}, we  deduce that
\begin{equation*}
\begin{split}
\|L_{T',T}\|&=\inf \{ c> 1: \ Q_\rho(rT,U)\leq c^2 Q_\rho(rT', U)\
\text{ for any } \ r\in
[0,1)\}\\
&=
 \inf\{c> 1: \ K_\rho(z,T)\leq c^2K_\rho(z,T')\ \text{ for
any } \ z\in \DD\}\\
&=\inf\{c> 1: \ K_\rho(z,T^*)\leq c^2K_\rho(z,{T'}^*)\ \text{ for
any } \ z\in \DD\} =\|L_{{T'}^*,T^*}\|.
\end{split}
\end{equation*}
Therefore $T\overset{H}{\prec}\, T'$ if and only if
$T^*\overset{H}{\prec}\, {T'}^*$.

\begin{theorem} \label{suciu}Let $T,T'\in B(\cH)$ be such that $T,T'\in
[\cC_\rho]_{<1}$. Then
$$
\|L_{T',T}\|=\sup_{z\in \DD}\| \Delta_{\rho, {T'}^*}(z)^{-1}(I-\bar
z {T'}^*)(I-\bar zT^*)^{-1} \Delta_{\rho, T^*}(z)\|,
$$
where
$$\Delta_{\rho, T}(z):= [\rho I+(1-\rho)(zT^*+\bar zT)+
(\rho-2)TT^*]^{1/2},\qquad z\in \DD.
$$
Moreover,
\begin{equation*}
\begin{split}
\delta_\rho(T,T')&=
  \ln \max \left\{ \left\| L_{T,T'} \right\|,
  \left\| L_{T',T}\right\|\right\}.
\end{split}
\end{equation*}
\end{theorem}
\begin{proof}
 If $T,T'\in [\cC_\rho]_{<1}$,  Theorem \ref{LCC}
implies
\begin{equation*}
\begin{split}
\|L_{T',T}\|&=\|L_{{T'}^*, T^*}\| = \sup_{z\in \DD}\| \Delta_{\rho,
T^*}(z) (I-zT)^{-1}(I-z
T') \Delta_{\rho, {T'}^*}(z)^{-1}\|\\
&=\sup_{z\in \DD}\| \Delta_{\rho, {T'}^*}(z)^{-1}(I-\bar z
{T'}^*)(I-\bar zT^*)^{-1} \Delta_{\rho, T^*}(z)\|.
\end{split}
\end{equation*}
Using now Theorem \ref{P-B}, we complete the proof.
\end{proof}
We mention  that when $\rho=1$, we recover a result obtained by I.
Suciu \cite{Su2}, using different methods. However, if $\rho>0$ and
$\rho\neq 1$, the result of Theorem \ref{suciu} seems to be new. We
also  remark that Proposition \ref{invariant-curve} , Proposition
\ref{d_K}, and part (i) of Theorem \ref{ker-metric} are new even in
the single variable case ($n=1$).

The next result makes an interesting connection between the  Harnack
domination  for $n$-tuples of operators in $\cC_\rho$ and   and the
Harnack domination for $\rho$-contractions ($n=1$), via the
reconstruction operator.

\begin{theorem}
\label{equivalent4} Let $A:=(A_1,\ldots, A_n)$ and $B:=(B_1,\ldots,
B_n)$ be in $\cC_\rho$ and let $c>0$. Then the following statements
are equivalent:
\begin{enumerate}
\item[(i)]
$A\overset{H}{{\underset{c}\prec}}\, B$;
\item[(ii)] $R_A\overset{H}{{\underset{c}\prec}}\, R_B$, where
$R_X:=  X_1^*\otimes R_1+\cdots +   X_n^*\otimes R_n$ is the
reconstruction operator  associated with $X:=(X_1,\ldots, X_n)\in
\cC_\rho$ and the right creation operators $ R_1,\ldots, R_n$.
\item[(iii)] $R_A^*\overset{H}{{\underset{c}\prec}}\, R_B^*$.
\end{enumerate}
\end{theorem}
\begin{proof}
First, assume that item (i) holds. Due to Theorem \ref{equivalent},
we have
\begin{equation}
\label{P<P} P_\rho(rA, S)\leq c^2P_\rho(rB,S)
\end{equation}
for any $r\in [0,1)$, where $S:=(S_1,\ldots, S_n)$ is the $n$-tuple
of left creation operators. Let $U$ be the unilateral shift on the
Hardy space $H^2(\TT)$. Since $R_i^* R_j=\delta_{ij} I$, the
$n$-tuple $( R_1\otimes U^*,\ldots,  R_n\otimes U^*)$ is a row
contraction acting from $[ F^2(H_n)\otimes H^2(\TT)]^{n}$ to $
F^2(H_n)\otimes H^2(\TT)$. Applying  the noncommutative Poisson
transform at $( R_1\otimes U^*,\ldots, R_n\otimes U^*)$  to
inequality \eqref{P<P}, we obtain
\begin{equation*}
\begin{split}
Q_\rho(rR_A,U)&= \left(
\text{\rm id}\otimes P_{( R_1\otimes U^*,\ldots,  R_n\otimes U^*)}\right)[P_\rho(rA,S)]\\
&\leq c^2\left( \text{\rm id}\otimes P_{( R_1\otimes U^*,\ldots,
R_n\otimes U^*)}\right)[P_\rho(rB,S)] =c^2Q_\rho(rR_B,U)
\end{split}
\end{equation*}
for any $r\in [0,1)$. Using Proposition \ref{n=1}, we obtain that
$R_A\overset{H}{{\underset{c}\prec}}\, R_B$.  Now, assume that (ii)
holds.  Proposition \ref{n=1} implies
\begin{equation}
\label{K<K}
 K_\rho(re^{it},R_A)\leq c^2K_\rho(re^{it},R_B),\qquad r\in [0,1)
 \text{ and }
t\in \RR.
\end{equation}
Taking $t=0$, we obtain  $P_\rho( rA, R)\leq c^2 P_\rho( rB,R)$ for
any $r\in [0,1)$, which, due to Theorem \ref{equivalent}, implies
$A\overset{H}{{\underset{c}\prec}}\, B$. The equivalence
$(ii)\leftrightarrow (iii)$ is  a consequence of  Proposition
\ref{n=1} and the fact that inequality \eqref{K<K} is equivalent to
\begin{equation*}
 K_\rho(re^{it},R_A^*)\leq c^2K_\rho(re^{it},R_B^*),\qquad r\in
 [0,1) \text{ and }
t\in \RR.
\end{equation*}
This completes the proof.
\end{proof}

We remark that, according  to Theorem \ref{LCC} and Corollary
\ref{LBA-inf}, we have
\begin{equation*}
 \|L_{B,A}\|=\|C_{\rho,A}C_{\rho,B}^{-1}\|
 =\inf\{c> 1:\  P_\rho(A, R)\leq c^2 P_\rho(B, R) \}
\end{equation*}
for any $A,B\in [\cC_\rho]_{<1}$, where $C_{\rho,A}$ is defined in
Theorem \ref{LCC}.
\begin{corollary}
\label{connection} If $A,B$ are $n$-tuples of operators in $
[\cC_\rho]_{<1}$, then $ \|L_{B,A}\|=\| L_{R_B, R_A}\|=\| L_{R_B^*,
R_A^*}\|$. Moreover,
$
\delta_\rho(A,B)=\delta_\rho(R_A, R_B).
$
\end{corollary}

       %

      \end{document}